\documentclass[11pt]{amsart}
\usepackage{amsmath}
\usepackage{amssymb, verbatim}
\usepackage{pstricks,pst-node,pst-plot}
\usepackage{comment}
\usepackage{color}
\usepackage{extarrows}
\usepackage{tikz}
\usepackage[colorlinks=true,
            linkcolor=black,
            citecolor=black,
            urlcolor=black]{hyperref}
\usepackage{enumitem}
\usepackage{tikz-cd}
\usepackage[numbers]{natbib}
\title[]{Special Lagrangian Submanifolds and Circle Collapse on K3}

\author[S. Picard]{S\'ebastien Picard}
\address{Department of Mathematics, University of British Columbia, Vancouver, Canada}
  \email{spicard@math.ubc.ca}

\author[F. Trinca]{Federico Trinca}
  \email{f.trincamath@gmail.com}

\theoremstyle{plain}
\newtheorem{thm}{Theorem}[section]
\newtheorem{prop}[thm]{Proposition}
\newtheorem{defn}[thm]{Definition}
\newtheorem{lem}[thm]{Lemma}

\theoremstyle{definition}
\newtheorem{ex}[thm]{Example}
\newtheorem{rk}[thm]{Remark}

\numberwithin{equation}{section}

\newcommand{\p}{\partial}
\newcommand{\be}{\begin{equation}}
\newcommand{\bea}{\begin{eqnarray}}
\newcommand{\eea}{\end{eqnarray}} 
\newcommand{\ee}{\end{equation}}

\renewcommand{\leq}{\leqslant}
\renewcommand{\geq}{\geqslant}
\renewcommand{\le}{\leqslant}

\renewcommand{\epsilon}{\varepsilon}

\setlist[enumerate]{
  itemsep=0.3em, leftmargin=2.5em
  }
\setlist[itemize]{
  itemsep=0.3em, leftmargin=2.5em
  }

\begin{document}

\maketitle

\begin{abstract}
We consider $K3$ surfaces collapsing to a three-dimensional affine base. We show that certain affine lines on the base lift to degenerating sequences of special Lagrangian two-spheres and tori in the collapsing $K3$ surface. In particular, we construct special Lagrangian two-spheres connecting pairs of Taub-NUT bubbles. These examples fit into the broader program of reconstructing special submanifolds from graphs and combinatorial data on a collapsed affine limit.
\end{abstract}


\section{Introduction}

\subsection{Background}
Calabi--Yau manifolds often contain special submanifolds. We are interested in submanifolds $S \subseteq X$ which minimize volume in their homology class, so that
\[
{\rm Vol}(S) \leq {\rm Vol}(S'), \quad S' \in [S].
\]
There are two distinguished classes of such submanifolds on a Calabi--Yau manifold $(X,\omega,\Omega)$: holomorphic submanifolds and special Lagrangian submanifolds. As Calabi--Yau structures often come in families $X_\epsilon$, one way to study their space of special submanifolds is to send the parameters $\epsilon$ to their extreme limits to obtain a simplified structure $B$. A gluing construction then recovers special submanifolds on $X_\epsilon$ near this degenerate limit $B$.

It was noticed during the development of the SYZ program \cite{SYZ} that in many setups \cite{ChiuLin, ChiuLiLin, DonaldsonScaduto, GrossSiebert, KontSoibel, Mikhalkin, ParkerThesis} there appears to be a correspondence between:
\begin{itemize}
\item A class of weighted graphs on an affine space $B$ with singularities, and
  \item Minimizing submanifolds along a sequence of compact Ricci-flat manifolds $X_\epsilon$ collapsing to $B$.
  \end{itemize}

From the perspective of differential geometry, a central problem in this program is the following: given a collapsing degeneration $X_\epsilon \rightsquigarrow B$ and a suitable graph $\Gamma \subseteq B$, construct a sequence of minimizing submanifolds converging to $\Gamma$. Even the case of a graph with a single edge can lead to substantive geometric analysis, as seen for example in recent work of Chiu-Lin \cite{ChiuLin} where special Lagrangian submanifolds on compact Calabi--Yau threefolds are reconstructed from a geodesic path between two points on a lower dimensional base.

  The purpose of this paper is to explore these ideas on certain Calabi--Yau twofolds. In this setting, a hyperk\"ahler rotation takes special Lagrangian submanifolds to holomorphic curves and vice versa. We make the choice to take the special Lagrangian perspective, though our results could be equivalently stated in terms of holomorphic curves. But first, we give a trivial example of affine lines lifting to special submanifolds in collapsing Calabi--Yau geometries.
  \begin{ex} \label{ex-trivial}
    Take $\mathbb{T}^4 = \mathbb{R}^4 / \Lambda$ with the degenerating metrics
    \[
g_\epsilon =  dx^2 + dy^2 +dz^2 + \epsilon^2 dw^2.
    \]
    Sending $\epsilon \rightarrow 0$ collapses circle fibers to an affine base.
    \[
      (\mathbb{T}^4,g_\epsilon) \rightarrow (B^3,g_E), \quad B^3=\mathbb{T}^3_{x,y,z}.
    \]
  Let $L_0 \subseteq B^3$ be an embedded straight line. Rotate coordinates on $B^3$ so that $L_0$ points in the $x$-direction. Then $L_0$ lifts to a special Lagrangian 2-torus $L_\epsilon \subseteq (\mathbb{T}^4,g_\epsilon, \omega_\epsilon, \Omega_\epsilon)$ given by
\[
L_\epsilon = \{ y=c_1, \ z = c_2 \},
\]
where the holomorphic volume form $\Omega_\epsilon$ and K\"ahler form $\omega_\epsilon$ are given by
\[
\omega_\epsilon = dx \wedge dy + \epsilon dw \wedge d z, \quad \Omega_\epsilon = (dx+idy) \wedge (\epsilon dw+i dz).
\]
The submanifolds $L_\epsilon \subseteq (\mathbb{T}^4, g_\epsilon)$ form a thin tubes that converges to a straight line.
    \end{ex}

 We look for the analog of this basic example with the ambient Calabi--Yau twofold $\mathbb{T}^4$ replaced by a $K3$ surface. For this, we can use Sun-Zhang's result \cite{SunZhang} to list possible collapsing limits of hyperk\"ahler structures on $K3$ to a lower dimensional base $B$; such a limit $B$ must be either:

\begin{itemize}
  \item Affine space $B^1$ of dimension 1: a unit interval.
  \item Affine space $B^2$ of dimension 2: punctured $\mathbb{S}^2$.
\item Affine space $B^3$ of dimension 3: $\mathbb{T}^3/\mathbb{Z}_2$.
  \end{itemize}

  \begin{rk}
There are examples of collapsing $K3$ sequences in all three settings: Foscolo \cite{FoscoloK3} for $B^3$, Gross-Wilson \cite{GrossWilson} for $B^2$, and Hein-Sun-Viaclovsky-Zhang \cite{Hein-Sun-Viaclovsky-Zhang} for $B^1$. There are also examples where the limit is non-collapsing, namely the Kummer construction, where the limiting affine space $B^4 = \mathbb{T}^4 / \mathbb{Z}_2$; see e.g. \cite{Donaldson-Kummer}. For more on the theory of collapsing $K3$ surfaces, see \cite{OO}.
\end{rk}

\begin{rk}
Beyond $K3$ surfaces, there is also a general theory of collapsing Ricci-flat metrics on Calabi-Yau manifolds; see e.g. \cite{Hein-Tosatti, STZ, Szekelyhidi, Tosatti} and the recent survey by Tosatti \cite{TosattiSurvey}.
\end{rk}

In this paper, we consider the case of $K3 \rightsquigarrow B^3$, as this is the analog of our guiding Example \ref{ex-trivial}. The study of the interaction between affine collapse and special submanifolds in the remaining cases of affine bases $B^1$ and $B^2$ remains open.

\subsection{Main result}
We now set up the statement of the main theorem. Let $B = \mathbb{T}^3/\mathbb{Z}_2$. First, we specify combinatorial data on $B$. Define a datum $\mathcal{D}$ specified by:
\begin{enumerate}
\item A weight $m_j \in \mathbb{Z}_{\geq 0}$ associated to each of the eight singular points $q_j \in B_{\rm sing}$.
  \item A collection of $n$ points $p_i \in B_{\rm reg}$ each with a weight $k_i \in \mathbb{Z}_{>0}$ such that
  \[
\sum_{j=1}^8 m_j +  \sum_{i=1}^n k_i = 16.
  \]
  \end{enumerate}

  From datum $\mathcal{D}$, Foscolo \cite{FoscoloK3} constructs families of $K3$ surfaces $(M_\epsilon,g_\epsilon)$ collapsing to $B$. Each weight describes a bubble shrinking to its corresponding point: the weight $m_j$ indicates that the bubble shrinking to $q_j$ belongs to the family of ALF spaces of type $D_{m_j}$, and the weight $k_i$ indicates that the bubble shrinking to $p_i$ belongs to the family of ALF spaces of type $A_{k_i-1}$.
  
  We will prove:

  \begin{thm} \label{mainthm}
    Let $B = \mathbb{T}^3/\mathbb{Z}_2$ be given the datum $\mathcal{D}$ specified as above. Let $L_0 \subseteq B_{\rm reg}$ be an affine line connecting a pair of marked points labelled $p_{1}, p_{2} \in B_{\rm reg}$. There exists a sequence $(M_\epsilon, g_\epsilon, \omega_\epsilon, \Omega_\epsilon)$ of $K3$ surfaces such that:
    \begin{enumerate}
    \item There is a special Lagrangian 2-sphere $L_\epsilon$ converging to $L_0$ in the Gromov-Hausdorff sense.
      \[
\begin{tikzcd}[column sep=large, row sep=large]
L_\epsilon  
  \arrow[r, "\epsilon \to 0"] 
  \arrow[d, hook] 
& 
L_0 
  \arrow[d, hook] \\
(M_\epsilon,g_\epsilon)
  \arrow[r, "\epsilon \to 0"] 
& 
(B, g_E)
\end{tikzcd}
\]
      \item There are $k_1-1$ and $k_2-1$ special Lagrangian 2-spheres shrinking to the vertex points $p_{1}$ and $p_{2}$, respectively.
      \end{enumerate}
  \end{thm}

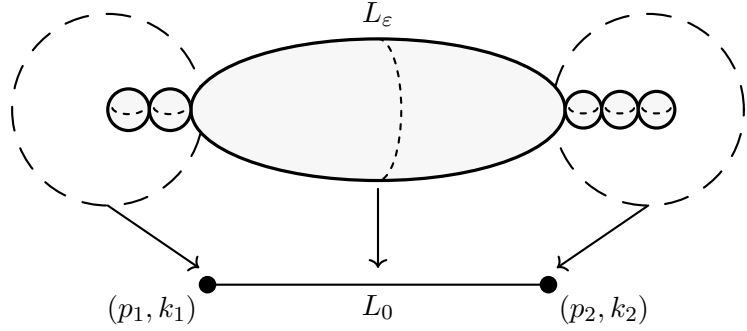
\begin{figure}

\begin{tikzpicture}[scale=1.1, line cap=round, line join=round]

\tikzset{
  redline/.style={black!75!black, very thick, fill=gray!6},
  dottedred/.style={black!75!black, thick, dashed, dash pattern=on 2pt off 3pt},
  bigcircle/.style={black, thick, dashed, dash pattern=on 9pt off 7pt},
  ball/.style={draw=black!75!black, very thick, fill=gray!6}
}

\def\a{2.25}   
\def\b{0.85}   
\def\rL{0.25}
\def\rR{0.22}

\draw[bigcircle] (-3.25,0) circle (1.15);
\draw[bigcircle] ( 3.25,0) circle (1.15);

\draw[redline] (0,0) ellipse [x radius=\a, y radius=\b];


\draw[dottedred] (0,\b) arc[start angle=90,end angle=-90,x radius=0.28,y radius=\b];

\draw[ball] (-3,0) circle (\rL);
\draw[ball] (-2.25-\rL,0) circle (\rL);

\draw[ball] (2.25+\rR,0) circle (\rR);
\draw[ball] (2.25+3*\rR,0) circle (\rR);
\draw[ball] (2.25+5*\rR,0) circle (\rR);

\foreach \x/\r in {-3/\rL, -2.50/\rL, 2.47/\rR, 2.91/\rR, 3.35/\rR}{
  \draw[dottedred] (\x-\r*0.78,0.02) arc[start angle=185,end angle=355,x radius=\r*0.78,y radius=\r*0.32];

}

\node at (0,1.15) {$L_\varepsilon$};

\draw[->, thick] (0,-0.95) -- (0,-1.90);
\node[below] at (0,-2.10) {$L_0$};

\coordinate (P1) at (-2.05,-2.10);
\coordinate (P2) at ( 2.05,-2.10);

\draw[black, thick] (P1) -- (P2);
\fill (P1) circle (3pt);
\fill (P2) circle (3pt);

\node[below left] at (P1) {$(p_1,k_1)$};
\node[below right] at (P2) {$(p_2,k_2)$};

\draw[->, thick] (-3.25,-1.15) -- (-2.15,-1.90);
\draw[->, thick] ( 3.25,-1.15) -- ( 2.15,-1.90);

\end{tikzpicture}

  \caption{A weighted graph with one edge and two vertices lifts to a collection of special Lagrangian 2-spheres on $K3$.} \label{fig:weightedgraph}
  \end{figure}

  \begin{rk} \label{rmk-collinear}
    A technical input in the construction of the ambient Calabi--Yau structure $(M_\epsilon,g_\epsilon,\omega_\epsilon,\Omega_\epsilon)$ in Theorem \ref{mainthm} is that the $A_{k_1-1}$, $A_{k_2-1}$ bubbles shrinking to $p_1$, $p_2$ are toric, i.e. the points $\{ y_a \}$ specified in the Gibbons--Hawking ansatz are collinear. One should think of the segments connecting these collinear points $\{ y_a \}$ as the 2-spheres shrinking to each vertex point $p_1$ and $p_2$, while the long segment connecting the pair of points $p_1$, $p_2$ corresponds to a half-line leaving from the outer left/right-most point $y_a$ (see Figure \ref{fig:weightedgraph}). Without the collinear condition, the ambient Calabi--Yau structure would need to be realigned by hyperk\"ahler rotation at each special submanifold.
\end{rk}

We restate Theorem \ref{mainthm} with the $k-1$ shrinking spheres neglected. Let $(M_\epsilon,g_\epsilon) \rightarrow (B,g_E)$ be a sequence of collapsing $K3$ surfaces to a 3-dimensional base \cite{FoscoloK3}. Let $O_1 \subseteq M_\epsilon$ be a multi-Taub-NUT bubble with $k_1$ distinct monopole points with zero center of mass shrinking to $p_1 \in B_{\rm reg}$, with similar notation for $(O_2,k_2,p_2)$. Let
\[
(O_1, k_1, p_1) \longleftarrow L_\epsilon \longrightarrow (O_2,k_2,p_2)
\]
be a 2-sphere joining one monopole point in one bubble to a monopole point in the other. We give more details on this 2-sphere in \S \ref{sec:bubble2bubble}, and we require that $L_\epsilon \subseteq M_\epsilon$ is the circle-bundle preimage of a segment $L_0 \subseteq B$ along the line through the bubble centers $p_i$. Our result is that for $\epsilon>0$ small, then $L_\epsilon$ can be deformed to a nearby special 2-sphere, i.e. either a holomorphic curve or special Lagrangian submanifold depending on choice of ambient Calabi-Yau structure.


  \begin{rk}
    Here is a simplified version of the theorem with all $m_j=0$. The purpose of allowing $m_j>0$ is so that Foscolo's construction spans an open set of the $K3$ moduli space, but since our construction of submanifolds takes place far from the bubbles associated to the weights $m_j$, we may as well set these parameters to zero. Consider a weighted graph $\Gamma \subseteq B_{\rm reg}$ with two vertices and an edge:
    \[
\Gamma = (V,E,w), \qquad 
V = \{p_1,p_2\}, \quad 
E = \{L_0\}, \quad 
w(p_i)=k_i.
\]
Assume the balancing condition $k_1+k_2=16$. The edge $L_0$ lifts to an elongated special Lagrangian 2-sphere collapsing to a line and the weighted vertices $(p_i,k_i)$ lift to a collection of $k_i-1$ special Lagrangian spheres shrinking to the vertex point. As the Calabi--Yau degenerates, these special submanifolds either stretch out and converge to a piecewise linear graph or shrink to a point. 
\end{rk}

There is another type of graph on $B$ which is readily lifted.

 \begin{thm} \label{mainthm2}
    Let $B = \mathbb{T}^3/\mathbb{Z}_2$ be given the datum $\mathcal{D}$ specified as above. Let $L_0 \subseteq B_{\rm reg}$ be an embedded affine line which avoids the marked points $p_i \in B_{\rm reg}$. There exists a sequence $(M_\epsilon, g_\epsilon, \omega_\epsilon, \Omega_\epsilon)$ of $K3$ surfaces such that:
    \begin{itemize}
    \item There is a special Lagrangian torus $\Sigma_\epsilon$ with $(\Sigma_\epsilon,g_\epsilon) \rightarrow (L_0,g_E)$ in the Gromov-Hausdorff sense.
      \end{itemize}
    \end{thm}

    This case of an edge with no vertices is much easier, and is in a sense the direct analog of Example \ref{ex-trivial} on $K3$. Indeed, if one were to run through the construction from the smooth affine base $B^3=\mathbb{T}^3$ without ever taking the $\mathbb{Z}_2$ quotient, the reconstructed Calabi--Yau manifold would be $\mathbb{T}^4$ instead of $K3$ and all the weights would be zero.

To summarize, we study two types of weighted graphs on the base $B_{\rm reg}$. An edge with no vertices corresponds to a torus. An edge with two weighted vertices corresponds to an elongated sphere with the vertices corresponding to a collection of shrinking spheres. Unlike other setups e.g. \cite{EsfaLi, GrossSYZ2012, Mikhalkin}, trivalent graphs cannot appear in this setting. The limiting objects of special submanifolds under $K3 \rightsquigarrow B^3$ collapse cannot allow two edges meeting at a nontrivial angle since the direction of a single edge determines the direction of the ambient hyperk\"ahler rotation of Calabi-Yau structure, and this observation also explains the collinearity condition in Remark \ref{rmk-collinear}.
    
\subsection{Related works}
We now review related results in the literature. Our result can be viewed as a $K3$ analog to the theorem of Chiu--Lin \cite{ChiuLin} which lifts paths on a base $\mathbb{P}^1$ to degenerating special Lagrangian 3-spheres on a compact Calabi--Yau threefold. More broadly, there are several works suggesting a correspondence between affine graphs and special submanifolds in collapsing Calabi--Yau geometries. For example, Mikhalkin \cite{Mikhalkin} established a correspondence between algebraic curves in $(\mathbb{C}^*)^2$ and weighted graphs in $\mathbb{R}^2$, and Chiu--Li--Lin \cite{ChiuLiLin} relate tropical curves to special Lagrangians in the cotangent bundle of the torus.

As for the ambient collapsing geometry, the construction of $K3$ surfaces collapsing to a 3-dimensional affine base was introduced by Foscolo \cite{FoscoloK3}, and our work relies extensively on the estimates developed there. In this setting, Oliveira \cite{Oliveira} studied geodesics, Hattori \cite{Hattori} studied the Dirichlet energy of maps, and non-holomorphic minimal spheres were constructed in \cite{FoscoloK3}, \cite{XuwenZhu} and \cite{FoscoloTrinca}. Related collapsing phenomena for gravitational instantons were studied by Salm \cite{Salm}.

\subsection{Organization}
  The paper is organized as follows:

  \textbf{Section 2.} Here we set notation and recall the definitions of hyperk\"ahler four-manifolds and their submanifolds.

  \textbf{Section 3.} We review the multi-Taub-NUT local model. These will appear as bubbles on the compact $K3$. A building block in our construction of elongated 2-spheres is a cigar submanifold in multi-Taub-NUT. A cigar in this context is a half-line coming out of a point in multi-Taub-NUT space where the $S^1$ fiber has collapsed. We will use these cigars to cap off the ends of a cylinder. 
  
  \textbf{Section 4.} We give an account of Foscolo's construction of $K3$ surfaces \cite{FoscoloK3}. Foscolo's construction can be understood in the context of the reconstruction problem \cite{GrossSYZ2012} in the SYZ program. We start from a singular affine base $B = \mathbb{T}^3 / \mathbb{Z}_2$ and look to construct a complex manifold which is a circle fibration over $B$ on a large open set. There are 8 point singularities $q_j \in B$. We also select $n$ points $p_i \in B_{\rm reg}$, which indicate the location of possible singular fibers. Let $\Lambda = \{q_j, p_i \}$. The selection of weights satisfying the balancing condition
  \[
\sum_{j=1}^8 m_j + \sum_{i=1}^n k_i = 16, \quad m_j \geq 0, \quad k_i>0
  \]
  determines a (non-principal) $S^1$-bundle
  \[
X(B_0) \rightarrow B_0
  \]
  over the base $B_0 = B \backslash \Lambda$. Foscolo's theorem is about compactifying $X(B_0)$ to a compact complex manifold $X(B)$. In the SYZ setup of torus bundles over an affine base, a similar compactification problem is studied by methods of algebraic geometry by Kontsevich-Soibelman \cite{KontSoibel} and Gross-Siebert \cite{GrossSiebert}.

  The compactification $X(B)$ comes in a family $X_\epsilon := X(B)_\epsilon$ with hyperk\"ahler metrics $g_\epsilon$. Over an open set $U_\epsilon \subseteq X_\epsilon$ tending to full volume as $\epsilon \rightarrow 0$, then $(X_\epsilon,g_\epsilon)$ is an $S^1$-fibration over $B_0$ with circle fibers collapsing as $\epsilon \rightarrow 0$.

  To construct this compactification $X_\epsilon$, Foscolo glues-in some local models to the boundary components of $X(B_0)$: an ALF space of type $D_{m_j}$ is attached near $q_j$ and a multi-Taub-NUT space with $k_i$ monopoles is attached near $p_i$. These local models are bubbles which shrink to their respective points $q_j$ or $p_i$ in the limit as $\epsilon \rightarrow 0$.

  \textbf{Section 5.} With the ambient geometry reviewed, we move on to constructing special Lagrangian submanifolds and start the proof of Theorem \ref{mainthm}. The construction of the $k_i-1$ shrinking spheres in Theorem \ref{mainthm} is essentially already known \cite{LotayOli, FoscoloK3}. The main technical content of this work is constructing an elongated special Lagrangian sphere spanning a pair of bubble regions. For this, the first step is to construct an approximate solution $L_\epsilon$. It is a thin cylinder connecting two bubble regions with its two ends capped off. The second step is to select an ambient Calabi--Yau structure $(\omega_\epsilon,\Omega_\epsilon)$ which solves the necessary condition
  \[
\int_{L_\epsilon} \omega_\epsilon = 0, \quad \int_{L_\epsilon} {\rm Im} \, \Omega_\epsilon = 0.
  \]
This ambient Calabi--Yau structure is chosen such that both the long sphere and the shrinking spheres in Theorem \ref{mainthm} satisfy the necessary topological requirement to be deformed to special Lagrangian. The final step is to deform $L_\epsilon$ from an approximate special Lagrangian to a genuine special Lagrangian submanifold. The deformation is generated by a normal vector field of the form $J a^\sharp$ where $a \in \Omega^1(L_\epsilon)$. The special Lagrangian equation is schematically of the form
  \[
\mathcal{F}(a) = 0,
  \]
  and to solve this, we Taylor expand
  \[
\mathcal{F}(a) = \mathcal{F}(0) + \mathcal{L}(a) + \mathcal{Q}(a)
  \]
  and seek a fixed point $\mathcal{N}(a)=a$ of 
  \[
\mathcal{N}(a) = \mathcal{L}^{-1}(-\mathcal{F}(0) - \mathcal{Q}(a)).
\]
The fixed point theorem requires $\| \mathcal{N}(u)- \mathcal{N}(v)\| < \|u-v\|$, and so we will need bounds on $\| \mathcal{L}^{-1} \|$, smallness of the approximate solution $\| \mathcal{F}(0) \|$, and a quadratic estimate $\| \mathcal{Q}(u) - \mathcal{Q}(v) \|$.

\textbf{Section 6.} This section contains the proof of Theorem \ref{mainthm2} on special Lagrangian tori $\Sigma_\epsilon$ lifting a line $L_0 \subseteq B_{\rm reg}$ avoiding the marked points $p_i \in B_{\rm reg}$. The proof follows readily from Foscolo's error estimates and a perturbation argument since the submanifold does not enter the bubbles.

\textbf{Section 7.} We return to the elongated sphere of Theorem \ref{mainthm}. As typical for gluing constructions, the main step of the proof is the estimate of the inverse of the linearized operator $\| \mathcal{L}^{-1} \|$. The linearized operator is $\mathcal{L}=d+d^\dagger$ on the 2-sphere $L_\epsilon$ and we are able to estimate
\[
\| a \|_{C^{1,\alpha}_\beta} \leq \epsilon^{-\frac{1}{2} - \iota} \| (d+d^\dagger) a \|_{C^{0,\alpha}_{\beta-1}}, \quad a \in \Omega^1(L_\epsilon)
\]
where $\beta<0$ and $\iota>0$ are small parameters, and $\rho^\beta$ is a weight in the H\"older norms for a suitable weight function $\rho$. In fact, we are able to obtain a uniform estimate with $C>1$ independent of $\epsilon$
\[
\| a^\perp \|_{C^{1,\alpha}_\beta} \leq C \| (d+d^\dagger) a^\perp \|_{C^{0,\alpha}_{\beta-1}}, \quad a^\perp \in \mathcal{H}^\perp
\]
when restricting to 1-forms in a certain subspace $\mathcal{H}^\perp$. The subspace of bad 1-forms $\mathcal{H}$ is spanned by $\kappa_1=\zeta dx$ and $\kappa_2=\zeta d \psi$, where $\zeta(x)$ is a fixed cutoff function localizing us to the center of the cylindrical region of $L_\epsilon$, and $dx, d \psi$ are the 1-forms corresponding to translation and rotation on the cylinder. We can then $L^2$-decompose an arbitrary 1-form by $a = a^\perp + \lambda_1 \kappa_1 + \lambda_2 \kappa_2$ and estimate the $\kappa_i$ manually. Due to the ratio
\[
\frac{\| \kappa_i \|_{C^1}}{\| \kappa_i \|_{L^2}} = O(\epsilon^{-1/2})
\]
on collapsing cylinders with metric $dx^2 + \epsilon^2 d \psi^2$, we pick up a power of $- \frac{1}{2}$ in the bound for the inverse of the linearized operator $d+d^\dagger$. Nevertheless, the approximate solution is so close to a genuine solution that the poor bounds for $\| \mathcal{L}^{-1} \|$ can be absorbed and the fixed point argument can be completed.

\section*{Acknowledgments}
We thank Shih-Kai Chiu, Lorenzo Foscolo, Yu-Shen Lin, and Jason Lotay for helpful comments and discussions.

\section{Hyperk\"ahler four-manifolds}
\subsection{Definition}
In this section, we recall the fundamental properties of hyperk\"ahler four-manifolds. Our approach follows Donaldson \cite{Donaldson-2forms}, who defines a hyperk\"ahler four-manifold in the following way:

\begin{defn}
  Let $X$ be an oriented four-manifold. A triple
  \[
\underline{\boldsymbol{\omega}} = (\omega_1,\omega_2,\omega_3)
\]
of two-forms $\omega_i \in \Omega^2(X)$ is said to be a hyperk\"ahler triple if
\begin{equation} \label{hk-pointwise}
\omega_1^2=\omega_2^2=\omega_3^2>0, \quad \omega_i \wedge \omega_j = 0 \ {\rm for} \ i \neq j,
\end{equation}
and the closedness conditions $d \omega_i = 0$ hold.
\end{defn}

We now review how to construct from the triple $\underline{\boldsymbol{\omega}}$ the following structures: a Ricci-flat metric, a sphere of complex structures, and holomorphic volume forms. First, from a hyperk\"ahler triple $\underline{\boldsymbol{\omega}}$, we obtain a metric tensor $g_{\underline{\boldsymbol{\omega}}}$ by the formula
\begin{equation} \label{hk-metric}
g_{\underline{\boldsymbol{\omega}}}(V,W) = {1 \over 6} \frac{\epsilon^{ijk} \iota_V \omega_i \wedge \iota_W \omega_j \wedge \omega_k}{\omega_1^2/2}
\end{equation}
 and the relation \eqref{hk-pointwise} can be rewritten as
\begin{equation} \label{hk-pointwise2}
{1 \over 2} \omega_a \wedge \omega_b = \delta_{ab} \, d {\rm vol}_{g_{\underline{\boldsymbol{\omega}}}}.
\end{equation}
Formula \eqref{hk-metric} can be found in e.g. \cite{FineYao}, and from the perspective of $G_2$ geometry on $X \times T^3$ it is a natural tensor to form. We refer to the metric $g_{\underline{\boldsymbol{\omega}}}$ as the hyperk\"ahler metric associated to $\underline{\boldsymbol{\omega}}$.

\begin{ex}
  Let $\underline{\boldsymbol{\omega}}$ be a hyperk\"ahler triple. At any point $p \in X$, we may choose local coordinates $(x_1,x_2,x_3,\psi)$ such that $\underline{\boldsymbol{\omega}}$ assumes the normal form
  \begin{align}
    \omega_1|_p &= dx_1 \wedge d \psi + dx_2 \wedge dx_3, \label{normalform}\\
    \omega_2|_p &= dx_2 \wedge d \psi + dx_3 \wedge dx_1, \nonumber\\
    \omega_3|_p &= dx_3 \wedge d \psi + dx_1 \wedge d x_2. \nonumber
  \end{align}
  In these coordinates, direct calculation of \eqref{hk-metric} gives
  \[
   g_{\underline{\boldsymbol{\omega}}}|_p = \delta_{ij}, \quad d {\rm vol}_{ g_{\underline{\boldsymbol{\omega}}}} |_p = dx_1 \wedge dx_2 \wedge dx_3 \wedge d \psi.
    \]
\end{ex}

From now on, we drop the subscript on the metric tensor and simply write $g$ instead of $g_{\underline{\boldsymbol{\omega}}}$.

\subsection{Complex structure}
A hyperk\"ahler structure $(X,\underline{\boldsymbol{\omega}})$ on a four-manifold defines a triple of complex structures $(J_1,J_2,J_3)$. These are defined by the conditions
\begin{equation} \label{hk-cplxstr}
\omega_i(\cdot, \cdot) = g(J_i \cdot, \cdot).
\end{equation}
A direct calculation using \eqref{hk-metric} and \eqref{hk-cplxstr} verifies that the complex structures $J_i$ satisfy the quaternionic relations
\begin{equation} \label{hk-quater}
J_i^2=-{\rm id}, \quad J_1 J_2 = J_3.
\end{equation}
We will see below that the Nijenhuis tensor for each $J_i$ is zero so that these complex structures are integrable. Therefore $\omega_i$ is the K\"ahler form of the K\"ahler manifold $(X, J_i, g)$. 

\subsubsection{The sphere}
For any vector $v \in \mathbb{S}^2 \subseteq \mathbb{R}^3$, we also obtain a complex structure denoted $J_v$ given by
\[
J_v = v_1 J_1 + v_2 J_2 + v_3 J_3,
\]
with associated K\"ahler form $\omega_v = \sum v_i \omega_i$.

\subsubsection{Holomorphic volume forms}
Fix a vector $v \in \mathbb{S}^2$ and assemble the complex manifold $(X,J_v)$. Consider
\begin{equation} \label{hyperK-holoform}
\Omega_v = \omega_{v_1^\perp} + i \omega_{v_2^\perp}, \quad \Omega_v \in \Omega^{2,0}(X,J_v)
\end{equation}
where $v_1^\perp,v_2^\perp \in \mathbb{S}^2$ are a pair of orthogonal vectors which are orthogonal to the chosen $v$. Note that there is an $S^1$-family of such choices, and hence an $S^1$-family of possible $\Omega_v$. A pair $(\omega_v,\Omega_v)$ will be referred to as a Calabi-Yau structure, and such a pair satisfies the Calabi-Yau equations
\begin{align} \label{su2}
  & \ d \omega_v = 0, \quad d \Omega_v = 0,\\
  & \ {1 \over 2} \Omega_v \wedge \bar{\Omega}_v = \omega_v^2, \quad \omega_v \wedge \Omega_v=0.
\end{align}
In other words, $\Omega_v$ is a holomorphic volume form on $(X,J_v)$ with constant norm. A standard calculation in K\"ahler geometry gives
\[
R_{i \bar{j}} = - \partial_i \partial_{\bar{j}} \log \omega_v^2 =0
\]
and so the Riemannian metric $g$ associated to $\underline{\boldsymbol{\omega}}$ is Ricci-flat.

\subsubsection{Example}
To be explicit, we can consider the complex structure $J_1$. Then
\[
\Omega_1 = \omega_2 + i \omega_3,
\]
is the holomorphic volume form associated to the complex structure $(X,J_1)$, and we can write $\underline{\boldsymbol{\omega}} = (\omega_1, {\rm Re} \, \Omega_1, {\rm Im} \, \Omega_1)$. We see that $\Omega_1$ is of type $(2,0)$. Indeed, if $V \in T^{0,1}(X,J_1)$ and $Y \in T_{\mathbb{C}}(X)$, then that
\[
\Omega_1(V,Y) = 0,
\]
follows from $J_2J_1=-J_3$ and $J_1|_{T^{0,1}(J_1)} = -i$. Let us now check that the Nijenhuis tensor of $J_1$ vanishes. This is equivalent to the following property: for any $\eta \in \Omega^{1,0}(X,J_1)$, then $(d \eta)^{0,2}=0$. This follows from $\Omega_1 \in \Omega^{2,0}$ being closed and nowhere vanishing, since
\[
0 = d( \eta \wedge \Omega_1) = d \eta \wedge \Omega_1
\]
and so $(d \eta)^{0,2}=0$.

\subsection{Special submanifolds}
We will study two kinds of special submanifolds contained in a hyperk\"ahler four-manifold. Fix a direction $v \in \mathbb{S}^2$ and a Calabi-Yau structure $(X,J_v,\omega_v,\Omega_v)$, and for the moment drop the subscript $v$.

\subsubsection{Special Lagrangian submanifolds} A submanifold $L \subset (X,g,\omega,\Omega)$ of real dimension 2 is special Lagrangian if there exists an angle $e^{-i \theta} \in \mathbb{C}^*$ such that
\[
 {\rm Re} \, e^{-i \theta} \Omega|_L = d {\rm vol}_g,
\]
so that $L$ is calibrated by ${\rm Re} \, e^{-i \theta} \Omega$. This condition is equivalent (see \cite{HarveyLawson}) to
\begin{equation} \label{slag-eq}
\omega|_L= 0, \quad {\rm Im} \, e^{-i \theta} \Omega|_L = 0.
\end{equation}
Harvey-Lawson proposed these special submanifolds as distinguished objects on a complex manifold (not necessarily K\"ahler) with trivial canonical bundle (Section V.3. in \cite{HarveyLawson}). In the particular case when the ambient complex manifold admits a K\"ahler structure $\omega$, then Lagrangian submanifolds of $(X,\omega,\Omega)$ are special Lagrangian if and only if their mean curvature is zero. The special Lagrangian equations were later rediscovered in the string theory literature by Becker-Becker-Strominger \cite{BBS} as a constraint for supersymmetric cycles on Calabi-Yau threefolds.

One implication of the special Lagrangian equations \eqref{slag-eq} over a K\"ahler Calabi-Yau manifold $(X,\omega,\Omega)$ is that these submanifolds minimize the area functional of the associated metric $g$ in a given homology class $[L] \in H_2(X,\mathbb{R})$. The minimizing volume is given by the topological quantity
\[
{\rm vol}(L) = {\rm Re} \, e^{-i \theta} [\Omega] \cdot [L].
\]

\subsubsection{Holomorphic submanifolds}
Next, a submanifold $C \subseteq X$ of real dimension 2 is a holomorphic curve if
\[
\omega|_C = d {\rm vol}_g.
\]
In other words, $C$ is calibrated by $\omega$. The fact that $C$ is a complex analytic submanifold is an application of Wirtinger's inequality; see e.g. (1.25) in \cite{Demailly} for details. Holomorphic curves minimize volume in a given homology class $[C] \in H_2(X,\mathbb{R})$ with minimal volume
\[
{\rm vol}(C) = [\omega] \cdot [C].
\]

\subsubsection{Exchange} \label{subsect:exchange}
On hyperk\"ahler four-manifold, we can interchange holomorphic curves with special Lagrangian submanifolds by hyperk\"ahler rotation. Indeed, suppose $L \subseteq (X,\omega_v,J_v)$ with $v \in \mathbb{S}^2$ is a holomorphic curve on a hyperk\"ahler four-manifold. Then
\[
\omega_{v^\perp}|_L = 0
\]
for all $v^\perp \in \mathbb{S}^2$ perpendicular to $v$; this gives an $S^1$-family of such $v^\perp$. Fix such a $v^\perp$, and next select a pair of vectors $w_1, w_2$ such that $(v^\perp, w_1, w_2)$ is an oriented orthonormal basis; this also comes in an $S^1$-family of choices. Using $(w_1,w_2)$, we define the holomorphic volume form $\Omega_{v^\perp}=\omega_{w_1}+i \omega_{w_2}$ as in \eqref{hyperK-holoform}. Choose an angle $e^{-i \theta}$ by writing $v= \cos \theta w_1 + \sin \theta w_2$. Then we have
\[
{\rm Im} \, (e^{-i \theta} \Omega_{v^\perp}) = 0,
\]
and $L \subseteq (X, \omega_{v^\perp}, \Omega_{v^\perp})$ is a special Lagrangian submanifold. See Figure \ref{fig:hol2slag}.

On the other hand, if $(v^\perp, w_1, w_2)$ is an oriented orthonormal basis and $L \subseteq (X, \omega_{v^\perp}, \omega_{w_1}+i \omega_{w_2})$ is special Lagrangian with angle $e^{-i \theta}$, then
\[
  {\rm Re} \, e^{-i \theta} \Omega = \omega_v, \quad v = \cos \theta w_1 + \sin \theta w_2
  \]
and $L \subseteq (X,\omega_v, J_v)$ is a holomorphic curve.

The exchange described above can be computed explicitly by the normal form \eqref{normalform} of $\underline{\boldsymbol{\omega}}$, using for example $v=e_1$, $T_p L = {\rm span} \{ e_1, \partial_\psi \}$, $v^\perp = e_3$, $w_1 = \cos \theta e_1 - \sin \theta e_2$, and $w_2 = \sin \theta e_1 + \cos \theta e_2$.

\begin{figure}
\begin{tikzpicture}[scale=1]
  \def\R{1.8}      
  \def\tilt{25}    

  \draw[line width=0.8pt] (0,0) circle (\R);

  \foreach \p in {0,90}{
    \begin{scope}[rotate=\p]
      \draw[gray!80, densely dashed]
        (0,0) ellipse [x radius=\R, y radius={\R*sin(\tilt)}];
    \end{scope}
  }

  \begin{scope}[rotate=90]
\draw[->, thick] (0,0) -- (0.71,0.71) node[above] {$w_1$};
\draw[->, thick] (0,0) -- (0.71,-0.71) node[above] {$w_2$};
\draw[thick, gray!60] (0,0) -- (1.8,0);
 \draw[thick, gray!60] (0.3,0) arc[start angle=0, end angle=45, radius=0.3];
  \node at (0.52,0.2) {$\theta$};
    \end{scope}
  
  \filldraw[black] (\R,0) circle (2pt) node[right] {$v^\perp$};
  \filldraw[black] (0,\R) circle (2pt) node[above] {$v$};

\end{tikzpicture}
\caption{Hyperk\"ahler rotation from holomorphic curves to special Lagrangians.} \label{fig:hol2slag}
\end{figure}
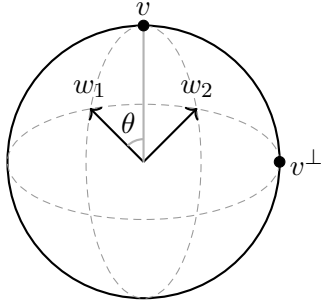


\section{Gibbons-Hawking ansatz geometry}
We set conventions for the Gibbons-Hawking ansatz \cite{GH78}. Our presentation broadly follows \cite{LotayOli} with opposite sign conventions to match \cite{FoscoloK3}.

\subsection{Conventions and definitions}
\subsubsection{Local coordinates}
Let $(x_1,x_2,x_3)$ denote coordinates on an open set $U \subset \mathbb{R}^3$. In practice, $U$ will often be taken as the complement of a set of points. Let $X$ be the total space of a $U(1)$-bundle with projection $\pi: X \rightarrow U$. Our conventions for the local angle coordinate is $e^{i \psi}$ which transforms as
\[
e^{i \psi} \mapsto e^{i \tau} e^{i \psi}, \quad \tau = \tau(x_1,x_2,x_3)
\]
where $\tau$ is the transition function of the change of coordinates. In other words, the bundle projection map in these local coordinates is
\[
  (x_1,x_2,x_3,e^{i \psi}) \overset{\pi}{\mapsto} (x_1,x_2,x_3),
\]
and
\[
\xi \overset{\rm loc}{=} {\partial \over \partial \psi}
\]
is a global vector field on $X$ generating the $U(1)$ action.

Let $\theta \in \Omega^1(X,\mathbb{R})$ be a connection on $X$ with $\theta(\xi)=1$ taking the local form
\[
\theta \overset{\rm loc}{=}  d \psi - i \pi^* \mathcal{A}
\]
where $\mathcal{A}$ is the local connection 1-form on the base $U_{(x_1,x_2,x_3)}$ valued in $i \mathbb{R}$ which transforms as
\[
\mathcal{A} \mapsto \mathcal{A} - i d \tau
\]
under change of coordinates. The curvature of the connection $F \in \Omega^2(U)$ is defined by $F = d \mathcal{A}$, or
\[
d \theta = - i F.
\]

\subsubsection{Metric ansatz}
Suppose now that
\begin{equation} \label{GH-harm}
d \theta = \star_{\mathbb{R}^3} d h,
\end{equation}
where $h: U \rightarrow (0,\infty)$ is a positive harmonic function. Note that if $h$ satisfies \eqref{GH-harm}, then it must be harmonic. The Gibbons-Hawking ansatz is given by the triple
\[
\underline{\boldsymbol{\omega}}^{\rm gh} = (\omega_1^{\rm gh},\omega_2^{\rm gh},\omega_3^{\rm gh})
\]
where we define
\begin{align}
  &\ \omega_1^{\rm gh} = dx_1 \wedge \theta + h dx_2 \wedge dx_3, \label{GH-ansatz}\\
                    &\ \omega_2^{\rm gh} = dx_2 \wedge \theta + h dx_3 \wedge dx_1, \nonumber\\
  &\ \omega_3^{\rm gh} = dx_3 \wedge \theta + h dx_1 \wedge dx_2. \nonumber
\end{align}
Direct calculation using \eqref{GH-harm} gives
\[
d \omega_a^{\rm gh} = 0.
\]
We can also directly compute
\[
{1 \over 2} \omega_a^{\rm gh} \wedge \omega_b^{\rm gh} = \delta_{ab} \, d {\rm vol}_{g^{\rm gh}},
\]
where $g$ is the metric
\[
g^{\rm gh} = h \pi^* g_{\mathbb{R}^3} + h^{-1} \theta \otimes \theta.
\]
The metric can also be derived from \eqref{hk-metric}, and so $(X,\underline{\boldsymbol{\omega}})$ forms a hyperk\"ahler structure with associated metric $g^{\rm gh}$.

\subsubsection{Special submanifolds}
Consider a straight line segment $\gamma$ in $U \subseteq \mathbb{R}^3$. Let $L = \pi^{-1}(\gamma)$ be the pre-image of this line in the total space of the $U(1)$-bundle $X \overset{\pi}{\rightarrow} U$. It was noticed in \cite{LotayOli} that $L \subseteq X$ can be viewed as either a holomorphic submanifold or a special Lagrangian submanifold depending on the choice of Calabi-Yau structure. Indeed, let $v$ be a unit vector specifying the direction of the line $\gamma$. Then
\[
\omega_v^{\rm gh}|_L = d {\rm vol}|_L
\]
so that $L$ is a holomorphic curve on $(X,J_v)$. To be concrete, say $v = e_1$. Then for any $p \in L$, we have
\[
T_p L = {\rm span} \, \bigg\{ {\partial \over \partial x_1}, {\partial \over \partial \psi} \bigg\}.
\]
It follows that
\[
\omega_1^{\rm gh}|_L = dx_1 \wedge d \psi
\]
and on the other hand
\begin{equation} \label{cigar-metric}
  g^{\rm gh}|_L = h dx_1^2 + h^{-1} (d \psi - i \mathcal{A}_1 dx_1)^2,
\end{equation}
so that $\omega_1^{\rm gh}|_L = d {\rm vol}|_L$. Thus $L \subseteq (X,J_1)$ is a holomorphic curve. Alternatively, $L$ can also be understood as a special Lagrangian submanifold with respect to Calabi-Yau structures of the form $(X, \omega_{v^\perp}, \Omega_{v^\perp})$ .

\begin{rk}
Any circle-invariant minimal surface in $(X, g^{\rm gh})$ is of this form \cite{LotayOli}. However, there are examples of compact minimal surfaces in this setting that are not circle-invariant (and unstable) \cite{Trinca22}.
\end{rk}

\subsection{Multi-Taub-NUT spaces} \label{sec:multiTAUB}
In this paper's main construction, we will use the following particular instance of the Gibbons-Hawking ansatz. Select points $y_1,\dots y_k$ in $\mathbb{R}^3$ and let $U = B_R \backslash \{y_1,\dots, y_k \}$ with $B_R \subseteq \mathbb{R}^3$ a large ball centered at the origin. Consider harmonic functions of the form
\[
h(x) =m+\sum_{i=1}^k \frac{1}{2 |x-y_i|} + \ell(x),
\]
where $m \in \mathbb{R}_{>0}$ and $\ell(x)= \sum_i a_i x_i$ is a linear function such that $h>0$ over $U$.

\begin{rk}
The multi-Taub-NUT geometry is traditionally defined without the linear function, namely $\ell(x) \equiv 0$. We include it here as this local model will appear as bubbles on the $K3$ surface \cite{FoscoloK3} and including this linear term will help match the local model to the geometry on the global compact manifold.
\end{rk}

To create a Gibbons-Hawking space, we must verify that ${1 \over 2 \pi} [\star dh] \in H^2(U,\mathbb{Z})$, and this holds because integrals over small spheres around the punctures $y_i$ are equal to
\begin{equation} \label{local-charge}
\frac{1}{2 \pi} \int_{\partial B_\epsilon(y_i)} \star dh = -1. 
\end{equation}
The condition of integer coefficients in cohomology guarantees that ${1 \over 2 \pi} [\star dh]$ comes from the curvature of a bundle: there exists a $U(1)$-bundle $N \overset{\pi}{\rightarrow} U$ with connection form $\theta$ so that $d \theta = \star d h$. We can then form the Gibbons-Hawking ansatz \eqref{GH-ansatz} and obtain a hyperk\"ahler structure with metric
\[
g^{\rm tn} = h g_{\mathbb{R}^3} + h^{-1} \theta \otimes \theta.
  \]
Next, we note that the geometry of the total space $N$ of the $U(1)$ bundle extends over the points $y_i$. For each point $y_i$, glue along
\[
\pi^{-1}(B_\epsilon(y_i) \backslash \{ y_i \}) \cong \hat{B}_\epsilon \backslash \{ 0 \} \subseteq \mathbb{R}^4
\]
and let $\hat{N}$ be the smooth manifold obtained by adding the balls $\hat{B}_\epsilon \subseteq \mathbb{R}^4$ with origin included. Then
\[
\pi: \hat{N} \rightarrow B_R
\]
is a smooth map such that the circle fibers collapse upon approach to the points $\{y_i\}$. We now drop the hat notation and simply write $N$ for the extension to $y_i$. Let us look at the metric $g^{\rm tn}$ behavior near the marked points. Near each $y_i$, we convert to spherical coordinates $(r,\varphi_1,\varphi_2)$ on $\mathbb{R}^3$ so that $r$ is radial coordinate to $y_i$. Then the hyperk\"ahler metric appears as
\[
g = {1 \over 2r} (dr \otimes dr + r^2 g_{\mathbb{S}^2}) + 2r \theta \otimes \theta + O(1)
\]
and the change of variables $\rho =  \sqrt{2} r$ reveals the smooth metric
\[
g =  d \rho \otimes d \rho + \rho^2 \bigg( {1 \over 4} g_{\mathbb{S}^2} + \theta^{\rm mono} \otimes \theta^{\rm mono} \bigg) + O(1),
\]
where
\[
\theta^{\rm mono} = d \psi - i \mathcal{A}^{\rm mono}, \quad \mathcal{A}^{\rm mono} = \frac{i}{2}(1+\cos \varphi_1) d \varphi_2.
  \]
In fact, $d \rho^2 + \rho^2 (\frac{1}{4} g_{S^2} + (\theta^{\rm mono})^2)$ is the Euclidean metric on $\mathbb{R}^4$ in polar coordinates since $g_{S^3} = \frac{1}{4} g_{S^2} + (\theta^{\rm mono})^2$ by the Hopf fibration construction. So $(N,g)$ is a smooth Ricci-flat manifold which is the total space of a circle bundle on the complement of finitely many points where the circle fibers collapse.

We will denote the multi-Taub-NUT geometry by $(N, \underline{\boldsymbol{\omega}}^{\rm tn})$. Later on, when we glue-in this local model to appropriate scale, we will use the scaling map
\begin{equation} \label{scaling}
S_\epsilon (x,e^{i \psi}) = (\epsilon x, e^{i \psi})
\end{equation}
which sends $x \mapsto \epsilon x$ on the base $\mathbb{R}^3$ and hence satisfies $S_\epsilon^* \rho = \epsilon \rho$ where 
\begin{equation} \label{rho-TN}
\rho: N \rightarrow [0,\infty), \quad \rho(x,e^{i \psi}) = |x|
\end{equation}
 is the distance function to the origin in $\mathbb{R}^3$.

 \subsubsection{Circle invariant holomorphic submanifolds}
Next, we consider circle-invariant holomorphic submanifolds in the multi-Taub-NUT geometry $(N, \underline{\boldsymbol{\omega}}^{\rm tn})$; see also \cite{LotayOli} for further analysis. It was observed in \cite{Trinca22} that any compact holomorphic curve in the multi-Taub-NUT space projects into the union of lines connecting the singular points of the harmonic function $h$, which means that any compact holomorphic curve is circle-invariant. Let us describe these holomorphic curves.

Consider a straight line $\gamma$ in $\mathbb{R}^3$ joining two distinct points $y_i$, and let $L = \pi^{-1}(\gamma)$ be the pre-image of this line in the total space of the circle bundle. Since the circles collapse at the endpoints $y_i$, we have that $L \subseteq N$ is a submanifold which is topologically a 2-sphere $L \cong S^2$; see Figure \ref{fig:sphere}. Its self-intersection is
\[
L \cdot L = -2.
\]
Intuitively, this $-2$ is the sum of the local charges \eqref{local-charge} at each endpoint. Up to a hyperk\"ahler rotation, this $L$ can either be seen as a holomorphic $(-2)$-curve or a special Lagrangian 2-sphere. 

\begin{figure}
\begin{tikzpicture}[scale=0.55, line cap=round, line join=round]
  \def\L{6.2}
  \def\H{1.25}
  \def\a{0.16} 

  \pgfmathdeclarefunction{profileA}{1}{%
    \pgfmathparse{\H*sqrt(max(0,1-(#1/\L)^2))}%
  }

  \fill[gray!8]
    plot[domain=-\L:\L, samples=200] ({\x},{profileA(\x)})
    -- plot[domain=\L:-\L, samples=200] ({\x},{-profileA(\x)})
    -- cycle;
  
  \draw[thick] plot[domain=-\L:\L, samples=200] ({\x},{profileA(\x)});
  \draw[thick] plot[domain=-\L:\L, samples=200] ({\x},{-profileA(\x)});

  \draw[thick] (-\L,0) -- (\L,0);

  \foreach \x in {-5.75,-4.85,-3.7,0,3.7,4.85,5.75} {
    \pgfmathsetmacro{\r}{profileA(\x)}
    \pgfmathsetmacro{\rx}{\a*\r}

    \draw[thick,dashed,dash pattern=on 4pt off 4pt]
      (\x,\r) arc[start angle=90,end angle=270,
      x radius=\rx, y radius=\r];

    \draw[thick]
      (\x,\r) arc[start angle=90,end angle=-90,
      x radius=\rx, y radius=\r];
  }

  \fill (-\L,0) circle (3.8pt);
  \fill (\L,0) circle (3.8pt);

  \node[left] at (-\L-0.18,0) {\(y_1\)};
  \node[right] at (\L+0.18,0) {\(y_2\)};
\end{tikzpicture}

\caption{A sphere connecting two points where the circle fibers collapse in multi-Taub-NUT.} \label{fig:sphere}
\end{figure}
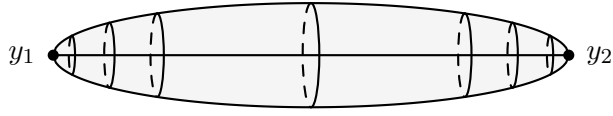

\subsubsection{Cigar} \label{sec:cigar}
In our construction we will encounter a cigar submanifold. Let $y_a \in \mathbb{R}^3$ be $k$ distinct points. We rotate and translate coordinates in $\mathbb{R}^3$ to pin down
\[
y_1 = (1,0,0).
\]
Let
\[
h^{\rm tn} = 1 + \sum_{a=1}^k \frac{1}{2 |x-y_a|}
\]
be a harmonic function defining a multi-Taub-NUT space $N \overset{\pi}{\rightarrow} U$ where $U = \mathbb{R}^3 \backslash \cup_{a=1}^k \{y_a \}$ and the metric is $g^{\rm tn} = h g_{\mathbb{R}^3} + h^{-1} \theta^{\rm tn} \otimes \theta^{\rm tn}$, with
\[
\theta^{\rm tn} = d \psi - i \pi^* \mathcal{A}^{\rm tn}.
\]
The connection 1-form $\mathcal{A}^{\rm tn}$ has an explicit local expression. Near a charge $y_a$, we set spherical coordinates $(r, \varphi_1, \varphi_2)$ in $\mathbb{R}^3$ such that the radius $r$ is based at $y_a$. Then for
\[
h^{\rm mono} = \frac{1}{2r}, \quad \star_{\mathbb{R}^3} d h^{\rm mono} = - \frac{\sin \varphi_1}{2} \, d \varphi_1 \wedge d \varphi_2,
\]
the local connection 1-form over, say the south hemisphere, is given by
\begin{equation} \label{A-single}
\mathcal{A}^{\rm mono} \overset{\rm loc}{=}  \frac{i}{2} (1 + \cos \varphi_1) d \varphi_2,
\end{equation}
solves $d \theta = \star d h$, and the $U(1)$ transition function $e^{i \tau}$ for this one charge is $e^{i \tau}=e^{i \varphi_2}$.

The full $\mathcal{A}^{\rm tn}$ solving $d \theta^{\rm tn} = \star d h^{\rm tn}$ is a sum of these local models with one contribution from each $y_a$. Over $\{|x| \gg 1\}$ with the south hemisphere trivialization for each charge, we have
\[
\mathcal{A}^{\rm tn} = \sum_{a=1}^k \mathcal{A}^{{\rm mono}}_{y_a}.
\]
The estimate that we will need from this construction is that
\begin{equation} \label{connection-A-decay}
| \mathcal{A}^{\rm tn}|_{g_E} \leq C |x|^{-1}, \quad |x| \gg 1,
\end{equation}
over both $U(1)$-bundle trivializations on $\{ |x| \gg 1 \}$.

 Consider the non-compact surface  
\[
L_{\rm cig} \subseteq N, \quad L_{\rm cig} = \pi^{-1}( \{ (x_1,0,0) : 1 \leq x_1 < \infty \})
\]
which is the pre-image of a straight line emanating from $y_1$ to infinity; see Figure \ref{fig:cigar}. We assume that the remaining points $\{ y_a \}_{a=2}^k$ do not intersect this straight line.

We write $g_{\rm cig}$ for the cigar metric, which is the geometry on $L_{\rm cig}$ induced from $(N,g^{\rm tn})$. The cigar metric \eqref{cigar-metric} in coordinates is
\[
g_{\rm cig} = h \, dx \otimes dx + h^{-1} (d \psi - i \mathcal{A}_1(x) dx )^2
  \]
 where the 1-variable function
  \[
h: (1, \infty) \rightarrow (0,\infty)
  \]
  is given by
  \begin{equation} \label{cigar-h}
h(x) = 1 + \frac{1}{2} \sum_{a=1}^k \frac{1}{( |x-y_a^1|^2 + |y_a^2|^2+ |y_a^3|^2)^{1/2}}.
  \end{equation}
  The submanifold $(L_{\rm cig}, g_{\rm cig})$ is asymptotically cylindrical in the following sense. There exists a compact set $K \subseteq L_{\rm cig}$ such that
  \[
L_{\rm cig} \backslash K \cong  (R, \infty) \times S^1
  \]
  and
  \begin{equation} \label{cig-cyl-est}
|g_{\rm cig} - g_{\rm e}|_{g_{\rm e}} + |x| |\partial g_{\rm cig}|_{g_{\rm e}} \leq C |x|^{-1}, \quad {\rm on } \ L_{\rm cig} \backslash K,
\end{equation}
where $g_{\rm e} = dx^2 + d \psi^2$ on $(R,\infty) \times S^1$. We now state a prove a lemma used later on in the gluing construction of \S \ref{sec:invertL}. We remark that a related result is obtained in \cite{FoscoloTrinca}, where it is proved that bounded holomorphic sections of the normal bundle of the cigar in a single-charge Taub-NUT must be zero; the difference here is that we consider cigars in multi-Taub-NUT and assume the stronger condition of decay rather than boundedness.

\begin{figure}
\begin{tikzpicture}[scale=1.1]

  \def\R{0.70}        
  \def\xTip{1.0}      
  \def\xJoin{5.3}     
  \def\xEnd{6.6}      

  \fill[gray!8]
    (\xTip,0)
      .. controls (1.25,0.55) and (2.2,\R) .. (3.6,\R)
      .. controls (4.6,\R) and (\xJoin-0.2,\R) .. (\xJoin,\R)
      -- (\xEnd,\R)
      .. controls (\xEnd+0.18,\R) and (\xEnd+0.18,-\R) .. (\xEnd,-\R)
      -- (\xJoin,-\R)
      .. controls (\xJoin-0.2,-\R) and (4.6,-\R) .. (3.6,-\R)
      .. controls (2.2,-\R) and (1.25,-0.55) .. (\xTip,0)
      -- cycle;

  \begin{scope}
    \clip
      (\xTip,0)
        .. controls (1.25,0.55) and (2.2,\R) .. (3.6,\R)
        .. controls (4.6,\R) and (\xJoin-0.2,\R) .. (\xJoin,\R)
        -- (\xEnd,\R)
        .. controls (\xEnd+0.18,\R) and (\xEnd+0.18,-\R) .. (\xEnd,-\R)
        -- (\xJoin,-\R)
        .. controls (\xJoin-0.2,-\R) and (4.6,-\R) .. (3.6,-\R)
        .. controls (2.2,-\R) and (1.25,-0.55) .. (\xTip,0)
        -- cycle;

    \foreach \x/\yr in {1.6/0.62, 3.6/\R, 6.0/\R} {
      \draw
        (\x,-\yr) arc[start angle=-90, end angle=90,
                     x radius=0.20, y radius=\yr];
      \draw[dashed]
        (\x,\yr) arc[start angle=90, end angle=270,
                     x radius=0.20, y radius=\yr];
    }
  \end{scope}

  \draw[thick]
    (\xTip,0)
      .. controls (1.25,0.55) and (2.2,\R) .. (3.6,\R)
      .. controls (4.6,\R) and (\xJoin-0.2,\R) .. (\xJoin,\R)
      -- (\xEnd,\R)
      .. controls (\xEnd+0.18,\R) and (\xEnd+0.18,-\R) .. (\xEnd,-\R);

  \draw[thick]
    (\xTip,0)
      .. controls (1.25,-0.55) and (2.2,-\R) .. (3.6,-\R)
      .. controls (4.6,-\R) and (\xJoin-0.2,-\R) .. (\xJoin,-\R)
      -- (\xEnd,-\R)
      .. controls (\xEnd+0.18,-\R) and (\xEnd+0.18,\R) .. (\xEnd,\R);

  \draw[->,thick] (\xEnd,0) -- (7.2,0) node[right] {$\infty$};

  \foreach \x in {-1,0,1} {
    \fill (\x,0) circle (2pt);
  }

\end{tikzpicture}

\caption{The cigar $L_{\rm cig}$ as it appears in the multi-Taub-NUT space with three points.} \label{fig:cigar}
\end{figure}
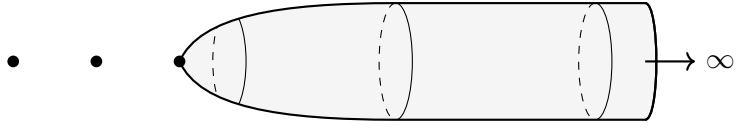

  \begin{prop} \label{cig-no-slow-forms}
   Let $\iota>0$. Let $a \in \Omega^1(L_{\rm cig}, g_{\rm cig})$ satisfy
    \[
      (d+d^\dagger) a = 0
    \]
    and
    \[
|a| + |x| |\nabla a| \leq C |x|^{-\iota}, \quad {\rm on} \ L_{\rm cig} \backslash K.
    \]
Then $a \equiv 0$.    
\end{prop}

\begin{proof}
  We use that $H^1(L_{\rm cig})=0$, and so $a = d u$ for some function $u$ solving
  \[
\Delta_{\rm cig} u = 0
  \]
  and
  \[
|u| + |x| |\nabla u| + |x|^2 |\nabla^2 u | \leq C |x|^{1-\iota}, \quad {\rm on} \ L_{\rm cig} \backslash K.
  \]
  We will then conclude that $a \equiv 0$ since every harmonic function on $(L_{\rm cig}, g_{\rm cig})$ with sublinear growth must be constant. This is standard, but we give the full details for completeness. The main step is to show
    \begin{equation} \label{cig-u-bdd}
|u| \leq C, \quad {\rm on} \ L_{\rm cig} \backslash K.
\end{equation}
We complete the proof using \eqref{cig-u-bdd} and prove it afterwards.  Let $\chi_R$ be a cutoff function such that
\[
\chi_R(x)=
\begin{cases}
1, & |x|<R,\\
0, & |x|>2R,
\end{cases}
\qquad
|\nabla \chi_R|_{\rm cig} \leq C R^{-1}.
\]
Integrating by parts and using $\Delta_{\rm cig} u = 0$ gives
\[
\int_{L_{\rm cig}} \chi_R^2 |\nabla u|^2 \, d {\rm vol}_{\rm cig} = - 2 \int_{L_{\rm cig}} u \chi_R \langle \nabla \chi_R, \nabla u \rangle \, d {\rm vol}_{\rm cig},
\]
from which follows the Caccioppoli inequality
\[
\int_{L_{\rm cig}} \chi_R^2 |\nabla u|^2 \, d {\rm vol}_{\rm cig} \leq C \int_{ \{ |x| < 2 R \}} u^2 |\nabla \chi_R|^2 \, d {\rm vol}_{\rm cig}.
\]
Since $|u| \leq C$ and ${\rm Vol}\{ |x| < R \} \leq C R$, we obtain
\[
\int_{L_{\rm cig}} \chi_R^2 |\nabla u|^2 \, d {\rm vol}_{\rm cig} \leq C R^{-1}.
\]
Sending $R \rightarrow \infty$ implies $\nabla u \equiv 0$. Therefore $u$ is a constant function and $a \equiv 0$.

 We now prove \eqref{cig-u-bdd}. First, we write
  \[
\Delta_{\rm e} u = (\Delta_{\rm e} - \Delta_{\rm cig}) u
\]
and use \eqref{cig-cyl-est} to estimate
\[
f:= \Delta_{\rm e} u, \quad |f| \leq C |x|^{-2 - \iota},
\]
where $f$ is defined on $(R,\infty) \times S^1$. Extend $f$ arbitrarily to $\tilde{f}$ on $\mathbb{R} \times S^1$ in such a way that
\[
{\rm supp} \, \tilde{f} \subseteq [0,\infty) \times S^1.
\]
We can expand $\tilde{f}$ by Fourier series.
\[
\tilde{f}(x,e^{i \psi}) = \sum_{n=-\infty}^\infty a_n(x) e^{i n \psi},
\]
and note
\[
|a_n(x)| \leq C |x|^{-2 -\iota}, \quad {\rm supp} \, a_n \subseteq [0,\infty).
\]
Substituting the ansatz
\[
u^{\rm decay} = \sum_{n=-\infty}^\infty b_n(x) e^{i n \psi}
\]
into the equation $\Delta_{\rm e} u^{\rm decay} = \tilde{f}$ gives the ODEs
\[
b_n'' - n^2 b_n = a_n
\]
which are readily solved. Indeed, by integrating $a_0(x)$ twice we can produce a choice of $b_0(x)$ such that
\[
|b_0(x)| \leq C |x|^{-\iota},
\]
and $b_n(x)$ for $n \geq 1$ have the well-known harmonic oscillator ODE solution, which is commonly derived using the Fourier transform in an engineering course, given by
\[
b_n(x) = - \frac{1}{2n} \int_{-\infty}^\infty a_n(x-s) e^{-n|s|} \, ds, \quad |b_n| \leq C n^{-2} |x|^{-2 - \iota}.
\]
Thus we have constructed a solution to
\[
\Delta_{\rm e} u^{\rm decay} = \tilde{f}, \quad |u^{\rm decay}| \leq C |x|^{-\iota}, \quad {\rm on} \ \mathbb{R} \times S^1.
\]
Returning to our given $u$, we have
\[
\Delta_{\rm e}( u - u^{\rm decay}) = 0, \quad |u- u^{\rm decay}| \leq C|x|^{1-\iota}, \quad {\rm on} \ (R,\infty) \times S^1.
\]
Writing
\[
  u- u^{\rm decay} = \sum_{n=-\infty}^\infty b_n(x) e^{i n \psi},
  \]
  we see that the Laplace equation implies $b_0 = c_1 + c_2 x$ and hence $c_2 = 0$, and also $b_n = c_1 e^{nx} + c_2 e^{-nx}$ and hence $c_1=0$. Therefore
  \[
|u- u^{\rm decay} | \leq C, \quad {\rm on} \ (R,\infty) \times S^1,
  \]
  from which estimate \eqref{cig-u-bdd} follows.

  \end{proof}

\section{Collapsing K3 Surfaces} \label{sec:foscolosummarized}
In this section, we summarize Foscolo's gluing construction of K3 surfaces \cite{FoscoloK3}. This produces a sequence of K3 surfaces which collapse to $\mathbb{T}^3/\mathbb{Z}_2$.

\subsection{Metrics on the punctured torus}
Let $\mathbb{T} = \mathbb{R}^3 / \Lambda$ be a 3-torus with flat metric $g_{\mathbb{T}}$, where $\Lambda$ is a lattice. Let $\{ q_j \}_{j=1}^8$ be the fixed points of the involution $\tau: \mathbb{T} \rightarrow \mathbb{T}$ given by $x \mapsto -x$. Let $\{ p_i \}_{i=1}^{n}$ be $n$ distinct points on $\mathbb{T}$ not equal to any $q_j$. 

We will work on the punctured torus
\[
\mathbb{T}^* = \mathbb{T} \backslash \{ q_1, \dots, q_8, p_1, -p_1, \dots, p_n, -p_n \}.
\]
For each point in $x \in \{ q_1, \dots, q_8, p_1, -p_1, \dots, p_n, -p_n \}$, fix a small neighborhood of the chosen point $x$ and let $r_x$ be the local distance function to $x$. We arbitrarily extend these local functions $r_x$ and denote by
\[
r: \mathbb{T} \rightarrow [0,\infty)
\]
an arbitrary extension to the entire compact $\mathbb{T}$ such that
\[
  \{ r = 0 \} = \bigcup \{q_j \} \cup \bigcup \{ p_i \} \cup \bigcup \{ - p_i \}
\]
and $r$ agrees with the local distance function on small neighborhoods of the punctures. 

\subsubsection{Harmonic function with prescribed singularities}
Attach to each $\{ q_j \}_{j=1}^8$ and $\{ p_i \}_{i=1}^{n}$ the following integers
\[
\{ m_j \}_{j=1}^8, \quad \{ k_i \}_{i=1}^n,
\]
where $m_j \geq 0$ and $k_i > 0$. The global gluing construction relies on the following observation:

\begin{lem} \cite{FoscoloK3} \label{harmonic-presc-sing}
Suppose the following balancing condition holds:
\begin{equation} \label{balancing-cond}
\sum_{j=1}^8 m_j + \sum_{i=1}^n k_i = 16.
\end{equation}
Then there exists a harmonic function $h$ on $(\mathbb{T}^*,g_{\mathbb{T}})$ with prescribed singularities
\begin{align*}
  h &\overset{\rm loc}{=} \frac{2 m_j - 4}{2 r} + O(1), \quad {\rm near} \ q_j \\
  h &\overset{\rm loc}{=} \frac{k_i}{2 r} + O(1), \quad {\rm near} \ p_i \ {\rm or} \ -p_i.
\end{align*}
\end{lem}

\begin{proof}
We give the proof to see the role of the balancing condition. First, take $f \in C^\infty(\mathbb{T}^*)$ to be an arbitrary global function extending the local functions
\begin{align*}
  f &\overset{\rm loc}{=} \frac{2 m_j - 4}{2 r}, \quad {\rm near} \ q_j \\
  f &\overset{\rm loc}{=} \frac{k_i}{2 r}, \quad {\rm near} \ p_i \ {\rm or} \ -p_i.
\end{align*}
Since $\Delta_{\mathbb{R}^3} {1 \over r} \equiv 0$ in a neighborhood of each puncture, we have
\[
\Delta_{g_{\mathbb{T}}} f = \psi, \quad \psi \in C^\infty(\mathbb{T})
\]
where $\psi$ can be viewed as a smooth bounded function on the entire compact torus $\mathbb{T}$ vanishing in a neighborhood of each distinguished point. The balancing condition \eqref{balancing-cond} implies
\begin{align*}
  \int_{\mathbb{T}} \psi \, d {\rm vol}_{\mathbb{T}} &= -\sum_\ell \int_{\partial B_\epsilon(x_\ell)} \nabla f \cdot {\bf n}^{\rm out} \, dS\\
                                                     &= 2 \pi \bigg[ \sum_j (2 m_j - 4) + \sum_i k_i + \sum_i k_i \bigg]\\
  &= 0
\end{align*}
Here $x_\ell$ denotes a $q_j$, $p_i$, or $-p_i$. Since the integral vanishes over the compact torus, we may solve $\Delta_{g_{\mathbb{T}}} \bar{f} = \psi$ for a smooth function $\bar{f} \in C^\infty(\mathbb{T})$. We then let
\[
h = f - \bar{f}
\]
be the desired harmonic function with prescribed singularities.
\end{proof}

We will write
\begin{equation} \label{h-sing}
\begin{aligned} 
  h &\overset{\rm loc}{=} \lambda_i + \frac{k_i}{2 r} + \ell_i+ O(r^2), \quad {\rm near} \ p_i \ {\rm or} \ -p_i,\\
   h &\overset{\rm loc}{=} \lambda_j + \frac{2m_i-4}{2 r} + \ell_j+ O(r^2), \quad {\rm near} \ q_j \nonumber
\end{aligned}
\end{equation}
where $|\ell_i| \leq C r$ has linear growth.

\subsubsection{Torus bulk region}
Let $h$ be the harmonic function from Lemma \ref{balancing-cond}. By (\cite{FoscoloK3} Proposition 4.3), there exists a principal $U(1)$-bundle $P \rightarrow \mathbb{T}^*$ with a connection 1-form $\theta \in \Omega^1(P)$ solving
\[
d \theta = \star_{g_{\mathbb{T}}} d h.
\]
We set
\[
h_\epsilon = 1 + \epsilon h,
\]
so that near $p_i$ we have the expansion
\[
h_\epsilon \overset{\rm loc}{=} 1+ \epsilon \lambda_i + \frac{\epsilon k_i}{2 r} + \epsilon \ell_i+ \epsilon O(r^2).
\]
Let $(x_1,x_2,x_3)$ be angle coordinates on $\mathbb{T}^3$. Foscolo's hyperk\"ahler bulk model $\underline{\boldsymbol{\omega}}^{\rm bulk}_{\epsilon}$ over $P$ is:
\begin{equation} \label{bulk-metric}
\begin{aligned}
  \omega^{\rm bulk}_{\epsilon,a} &= \epsilon dx_a \wedge \theta + h_\epsilon dx_b \wedge dx_c, \quad \epsilon_{abc} = 1 \\
  g^{\rm bulk}_\epsilon &= h_\epsilon \pi^* g_{\mathbb{T}^3} + \epsilon^2 h_\epsilon^{-1} \theta \otimes \theta .
\end{aligned}
\end{equation}
The ansatz defines a hyperk\"ahler triple where $ \{ h_\epsilon > 0 \}$. In the upcoming gluing construction, we will only work on the set $P \cap \{ r > R_0 \epsilon \}$ where $R_0 \gg 1$, and on this set $h_\epsilon > 0$ once $\epsilon$ is taken small enough. The involution $\tau$ on $\mathbb{T}^*$ lifts to the circle bundle $P$ \cite{FoscoloK3}. We take the quotient by this involution and obtain the manifold
\[
M_\epsilon^{*} = P / \tau.
\]
We will need an asymptotic expansion of $\theta = d \psi - i \mathcal{A}$ near $p_i$. Let $\mathcal{A}^{k_i}$ be the explicit solution to
\[
-i d \mathcal{A}^{k_i} =  \star_{g_{\mathbb{T}}} d \bigg[ \frac{k_i}{2 r} + \ell_i \bigg]
\]
given by
\begin{equation} \label{multi-A}
\mathcal{A}^{k_i} =  k_i \mathcal{A}^{\rm mono} + \frac{i}{2} \epsilon_{\alpha \beta \gamma} c^\alpha x_\beta dx_\gamma,
\end{equation}
where $\ell_i = c^\alpha x_\alpha$ and $\mathcal{A}^{\rm mono}$ is defined in \eqref{A-single}, valid on one of the two trivializations of the bundle over $B_{\rho} \backslash \{ p_i \}$. We claim that up to choice of gauge, we have
\begin{equation} \label{bulk-A}
| \mathcal{A} - \mathcal{A}^{k_i}|_{g_{\mathbb{T}}} \leq C  r^2, \quad |\mathcal{A}|_{g_{\mathbb{T}}} \leq C r^{-1},
  \end{equation}
near $p_i$. For this, we start from
\[
d (\mathcal{A}^{k_i} - \mathcal{A}) = \star_{g_{\mathbb{T}}} d \bigg[ O (r^2) \bigg].
\]
In other words,
\[
  d (\mathcal{A}^{k_i} - \mathcal{A}) = F, \quad |F|_{g_{\mathbb{T}}} \leq C r, \quad {\rm on} \ B_{\rho} \backslash \{ p_i \}.
\]
In polar coordinates $(r,\varphi^1, \varphi^2)$ on $\mathbb{R}^3$ centered at $p_i$, we write
\[
F = F_{ri} \, dr \wedge d \varphi^i + F_{12} \, d \varphi^1 \wedge d \varphi^2
\]
and solve $d \alpha = F$ with $|\alpha|_{g_{\mathbb{T}}} \leq C r^2$ by $dF = 0$ and radial integration
\[
\alpha = \bigg[ \int_0^r F_{ri} (s, \cdot) \, ds \bigg] \, d \varphi^i.
\]
Thus
\[
d (\mathcal{A}^{k_i} - \mathcal{A} - \alpha) = 0, \quad {\rm on} \ B_{\rho} \backslash \{ p_i \}.
\]
Since the punctured ball is simply-connected, there exists a function $f$ such that $\mathcal{A}^{k_i} - \mathcal{A} - \alpha = df$. A gauge transformation on the $U(1)$ bundle has the effect of
\[
\mathcal{A} \mapsto \mathcal{A} +df ,
\]
which sets the gauge for estimate \eqref{bulk-A}, where the estimate is
\[
|\mathcal{A} +df - \mathcal{A}^{k_i}|_{g_{\mathbb{T}}} = |\alpha|_{g_{\mathbb{T}}} \leq C r^2.
\]
Introducing the factors of $\epsilon$ by the definition of $g_\epsilon^{\rm bulk}$ gives \eqref{bulk-A}.

\subsection{Glued compact 4-manifold} \label{sec:gluingregion}
The space $M_\epsilon^{*}$ will be desingularized by gluing-in suitable local models to form a compact four-manifold $M_\epsilon$. We remove from $M_\epsilon^{*}$ small balls around the 8 points $q_j$ and the $n$ points $p_i$ to obtain a manifold with boundary with $n+8$ components. We form $M_\epsilon$ as a smooth compact 4-manifold obtained by gluing local models onto these boundary components. The 8 boundary components corresponding to $q_j$ are matched with a $D_{m_j}$ ALF space (for a construction of these, see \cite{AtiyahHitchin, BiquardMinerbe, SS}). This gluing at $q_j$ is explained in \cite{FoscoloK3}, but since we will work away from this region we omit the details. We focus on the $n$ boundary components corresponding to the points $p_i$.

\subsubsection{Bubble at $p_i$}
To each of the $n$ points $p_i \in M_\epsilon^*$ with corresponding weight $k_i$, we will introduce a multi-Taub-NUT space $N_i$ with $k_i$ distinct points following conventions in Section \ref{sec:multiTAUB}. Specifically, we build a multi-Taub-NUT space $N_i \overset{\pi}{\rightarrow} U_{i}$, where
\[
U_{i} = B_{\rho_0 \epsilon^{-1}} \backslash \{ y_1, \dots, y_{k_i} \}
\]
with $y_r \in \mathbb{R}^3$ distinct points chosen to satisfy the vanishing dipole moment condition
\begin{equation} \label{zero-dipole}
\sum_{a=1}^{k_i} y_a = 0.
\end{equation}
This condition will be needed later on in the error estimate \eqref{zero-dipole2}. We take the harmonic function building the circle bundle to be
\begin{equation} \label{N_i-localgeom}
h_{{\rm tn}, i} = 1 + \epsilon \lambda_i + \sum_{a=1}^{k_i} \frac{1}{2|x-y_a|} + \epsilon^2 \ell_i.
\end{equation}
The constants $\lambda_i$ and linear function $\ell_i$ comes from the asymptotics of the harmonic function on the punctured $\mathbb{T}^*$ near $p_i$; see \eqref{h-sing}. We have positivity $h_{{\rm tn}, i}>0$ over $U_i$ for $\epsilon>0$ small enough. To summarize, to each point $p_i$ we associate the local model:
\[
p_i \rightsquigarrow (N_i, h_{{\rm tn},i}, \underline{\boldsymbol{\omega}}^{{\rm tn},i}).
  \]
We will rescale this Taub-NUT local model $N_i$ to match-up with an annulus surrounding $p_i \in M_\epsilon^*$ in the bulk space by using the scaling map $S_\lambda: x \mapsto \lambda x$ given in \eqref{scaling}. Explicitly, near each $p_i$ we identify
\[
A_i = \{R_0 \epsilon < r < \rho_0 \} \subseteq P
\]
on the bulk space circle bundle $P \overset{\pi}{\rightarrow} \mathbb{T}^*$ with
\[
A^{\rm tn}_i = \{R_0 \epsilon <  |x| < \rho_0 \} \subseteq N_{i,\epsilon}
\]
where $N_{i,\epsilon} \overset{\pi}{\rightarrow} U_{i,\epsilon}$ is given by
\[
N_{i, \epsilon} = S_\epsilon(N_i), \quad U_{i,\epsilon} = B_{\rho_0} \backslash \{\epsilon y_1, \dots, \epsilon y_{k_i} \}.
\]
The constants $0<\rho_0<1$, $R_0>1$ are parameters that will be fixed later on. We identify $|x|$ on $N_{i,\epsilon}$ with $r$ on $P$ near $p_i$.

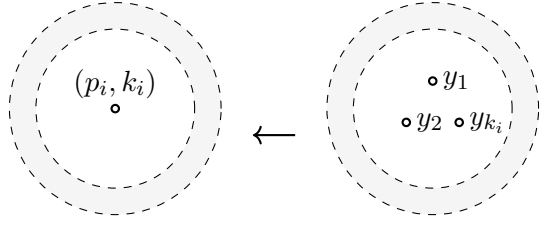
\begin{figure}
\begin{tikzpicture}[scale=0.7]
  \def\r{1.5}        
  \def\R{2.0}        
  \def\X{3}          
  \def\offset{-0.5}  

\fill[gray!8, even odd rule] 
    (-\X,0) circle (\R) (-\X,0) circle (\r);
  \draw[dashed] (-\X,0) circle (\r);
  \draw[dashed] (-\X,0) circle (\R);
  \draw[thick] (-\X,0) circle (0.07);
  \node[above] at (-\X,0) {$(p_i, k_i)$};

  \fill[gray!8, even odd rule] 
    (\X,0) circle (\R) (\X,0) circle (\r);
  \draw[dashed] (\X,0) circle (\r);
  \draw[dashed] (\X,0) circle (\R);

  \foreach \x/\y/\name in {
    \X/0.52/y_1,
    {\X - 0.5}/-0.26/y_2,
    {\X + 0.5}/-0.26/y_{k_i}
  } {
    \draw[thick] (\x,\y) circle (0.07);
    \node[right] at (\x,\y) {$\name$};
  }

  \draw[->, thick] (0.4, \offset) -- (-0.4, \offset);

\end{tikzpicture}

\caption{The singularity $p_i$ with weight $k_i$ is replaced by $k_i$ points of weight 1.} \label{fig:sing-weight}
\end{figure}

We now explain why these two annular regions $A_i$ and $A^{\rm tn}_i$ can be identified by diffeomorphism (see Figure \ref{fig:sing-weight}). Both bases may be identified with $B = \{ R_0 \epsilon < |x| < \rho_0 \} \subseteq \mathbb{R}^3$. The two regions $A_i$ and $A^{\rm tn}_i$ correspond to two circle bundles $L_1 \rightarrow B$ and $L_2 \rightarrow B$. Since $B$ retracts to a 2-sphere, by the classification of line bundles these bundles are isomorphic provided their Chern classes are equal. This amounts to verifying that the integrals
\[
\frac{1}{2 \pi} \int_{ \{ |x| = \rho_0 \}} d \theta
\]
match to the integer $k_i$, which holds true by the definitions of $h$ \eqref{h-sing} and $h^{\rm tn}_i$ \eqref{N_i-localgeom}. The identification between the two gluing regions is thus of the form
\begin{equation} \label{gluing-ann}
\phi: A_i \rightarrow A^{\rm tn}_i, \quad \phi(x, e^{i \psi}) \overset{\rm loc}{=} (x, f(x) e^{i \psi}).
\end{equation}

\subsubsection{Compactification} To summarize, near each $p_i$, we glue-in the local model $N_{i,\epsilon}$ via $M_\epsilon^* \cup_\phi N_{i,\epsilon}$. Near each $q_j$, the construction \cite{FoscoloK3} glues-in a local model $M_{j,\epsilon}$ to $M_\epsilon^*$, where $(M_{j, \epsilon},\epsilon^2 \omega_{M_{j,\epsilon}})$ is a $D_{m_j}$ ALF gravitational instanton. This defines a compact manifold $M_\epsilon$. 

\subsubsection{Weight} We define the weight function
\[
\rho: M_\epsilon \rightarrow [0,\infty)
\]
defined by
\begin{equation} \label{rho-defn}
\rho =
\begin{cases}
\frac{1}{2} R_0 \epsilon & \text{on } \{ |x| \leq \frac{1}{2} R_0 \epsilon \} \cap N_{i,\epsilon} , \\
|x| & \text{on } \{ \frac{1}{2} R_0 \epsilon  <  |x| < \rho_0 \} \cap A, \\
\frac{1}{2} R_0 \epsilon   & \text{on } \{ |x| \leq \frac{1}{2}R_0 \epsilon \} \cap M_{i,\epsilon}, \\
1 & \text{on } \{ |x| \geq 2 \rho_0 \} \cap A
\end{cases}
\end{equation}
and $\rho$ interpolates monotonically in the remaining regions. Here $A$ denotes either $A_i$ or $A_j$, where $A_i \subseteq M_\epsilon^*$ denotes a collar around $p_i$ glued to the local model $N_{i,\epsilon}$, and $A_j \subseteq M_\epsilon^*$ denotes a collar around $q_j$ glued to the local model $M_{i,\epsilon}$.

\subsection{Reference metric} \label{sec:refmetric}
Having described the topological gluing of manifolds from $M_\epsilon^*$ to the compact 4-manifold $M_\epsilon$, we now add geometry. We will use a reference triple $\underline{\boldsymbol{\omega}}_\epsilon$ which interpolates between $\underline{\boldsymbol{\omega}}^{\rm bulk}_\epsilon$ on the bulk torus region and the glued-in Taub-NUT local models $\underline{\boldsymbol{\omega}}^{{\rm tn},i}$ near the singularities $p_i$. The precise definition of the reference metric is given below in \eqref{bg-metric-regions}, but in this introductory passage we simply note that this cutoff function interpolation will be such that
\begin{align} \label{pre-perturbed-defn}
  \underline{\boldsymbol{\omega}}_\epsilon|_{\{ \rho \geq 2 \epsilon^{2/5} \}} &=  \underline{\boldsymbol{\omega}}_{\epsilon}^{\rm bulk} \\
\underline{\boldsymbol{\omega}}_\epsilon|_{\{ \rho_i \leq \epsilon^{2/5} \}} &=  \epsilon^2 S^*_{\epsilon^{-1}} \underline{\boldsymbol{\omega}}^{{\rm tn},i}.
\end{align}
Here $\rho_i$ denotes the function $\rho$ near $p_i$. The behaviour of $\underline{\boldsymbol{\omega}}_\epsilon$ near $q_j$ is analogous and described in \cite{FoscoloK3}, but our analysis will take place near the $p_i$ so we omit this.

\begin{rk}
The gluing collar $\epsilon^{2/5} \leq \rho \leq 2 \epsilon^{2/5}$ is chosen with a power of $\frac{2}{5}$ to control error terms on first derivatives of the metric with are often of order $\epsilon O(\rho^{-2})$; e.g. $|\partial h_\epsilon|$.
  \end{rk}

The Taub-NUT local models $\underline{\boldsymbol{\omega}}^{{\rm tn},i}$ specified earlier by the harmonic function given in \eqref{N_i-localgeom} are rescaled. The scaling is needed to match the local model with the small scale of the gluing region near each $p_i$. The explicit expression for the rescaled Taub-NUT local model is as follows. We will use the notation
\[
\underline{\boldsymbol{\omega}}^{{\rm tn},i}_\epsilon := \epsilon^2 S^*_{\epsilon^{-1}} \underline{\boldsymbol{\omega}}^{{\rm tn},i}.
\]
On
\[
\bigg[ \{ \rho_i \leq \epsilon^{2/5} \} \cap M_\epsilon \bigg] \cong \bigg[ \{ \rho \leq \epsilon^{2/5} \} \cap \{ N_\epsilon \overset{\pi}{\rightarrow} U_{\epsilon} \} \bigg]
\]
the local model is, where we drop for the moment the subscript $i$ labelling the local zone nearby $p_i$, given by
\begin{equation} \label{explicit-local-tn}
  \begin{aligned}
    (\omega^{{\rm tn}}_\epsilon)_\alpha &= \epsilon dx_\alpha \wedge \theta^{\rm tn}_\epsilon + h^{\rm tn}_\epsilon \, dx_\beta \wedge d x_\gamma, \quad \epsilon_{\alpha \beta \gamma}=1,\\
    g^{\rm tn}_\epsilon &= h^{\rm tn}_\epsilon g_{\mathbb{R}^3} + \epsilon^2 (h^{\rm tn}_\epsilon)^{-1}  \theta^{\rm tn}_\epsilon \otimes \theta^{\rm tn}_\epsilon,
    \end{aligned}
\end{equation}
where 
\begin{equation} \label{explicit-local-htn}
h^{\rm tn}_\epsilon(x) = 1 + \epsilon \lambda + \epsilon \sum_{a=1}^k \frac{1}{2 |x-\epsilon y_a|} + \epsilon \ell(x),
\end{equation}
and
\[
\theta^{\rm tn}_\epsilon = d \psi - i S_{\epsilon^{-1}}^* \mathcal{A}^{\rm tn},
\]
solves $d \theta^{\rm tn}_\epsilon = \epsilon^{-1} \star_{\mathbb{R}^3} d h^{\rm tn}_\epsilon$, in which
\[
\mathcal{A}^{\rm tn} = \sum_{a=1}^k \mathcal{A}^{{\rm mono}}_{y_a} + \frac{i \epsilon^2}{2} \epsilon_{\alpha \beta \gamma} c^\alpha x_\beta dx_\gamma
\]
and $\ell(x) = c^\alpha x_\alpha$ and $\mathcal{A}^{\rm mono}$ is given in \eqref{A-single}.



The reference triple of closed 2-forms $\underline{\boldsymbol{\omega}}_\epsilon$ defines a reference metric $g_\epsilon$ by \eqref{hk-metric}. We denote this structure on the glued compact 4-manifold as:
\[
(M_\epsilon, \underline{\boldsymbol{\omega}}_\epsilon, g_\epsilon).
\]
We refer to $\underline{\boldsymbol{\omega}}_\epsilon$ as the reference triple and $g_\epsilon$ the reference (or pre-perturbed) metric. This is not yet a hyperk\"ahler structure, as there will be errors from a cutoff function on the interpolation region. Foscolo \cite{FoscoloK3} later perturbs $\underline{\boldsymbol{\omega}}_\epsilon$ by implicit function theorem to a genuine hyperk\"ahler structure.


\subsubsection{Construction of reference metric}
We now give the details of the definition of $\underline{\boldsymbol{\omega}}_\epsilon$. We work on the glued collar $A_i$ near a given point $p_i$ which we identify with $A_i^{\rm tn}= \{ R_0 \epsilon < \rho < \rho_0 \}$ on $N_i \overset{\pi}{\rightarrow} U_{i, \epsilon}$. We drop the $i$ subscript in this subsection. The main technical input in the construction of the reference metric \cite{FoscoloK3} is that there exists $\underline{\boldsymbol{u}}_\epsilon \in \Omega^1(A)^{\oplus 3}$ such that
\begin{equation} \label{glue-collar}
\begin{aligned} 
  \underline{\boldsymbol{\omega}}_\epsilon^{\rm bulk}|_A &=  \underline{\boldsymbol{\omega}}^{\rm tn}_\epsilon + d \underline{\boldsymbol{u}}_\epsilon, \\
  \sup_{\{ \epsilon^{2/5} \leq \rho \leq 2\epsilon^{2/5} \}} |\nabla^k \underline{\boldsymbol{u}}_\epsilon|_{g^{\rm cyl}_\epsilon} &\leq C \epsilon^2 \epsilon^{\frac{1}{5}} \epsilon^{-\frac{2}{5}k}, \nonumber
\end{aligned}
\end{equation}
where $g_\epsilon^{\rm cyl} = g_{\mathbb{R}^3} + \epsilon^2 \theta \otimes \theta$. Assuming \eqref{glue-collar}, let $\zeta: [0,\infty) \rightarrow [0,1]$ be a cutoff function with $\zeta|_{\{ \rho \leq 1 \}} \equiv 0$ and $\zeta|_{\{ \rho \geq 2 \}} \equiv 1$, and define $\zeta_\epsilon = \zeta(\epsilon^{-2/5} \rho)$ so that $|\nabla \zeta_\epsilon| \leq C \epsilon^{-2/5}$. Let
\begin{equation} \label{bg-metric-regions}
\underline{\boldsymbol{\omega}}_\epsilon =
\begin{cases}
\underline{\boldsymbol{\omega}}^{\rm tn}_\epsilon & \text{on } \{ R_0 \epsilon \leq \rho \leq \epsilon^{2/5} \} , \\
\underline{\boldsymbol{\omega}}^{\rm tn}_\epsilon + d (\zeta_\epsilon \underline{\boldsymbol{u}}_\epsilon) & \text{on } \{ \epsilon^{2/5} \leq \rho \leq 2\epsilon^{2/5} \}, \\
\underline{\boldsymbol{\omega}}_\epsilon^{\rm bulk}  & \text{on } \{ 2\epsilon^{2/5} \leq \rho \leq \rho_0 \}.
\end{cases}
\end{equation}
This describes the geometry on the gluing region $A_i^{\rm tn}$ near $p_i$. Inside the glued local model $N_{\epsilon}$, we set
\[
\underline{\boldsymbol{\omega}}_\epsilon = \underline{\boldsymbol{\omega}}^{\rm tn}_\epsilon \quad \text{on } \{ \rho \leq R_0 \epsilon \} \cap N_{\epsilon}.
\]
This local gluing is to be repeated at each $p_i$. Near each $q_j$, the collar $\{ R_0 \epsilon < \rho < \rho_0 \} \subseteq M_\epsilon^*$ is identified with $\{ R_0 \epsilon < \rho < \rho_0 \} \subseteq M_{j}$, where $(M_{j}, \epsilon^2 \underline{\boldsymbol{\omega}}_{M_j})$ is a $D_{m_j}$ ALF gravitational instanton. Similarly, here we have
\begin{equation} 
\underline{\boldsymbol{\omega}}_\epsilon =
\begin{cases}
\underline{\boldsymbol{\omega}}_{M_j} & \text{on } \{ R_0 \epsilon \leq \rho \leq \epsilon^{2/5} \} , \\
\underline{\boldsymbol{\omega}}_{M_j} + d (\zeta_\epsilon \underline{\boldsymbol{u}}_\epsilon) & \text{on } \{ \epsilon^{2/5} \leq \rho \leq 2\epsilon^{2/5} \}, \\
\underline{\boldsymbol{\omega}}_\epsilon^{\rm bulk}  & \text{on } \{ 2\epsilon^{2/5} \leq \rho \leq \rho_0 \}.
\end{cases}
\end{equation}
where $u_\epsilon$ is an interpolating term between the bulk triple and the local model triple satisfying \eqref{glue-collar}. On the remaining set $\{ \rho \geq \rho_0 \}$ outside of the union of the gluing collars $A_i$, $A_j$, we set $\underline{\boldsymbol{\omega}}_\epsilon =\underline{\boldsymbol{\omega}}_\epsilon^{\rm bulk}$. This defines the reference triple on the compact manifold $M_\epsilon$.

To complete the construction of the reference metric $g_\epsilon$, it remains to prove the bound \eqref{glue-collar}. We include the details for this at the end of this section in \S \ref{sec:global2single}, \S \ref{sec:single2multi} and \S \ref{sec:addup}.

\subsubsection{Comparison to cylindrical geometry}
The region of the compact manifold $M_\epsilon$ given by $\{ \rho \geq R_0 \epsilon \}$ stays away from the full depth of the bubbles. When working here, we will use the model cylindrical metric $g_\epsilon^{\rm cyl}$ to take norms.
\[
g_\epsilon^{\rm cyl} = g_{\mathbb{R}^3} + \epsilon^2 \theta \otimes \theta.
\]
We will often use the following rough estimates for the reference metric $g_\epsilon$: over $\{ \rho \geq \epsilon R_0 \}$, there holds
\begin{equation} \label{geps2cyl}
  \begin{aligned}
|g_\epsilon - g_\epsilon^{\rm cyl}|_{g_\epsilon^{\rm cyl}} &\leq \frac{1}{2},\\
    \rho^k |(\nabla_{g_\epsilon^{\rm cyl}})^k g_\epsilon|_{g_\epsilon^{\rm cyl}} &\leq C_k, 
    \end{aligned}
\end{equation}
for fixed $R_0 \gg 1$ and all $\epsilon$ small enough. Over this region we will often move between measuring tensors with respect to $g_\epsilon^{\rm cyl}$ or $g_\epsilon$.

To understand \eqref{geps2cyl}, recall that $g_\epsilon$ is built from two pieces: the bulk metric and the Taub-NUT metric. 

\begin{itemize}
\item Comparison of $g_\epsilon^{\rm bulk}$ to $g_\epsilon^{\rm cyl}$. Recall $h_\epsilon = 1 + \epsilon h$ with $h$ locally given by \eqref{h-sing}, and recall $g_\epsilon^{\rm bulk}$ is defined in \eqref{bulk-metric}. Since $\rho^k \partial^k h$ is a bounded function on $\mathbb{T}$, the bound $\rho^k | (\nabla_\epsilon^{\rm cyl})^k g_\epsilon^{\rm bulk}|_{g_\epsilon^{\rm cyl}} \leq C$ holds on $P \cap \{ \rho > R_0 \epsilon \}$. This is all that is needed for the bulk part of the rough estimate \eqref{geps2cyl}, but later on we will also use the following refinements.

Consider the harmonic function $h_\epsilon$ restricted to $\{ \rho \geq \epsilon^{\frac{2}{5}} \}$. By construction,
\[
\sup_{\{ \rho = \epsilon^{2/5} \}} |h_\epsilon - 1| \leq C \epsilon^{\frac{3}{5}}.
\]
Since $h_\epsilon$ is harmonic, by the maximum principle this bound holds on $\{ \rho \geq \epsilon^{\frac{2}{5}} \}$. Next, on the flat torus $\mathbb{T}^*$ the partial derivatives $\partial_{x_i} h$ are well-defined harmonic functions. Differentiating \eqref{h-sing} gives
\[
\sup_{\{ \rho = \epsilon^{2/5} \}} |\partial h_\epsilon| \leq C \epsilon^{\frac{1}{5}}.
\]
Therefore
\begin{equation} \label{h-estimates}
\sup_{\{ \rho \geq \epsilon^{2/5} \}} |h_\epsilon - 1| \leq C \epsilon^{\frac{3}{5}}, \quad \sup_{\{ \rho \geq \epsilon^{2/5} \}} |\partial h_\epsilon| \leq C \epsilon^{\frac{1}{5}}.
\end{equation}
From \eqref{h-estimates}, we derive comparison estimates.
\begin{equation} \label{bulk2cyl}
\begin{aligned} 
\sup_{\{ \rho \geq \epsilon^{2/5} \}} |g_\epsilon^{\rm bulk} - g_\epsilon^{\rm cyl}|_{g_\epsilon^{\rm cyl}} &\leq C \epsilon^{\frac{3}{5}}, \\
  \sup_{\{ \rho \geq \epsilon^{2/5} \}} |\nabla_\epsilon^{\rm cyl} g_\epsilon^{\rm bulk}|_{g_\epsilon^{\rm cyl}} &\leq C \epsilon^{\frac{1}{5}}. 
\end{aligned}
\end{equation}
We will use these estimates in later on in Section \S \ref{sec:invertL}.

\item Comparison of $g^{\rm tn}_\epsilon$ to $g_\epsilon^{\rm cyl}$. The explicit expression for the rescaled Taub-NUT metric is given in \eqref{explicit-local-tn}, and working with this directly leads to the following estimates over ${\{ \epsilon R_0 \leq \rho \leq \rho_0 \}}$:
  \begin{equation} \label{bdd-htn}
|h^{\rm tn}_\epsilon-1| \leq C \epsilon \rho^{-1}, \quad h_\epsilon^{\rm tn} \geq C^{-1}, \quad \rho^k |\partial^k h^{\rm tn}_\epsilon| \leq C_k \epsilon \rho^{-1}.
    \end{equation}
Using this with $|S^*_{\epsilon^{-1}} \mathcal{A}^{\rm tn}|_{\mathbb{R}^3} \leq C \rho^{-1}$ and $|\mathcal{A}|_{\mathbb{R}^3} \leq C \rho^{-1}$ (see \eqref{bulk-A}) to control off-diagonal terms involving connection forms leads to
\begin{equation} \label{g-tn-infty}
  \begin{aligned}
 |g^{\rm tn}_\epsilon - g_\epsilon^{\rm cyl}|_{g_\epsilon^{\rm cyl}} &\leq C \epsilon  \rho^{-1}, \quad {\rm on} \ {\{ \epsilon R_0 \leq \rho \leq \rho_0 \}} \\
 |\nabla_\epsilon^{\rm cyl} g^{\rm tn}_\epsilon |_{g_\epsilon^{\rm cyl}} &\leq C \epsilon \rho^{-2}, \quad {\rm on}  \ {\{ \epsilon R_0 \leq \rho \leq \rho_0 \}} .
\end{aligned}
\end{equation}
The bound
\[
  \rho^k | (\nabla_\epsilon^{\rm cyl})^k g^{\rm tn}_\epsilon |_{g_\epsilon^{\rm cyl}} \leq C_k
\]
can also be obtained from the explicit expression \eqref{explicit-local-tn}.

\item The reference metric $g_\epsilon$ obtained by the reference triple \eqref{bg-metric-regions} also involves the interpolating term $d (\zeta_\epsilon u_\epsilon)$ on the gluing region. These are small by \eqref{glue-collar}. We also note that there are similar estimates for the metrics near $q_j$, and we omit the derivation and refer to \cite{FoscoloK3}. Putting everything together, we obtain the rough estimate \eqref{geps2cyl}.

\end{itemize}

For later use (case 2b in Lemma \ref{lem:key-estimate}), we also note that combining \eqref{glue-collar} \eqref{bulk2cyl}, \eqref{g-tn-infty} and keeping the worst power of $\epsilon$ gives
\begin{equation} \label{long-region-cyl}
\begin{aligned}
\sup_{\{ \epsilon R^{(1+\tau)/2} \le \rho \le \rho_0 \}}
    |g_\epsilon - g_\epsilon^{\rm cyl}|_{g_\epsilon^{\rm cyl}}
    &\le C \bigl(\epsilon^{\frac{3}{5}} + R^{-(1+\tau)/2}\bigr), \\[6pt]
\sup_{\{ \epsilon R^{(1+\tau)/2} \le \rho \le \rho_0 \}}
    |\nabla_\epsilon^{\rm cyl} g_\epsilon|_{g_\epsilon^{\rm cyl}}
    &\le C \bigl(\epsilon^{\frac{1}{5}} + \epsilon^{-1} R^{-(1+\tau)}\bigr),
  \end{aligned}
\end{equation}
for $0<\tau<1$ and $R \gg 1$.

\subsubsection{Basic estimates of the reference triple}
We note some bounds on quantities involving $\underline{\boldsymbol{\omega}}_\epsilon$. First, we note that for $k \geq 1$,
\begin{equation} \label{derivs-reftriple}
|\underline{\boldsymbol{\omega}}_\epsilon|_{g_\epsilon} \leq C, \quad \rho^k |\nabla^k \underline{\boldsymbol{\omega}}_\epsilon|_{g_\epsilon} \leq C_k.
\end{equation}
Indeed, we have
\[
| \omega_{\epsilon,i} |^2_{g_\epsilon} = 2,  \  \ \nabla \omega_{\epsilon,i} =0 \quad {\rm outside} \ \{ \epsilon^{2/5} \leq \rho \leq 2 \epsilon^{2/5} \},
\]
since it is a hyperk\"ahler triple there. Next,
\[
| \omega_{\epsilon,i} |_{g_\epsilon} \leq C \bigg( |\omega^{\rm tn}_{\epsilon,i}|_{g_\epsilon^{\rm cyl}} + |\nabla \zeta_\epsilon|_{g^{\rm cyl}_\epsilon} |u_{\epsilon,i}|_{g^{\rm cyl}_\epsilon} + |\zeta_\epsilon|_{g^{\rm cyl}_\epsilon} |\nabla u_{\epsilon,i}|_{g^{\rm cyl}_\epsilon} \bigg)
\]
over $\{ \epsilon^{2/5} \leq \rho \leq 2 \epsilon^{2/5} \}$ near $p_i$ after converting to $g_\epsilon^{\rm cyl}$ by \eqref{geps2cyl}. Applying bounds on $u$ \eqref{glue-collar} and estimating the local model \eqref{explicit-local-tn} directly, we can obtain \eqref{derivs-reftriple} for $k=0$. For higher $k \geq 1$, we use the following general bound for conversion between connections on $g_\epsilon$ and $g_\epsilon^{\rm cyl}$: for any tensor $T$, there holds
\begin{equation} \label{nabla2nablacyl}
|\nabla^{g_\epsilon} T|_{g_\epsilon} \leq C  \bigg( |\nabla^{\rm cyl}_\epsilon T |_{g^{\rm cyl}_\epsilon} +  |\nabla^{\rm cyl}_\epsilon g_{\epsilon} |_{g^{\rm cyl}_\epsilon} | T|_{g_\epsilon^{\rm cyl}} \bigg).
\end{equation}
Applying this to $\omega_{\epsilon,i}$ gives a bound of the form
\[
\rho |\nabla \omega_{\epsilon,i}|_{g_\epsilon} \leq C \bigg( \rho |\partial h^{\rm tn}_{\epsilon} | + \rho \sum_{k=0}^2 |\nabla^k \zeta_\epsilon|_{g_\epsilon^{\rm cyl}}  |\nabla^{2-k} u_{\epsilon,i}|_{g_\epsilon^{\rm cyl}}  +  |\nabla^{\rm cyl}_\epsilon g_{\epsilon} |_{g^{\rm cyl}_\epsilon}  \bigg),
\]
on the collar region. Applying \eqref{glue-collar}, \eqref{geps2cyl}, and \eqref{bdd-htn} proves \eqref{derivs-reftriple} for $k=1$.
We omit the proof of higher $k \geq 2$. We note that we have also omitted the analysis near $q_j$ as we will not need this; we refer to \cite{FoscoloK3}. By a similar argument, we can also estimate the failure of orthogonality
\begin{equation} \label{reftriple-ortho}
|\langle \omega_{\epsilon,i}, \omega_{\epsilon,j} \rangle_{g_\epsilon} | \leq C \epsilon^2 \epsilon^{-\frac{1}{5}}, \quad |\omega_{\epsilon,i}|^2_{g_\epsilon} \geq 2- C  \epsilon^2 \epsilon^{-\frac{1}{5}}.
\end{equation}
Let us estimate the volume of $(M_\epsilon,g_\epsilon)$. By \eqref{geps2cyl}, we have
\[
\int_{ \{ \rho \geq \epsilon R_0 \}} d {\rm vol}_{g_\epsilon} \leq C \epsilon,
\]
since the geometry is equivalent to a cylinder on this region. On the other hand, near $p_i$ the metric is $g_\epsilon = \epsilon^2 S_{\epsilon^{-1}}^* g^{\rm tn}$, so pulling back integrals gives
\[
\int_{ \{ \rho \leq \epsilon R_0 \}} d {\rm vol}_{g_\epsilon} = \int_{ \{ \rho \leq R_0 \} \cap N_i } d {\rm vol}_{\epsilon^2 g^{\rm tn}} \leq C \epsilon^4
\]
and the estimate near $q_j$ is analogous. Hence
\begin{equation} \label{vol-mfd}
\int_{M_\epsilon} d {\rm vol}_{g_\epsilon} \leq C \epsilon.
  \end{equation}
Finally, we note that estimate \eqref{glue-collar} implies the following error estimate over the interpolation region near $p_i$
\begin{equation} \label{collar-error}
\sup_{\{ \epsilon^{2/5} \leq \rho \leq 2 \epsilon^{2/5} \}} |\nabla^k (\underline{\boldsymbol{\omega}}_\epsilon - \underline{\boldsymbol{\omega}}^{\rm tn}_\epsilon) |_{g_\epsilon} \leq C \epsilon^{2- \frac{1}{5}} \epsilon^{-\frac{2}{5}k}.
\end{equation}
after converting $\nabla$ to $\nabla^{\rm cyl}$ as in \eqref{nabla2nablacyl}. Estimate \eqref{collar-error} will be used in the gluing construction, since our approximate solutions will turn out to be true solutions with respect to $\underline{\boldsymbol{\omega}}^{\rm tn}$ and we will need control of the error terms on this interpolation region.

\subsection{Genuine metric} \label{sec:genuine}
We have a compact 4-manifold $M_\epsilon$ with a reference triple $\underline{\boldsymbol{\omega}}_\epsilon$. The triple is deformed by
\begin{equation} \label{genuine-ansatz}
\underline{\boldsymbol{\tilde{\omega}}}_\epsilon = \underline{\boldsymbol{\omega}}_\epsilon + d \underline{\boldsymbol{a}}_\epsilon + \underline{\boldsymbol{\zeta}}_\epsilon
\end{equation}
where $\underline{\boldsymbol{a}}_\epsilon \in \Omega^1(M_\epsilon)^{\oplus 3}$ and
\[
\underline{\boldsymbol{\zeta}}_\epsilon = B \cdot \underline{\boldsymbol{\omega}}_\epsilon,
\]
with $B$ a constant $3 \times 3$ matrix depending on $\epsilon$. In other words, $\underline{\boldsymbol{\zeta}}_\epsilon \in \mathcal{H}^+ \otimes \mathbb{R}^3$ where $\mathcal{H}^+$ is the space of self-dual harmonic forms with respect to the reference $g_\epsilon$, and since $\underline{\boldsymbol{\omega}}_\epsilon$ is hypersymplectic then ${\rm Span}(\omega_1,\omega_2,\omega_3)_x = \Lambda^+ T^*_x M$ \cite{Donaldson-2forms, FineYao, FoscoloK3}. We will sometimes use the component notation
\[
\underline{\boldsymbol{\zeta}}_\epsilon=(\zeta_{\epsilon,1}, \zeta_{\epsilon,2}, \zeta_{\epsilon,3}), \quad \zeta_{\epsilon,k} =  B^i{}_k \omega_{\epsilon,i}.
\]
Foscolo's theorem is the following:

\begin{thm} [Theorem 6.17 in \cite{FoscoloK3}] For all $\epsilon>0$ sufficiently small, there exists $(\underline{\boldsymbol{a}}_\epsilon, \underline{\boldsymbol{\zeta}}_\epsilon)$ such that
\begin{equation} \label{genuine-eqn}
{1 \over 2} \tilde{\omega}_{\epsilon,a} \wedge \tilde{\omega}_{\epsilon,b} = \delta_{ab} \, d {\rm vol}_{\tilde{g}_{\underline{\boldsymbol{\omega}}}},
\end{equation}
together with an estimate: for all $\delta \in (- \frac{1}{2},0)$, there exists constants $C_k>1$ uniform in $\epsilon$ such that
\begin{equation} \label{genuine-est00}
\|\underline{\boldsymbol{a}}_\epsilon \|_{C^k_\delta} \leq C_k \epsilon^{\frac{11-2\delta}{5}}, \quad \| \underline{\boldsymbol{\zeta}}_\epsilon \|_{L^2} \leq C \epsilon^{\frac{11-2\delta}{5}}.
\end{equation}
for all $k \geq 0$.
\end{thm}
Therefore $\underline{\boldsymbol{\tilde{\omega}}}_\epsilon$ is a hyperk\"ahler triple on $M_\epsilon$ which is well-approximated by the reference triple $\underline{\boldsymbol{\omega}}_\epsilon$. Here the weighted norms are with respect to the reference metric $g_\epsilon$, so that 
  \[
\| T \|_{C^k_\delta} = \sum_{i=0}^k \sup_{M_\epsilon} \rho^i \rho^{-\delta} |\nabla^i T|_{g_\epsilon},
  \]
and e.g.
\[
\rho |d \underline{\boldsymbol{a}}_\epsilon|_{g_\epsilon} \leq C \rho^\delta \epsilon^{\frac{11-2\delta}{5}}.
\]
By choosing $\delta<0$ close to zero and using $\rho \geq \epsilon$ we may demote a $\frac{1}{5}$ to $\frac{1}{6}$.
\begin{equation} \label{da-est}
\rho |d \underline{\boldsymbol{a}}_\epsilon|_{g_\epsilon} \leq C \epsilon^{2+\frac{1}{6}}.
\end{equation}
Next, we derive an estimate on the constant matrix $B$.
\begin{equation} \label{B-est}
|B| \leq C \epsilon^{\frac{17-4\delta}{10}}.
  \end{equation}
Indeed, by \eqref{reftriple-ortho}, we have
\[
|\zeta_\epsilon|^2 =\bigg| \sum_i B^i{}_k \omega_{\epsilon,i} \bigg|^2_{g_\epsilon} \geq  \sum_i |B^i{}_k|^2,
\]
and therefore
 \begin{equation}
\epsilon^{1/2} |B|  \leq C \bigg( \int_{M_\epsilon} |\zeta_\epsilon|^2 \, d {\rm vol}_{g_\epsilon} \bigg)^{1/2} \leq C \epsilon^{\frac{11-2\delta}{5}}
 \end{equation}
by ${\rm vol}(M_\epsilon,g_\epsilon) = O(\epsilon)$ \eqref{vol-mfd} and \eqref{genuine-est00}. This proves \eqref{B-est}. Putting together these bounds, we obtain
\begin{equation} \label{genuine-est0}
  \begin{aligned}
    \rho |\underline{\boldsymbol{\tilde{\omega}}}_\epsilon - \underline{\boldsymbol{\omega}}_\epsilon |_{g_\epsilon} &\leq C (\rho^\delta \epsilon^{\frac{11-2\delta}{5}}+ \epsilon^{\frac{11-2\delta}{5}} \epsilon^{-\frac{1}{2}} \rho)\\
 \rho^2 |\nabla(\underline{\boldsymbol{\tilde{\omega}}}_\epsilon - \underline{\boldsymbol{\omega}}_\epsilon) |_{g_\epsilon}  &\leq C (\rho^\delta \epsilon^{\frac{11-2\delta}{5}}+ \epsilon^{\frac{11-2\delta}{5}} \epsilon^{-\frac{1}{2}} \rho^2)
  \end{aligned}
    \end{equation}
    Here we used \eqref{derivs-reftriple}, \eqref{genuine-est00} and \eqref{B-est}. By choosing $\delta < 0$ close to zero, we obtain
 \begin{equation} \label{genuine-est2}
\rho |\underline{\boldsymbol{\tilde{\omega}}}_\epsilon - \underline{\boldsymbol{\omega}}_\epsilon |_{g_\epsilon}  \leq C \epsilon^{2+\frac{1}{6}} (1+ \epsilon^{-\frac{1}{2}} \rho).
\end{equation}
It also follows from $\rho \geq \epsilon$ and the discussion above that
 \begin{equation} \label{genuine-est3}
\rho |\nabla (\underline{\boldsymbol{\tilde{\omega}}}_\epsilon - \underline{\boldsymbol{\omega}}_\epsilon) |_{g_\epsilon} +\rho^2 |\nabla^2 (\underline{\boldsymbol{\tilde{\omega}}}_\epsilon - \underline{\boldsymbol{\omega}}_\epsilon) |_{g_\epsilon}   \leq \epsilon^{1+\frac{1}{6}}.
\end{equation}
where we also include second order derivatives, whose estimate follows in a similar way, for later use (\S \ref{sec:est-errors}).

\subsection{Error estimates for the reference metric}
We now give the details on the bound \eqref{glue-collar}, which concerns the reference metric $\underline{\boldsymbol{\omega}}_\epsilon$ and estimates the error term interpolating between the local model and the bulk model. These estimates can also be found in \cite{FoscoloK3}.

  \subsubsection{Difference between global geometry and single charge} \label{sec:global2single}
We start by matching $\underline{\boldsymbol{\omega}}^{\rm bulk}_\epsilon$ near $p_i$ to a Taub-NUT space with a single charge of weight $k_i$ at the origin. Recall that the triple $\underline{\boldsymbol{\omega}}_\epsilon^{\rm bulk}$ given in \eqref{bulk-metric} is built from the harmonic function $h_\epsilon$, which near $p_i$ is 
\begin{equation} \label{h-pi}
h_\epsilon \overset{\rm loc}{=} h_{p_i} + \epsilon O(\rho^2),
\end{equation}
where
  \[
h_{p_i} = 1+\epsilon \lambda_i + \frac{\epsilon k_i}{2 \rho} + \epsilon \ell_i.
\]
Define the local model triple $\underline{\boldsymbol{\omega}}^{p_i}$ by
\begin{equation}
\omega^{p_i}_\alpha = \epsilon dx_\alpha \wedge \theta^{p_i} +  h_{p_i} dx_\beta \wedge dx_\gamma, \quad \epsilon_{\alpha \beta \gamma}=1,
\end{equation}
where $\theta^{p_i} = d \psi - i \mathcal{A}^{k_i}$ and $\mathcal{A}^{k_i}$ is defined in \eqref{multi-A}, so that
\begin{equation} \label{bulk-A2}
\mathcal{A} = \mathcal{A}^{k_i} + O(\rho^2)
\end{equation}
The difference between the global form and this local model is captured by
\[
\underline{\boldsymbol{\eta}} = \underline{\boldsymbol{\omega}}_\epsilon^{\rm bulk} - \underline{\boldsymbol{\omega}}^{p_i}, \quad d \underline{\boldsymbol{\eta}} = 0, \quad \underline{\boldsymbol{\eta}} \in \Omega^2(\{ 0< \rho \leq \rho_0 \})^{\oplus 3}.
\]
By \eqref{h-pi} and \eqref{bulk-A2}, we have the error estimate
\begin{equation} \label{eta-decay}
 |\underline{\boldsymbol{\eta}} |_{g^{\rm cyl}_\epsilon} \leq C \epsilon \rho^2.
\end{equation}
Here the norm is with respect to the metric $g^{\rm cyl}_\epsilon = g_{\mathbb{R}^3} + \epsilon^2 \theta^2$. Next, we find a potential $\eta = d u$. Write $(\rho, \phi)$ for polar coordinates on $\mathbb{R}^3$, so that local coordinates on the circle bundle restricted to $\{ 0 < \rho \leq \rho_0 \}$ are given by $(\rho,\phi,\psi)$. Denote the angle coordinates by $\varphi = (\phi,\psi)$, and write
\begin{equation} \label{eta-polar}
\eta = \alpha_i \, d \rho \wedge d \varphi^i + \gamma_{ij} \, d \varphi^i \wedge d \varphi^j.
\end{equation}
Since $d \eta =0$ then $\partial_\rho \gamma_{ij} = \partial_i \alpha_j$. We can then define the potential by
\[
u = \bigg[ \int_0^{\rho} \alpha_j(s, \cdot) \, ds \bigg] \, d \varphi^j.
\]
The integral is well-defined since $\alpha_j \rightarrow 0$ as $\rho \rightarrow 0$. This construction gives $u = \epsilon O(\rho^3)$, and so
\begin{equation} \label{short-est}
 \underline{\boldsymbol{\omega}}_\epsilon^{\rm bulk} = \underline{\boldsymbol{\omega}}^{p_i} + d \underline{\boldsymbol{u}}^{p_i} , \quad |\nabla^k \underline{\boldsymbol{u}}^{p_i}|_{g^{\rm cyl}_\epsilon} \leq C \epsilon \rho^{3-k}.
\end{equation}
We will only need this on the collar $\{ R_0 \epsilon \leq \rho \leq \rho_0 \}$.

\subsubsection{Difference between single charge and multi-charge} \label{sec:single2multi}
Next, we relate the local model $(N, g^{\rm tn})$ which has $k_i$ points of charge 1 to the space with a single charge of weight $k_i$ at the origin. Namely, we compare
\[
S_\epsilon^* h_{p_i} = 1+\epsilon \lambda_i + \frac{k_i}{2 |x|} + \epsilon^2 \ell_i,
\]
to
\[
h_{\rm tn} = 1 + \epsilon \lambda_i + \sum_{a=1}^{k_i} \frac{1}{2|x-y_a|} + \epsilon^2 \ell_i,
\]
at large distances. The second order Taylor expansion of $f(x) = |x|^{-1}$ is
\[
f(x+y) = \frac{1}{|x|} - \frac{x \cdot y}{|x|^3} + O \left( \frac{|y|^2}{|x|^3} \right),
\]
for $|x| \gg 1$, and so
\[
\frac{k_i}{|x|} - \sum_{a=1}^{k_i} \frac{1}{|x-y_a|} = - \frac{x}{|x|^3} \cdot \sum_{a=1}^{k_i} y_a + O(|x|^{-3}),
  \]
once $|x| \gg R_0$, where $R_0 \gg 1$ depends on $\{y_a\}$, is large enough. Since the $y_a$ are chosen such that the dipole moment vanishes \eqref{zero-dipole}, we have that
\begin{equation}  \label{zero-dipole2}
|h_{\rm tn} - S_\epsilon^* h_{p_i}| \leq C \rho^{-3}.
\end{equation}
We will also need
\[
\left| k_i \mathcal{A}^{\rm mono}_0 - \sum_{a=1}^{k_i} \mathcal{A}^{\rm mono}_{y_a} \right|_{\mathbb{R}^3} \leq C \rho^{-3}, \quad |x| \gg R_0,
\]
which can be derived in a similar way after rewriting \eqref{A-single} as
\[
\mathcal{A}^{\rm mono}_{y} = \frac{i}{2} \bigg( 1 + \frac{x_3-y_3}{|x-y|} \bigg) \bigg( \frac{-(x_2-y_2) dx_1 + (x_1-y_1)dx_2}{(x_1-y_1)^2+(x_2-y_2)^2} \bigg),
  \]
and expanding $\mathcal{A}_y = \mathcal{A}_0(x) - y^i \partial_i \mathcal{A}_0(x) + O(|y|^2|x|^{-3})$. Define the error by
\begin{equation} \label{multi2single}
\underline{\boldsymbol{\eta}} = S_\epsilon^* \underline{\boldsymbol{\omega}}^{p_i} - \epsilon^2 \underline{\boldsymbol{\omega}}^{\rm tn}, \quad d \underline{\boldsymbol{\eta}} = 0, \quad \underline{\boldsymbol{\eta}} \in \Omega^2 (\{ R_0 \leq \rho < \infty \})^{\oplus 3}.
\end{equation}
Then the $\rho^{-3}$-estimates above on $h$ and $\mathcal{A}$ give
\begin{equation} 
 |\underline{\boldsymbol{\eta}}|_{g^{\rm cyl}} \leq C \epsilon^2 \rho^{-3}.
\end{equation}
We can again write $\eta$ using polar coordinates \eqref{eta-polar} and define a potential $\eta = du$ by radial integration:
\[
u = - \bigg[ \int_\rho^\infty \alpha_j (s, \cdot) ds \bigg] d \varphi^j.
\]
The potential is of order
\begin{equation} \label{out-est}
  | \nabla^k u|_{g^{\rm cyl}} \leq C \epsilon^2 \rho^{-2-k}.
\end{equation}
We now transpose this local model to the collar region on the compact manifold by pullback. Upon setting $u^{\rm out} =  S_{\epsilon^{-1}}^* u$ and pulling back \eqref{multi2single} and \eqref{out-est} by $S_{\epsilon^{-1}}$, we obtain
\begin{equation} \label{outest}
\underline{\boldsymbol{\omega}}^{\rm tn}_\epsilon = \underline{\boldsymbol{\omega}}^{p_i}  + d \underline{\boldsymbol{u}}^{\rm out}, \quad |\nabla^k \underline{\boldsymbol{u}}^{\rm out} |_{g^{\rm cyl}_\epsilon} \leq C \epsilon^3 \rho^{-2-k}
\end{equation}
on the collar $\{ \epsilon R_0 \leq \rho \leq \rho_0 \}$. 

\subsubsection{Adding up the potentials} \label{sec:addup}
We therefore have
\[
\underline{\boldsymbol{\omega}}_\epsilon^{\rm bulk} = \underline{\boldsymbol{\omega}}^{\rm tn}_\epsilon + d (\underline{\boldsymbol{u}}^{p_i} - \underline{\boldsymbol{u}}^{\rm out}),
\]
on $\{ \epsilon R_0 \leq \rho \leq \rho_0 \}$. Let us restrict the estimates \eqref{short-est} and \eqref{outest} to $\{ \epsilon^{2/5} \leq \rho \leq 2 \epsilon^{2/5} \}$. The result is
\[
\sup_{\{ \epsilon^{2/5} \leq \rho \leq 2 \epsilon^{2/5} \}} | \nabla^k (\underline{\boldsymbol{u}}^{p_i} - \underline{\boldsymbol{u}}^{\rm out}) |_{g^{\rm cyl}_\epsilon} \leq C \epsilon^2 \epsilon^{\frac{1}{5}} \epsilon^{-\frac{2}{5} k}.
\]
This completes the proof of \eqref{glue-collar}.

\section{Construction of special Lagrangians}

\subsection{Deformation of special Lagrangians}
We now turn to the construction of special submanifolds in the ambient collapsing $K3$ surface. We can either take the point of view of holomorphic submanifolds or special Lagrangian submanifolds. We choose the special Lagrangian perspective.

Let $f: S^2 \rightarrow M$ be a smooth map parametrizing a 2-dimensional submanifold inside a $K3$ surface $(M,J,g,\omega,\Omega)$. This map can be deformed to a new map $f_a: S^2 \rightarrow M$ for each 1-form $a \in \Omega^1(S^2)$ via
\[
f_a = \exp_f J a^\sharp
\]
where $\exp$ is the exponential map of $g$. When the initial map $f: S^2 \rightarrow M$ satisfies $f^*\omega =0$ then $J a^\sharp$ is a normal vector field and all submanifolds nearby can be parametrized by a map of the form $f_a$ for some $a \in \Omega^1(S^2)$. Thus this ansatz is well-suited to explore nearby submanifolds, and in practice we will use this when the initial map satisfies $f^* \omega$ small rather than exactly zero.

We set up the special Lagrangian equation as looking for zeros of
\[
\mathcal{F} : \Lambda^1(S^2) \rightarrow \Lambda^0(S^2) \oplus \Lambda^2(S^2)
\]
where
\[
\mathcal{F}(a) = - f_a^* \omega - \star f_a^* {\rm Im} \, \Omega.
\]
When the initial map $f$ parametrizes a special Lagrangian submanifold, equation $\mathcal{F}(a)=0$ and its linearization play a central role in the study of the local moduli space of solutions; standard references include e.g. \cite{Marshall} or \cite{JoyceSurvey}. Here we will look for solutions to $\mathcal{F}(a) = 0$ from an initial map $f$ which is an approximate solution.


\subsection{Approximate solution} \label{sec:approxsoln}
\subsubsection{The sphere from bubble to bubble} \label{sec:bubble2bubble}
The goal here is to start from a line $L_0$ on a 3-dimensional affine base $B^3 = \mathbb{T}^3/ \mathbb{Z}_2$, and lift it to special Lagrangian spheres $L_\epsilon \subseteq M_\epsilon$ collapsing to $L_0$ along a degeneration of $K3$ surfaces
\[
M_\epsilon \rightsquigarrow B^3.
\]
The ambient $K3$ surfaces will be reconstructed by Foscolo's gluing scheme summarized in Section \ref{sec:foscolosummarized}. To start the setup, we let $\{p_i \}_{i=1}^n$ be a collection of points on the base $B^3$, and $\{m_j \}_{j=1}^8$ and $\{k_i \}_{i=1}^n$ a collection of integers satisfying the balancing condition
\begin{equation} \label{banal}
\sum_{j=1}^8 m_j + \sum_{i=1}^n k_i = 16
\end{equation}
as before. Let $L_0 \subseteq B^3$ be an affine line segment between two points $p_i$, say $p_1$ and $p_2$, and suppose that $L_0$ does not intersect the singular points $q_j$. We can lift this curve to $\mathbb{T}^3$ where it appears as a $\tau$-invariant set $\gamma \subseteq \mathbb{T}^3$ consisting of two connected components. 

We can now begin to reconstruct the ambient $K3$ surface. The first step is to build from \eqref{banal} a harmonic function $h: \mathbb{T}^* \rightarrow \mathbb{R}$ with prescribed singularities as in Lemma \ref{harmonic-presc-sing} and a circle bundle $P \overset{\pi}{\rightarrow} \mathbb{T}^*$ with connection $\theta$ solving $d \theta = \star dh$. The pair of line segments $\gamma \subseteq \mathbb{T}^*$ lift to two cylinders
\[
  L_{\rm cyl} = \pi^{-1}(\gamma) \subseteq P.
\]
Passing to the quotient by $\tau$, we obtain a single cylinder on $M_\epsilon^*$ with each end entering the gluing region. The ambient space $M_\epsilon^*$ is compactified by gluing-in an ALF space at each marked point $p_i$ (Section \ref{sec:gluingregion}) and at each involution fixed point $q_j$. To $p_1$ and $p_2$, we will glue-in a particular multi-Taub-NUT local model $N_i$ with $k_i$ distinct points.
\begin{rk}
The gluing relation \eqref{gluing-ann} ensures that once the cylinder enters the gluing region, it remains a cylinder. In the notation introduced in Section \ref{sec:gluingregion}, the cylinder traverses the gluing region from $A_i$ to $A_i^{\rm tn}$ and appears as $L_{\rm cyl} \subseteq A_i^{\rm tn}$ where $\pi(L_{\rm cyl})$ is a straight line segment in $\{ R_0 \epsilon < \rho < \rho_0 \} \subseteq \mathbb{R}^3$ with same direction.
\end{rk}
Next, we specify how to place the points $\{ y_a\}_{a=1}^{k_i}$ in the Taub-NUT local models at $p_1$ and $p_2$. We work on $P \rightarrow \mathbb{T}^*$ prior to taking the quotient, and focus our attention to one component of $\gamma \subseteq \mathbb{T}^*$ connecting $p_1,p_2 \in \mathbb{T}$. We may choose a coordinate chart on $\mathbb{T}$ containing this component of $\gamma \subseteq \mathbb{T}^*$ and rotate coordinates on the 3-torus such that $v=e_1$ is the direction of the line segment in this component of $\gamma$ from $p_1$ to $p_2$. We place the points $\{ y_a \}$ in the Taub-NUT local model $N_1$ at $p_1$ according to the following two conditions:

\begin{enumerate}
\item Collinearity condition: the points $\{ y_a \}$ are all along the line passing through the origin with direction $e_1$.
  \item Vanishing dipole condition: $\sum y_a = 0$.
  \end{enumerate}

We label $y_1$ as the point with most positive $x$-coordinate. The collinearity condition is imposed so that all 2-spheres connecting the $y_a$ within the bubble are all special Lagrangian with respect to the same ambient Calabi-Yau structure. The vanishing dipole condition is used in \S \ref{sec:single2multi}.
  
We glue-in these local models to $P \rightarrow \mathbb{T}^*$ as described in \S \ref{sec:gluingregion}. We then glue the Taub-NUT local model $N_2$ at $p_2$ using the same procedure, except that $y_1$ denotes the point with the most negative $x$-coordinate, so that the cylinder will close again at $y_1$.

This gluing operation caps off the cylinder $L_{\rm cyl}$ at both ends; indeed, the cylinder extends into each gluing region and runs into the point labelled $y_1$, collapsing the $S^1$-fibers and closing-up the surface. We denote by $L_\epsilon \subseteq M_\epsilon$ the resulting topological 2-sphere. We refer to Figure \ref{fig:approxsoln} for a schematic.

\begin{rk}
We can drop the collinearity condition if we are only interested in constructing the elongated 2-sphere converging to $L_0$ in Theorem \ref{mainthm} and not the $k-1$ shrinking spheres within the bubble. In that case, there should be a point $y_1$ along the line passing through the origin with direction $e_1$ to close-up the cylinder $L_{\rm cyl}$.
  \end{rk}

\begin{figure} 

\begin{tikzpicture}
  \def\radius{1}

  \draw[dashed] (0,0) circle (\radius);
  \node[below] at (0,-\radius - 0.2) {$N_1$}; 

  \draw[dashed] (4,0) circle (\radius);
  \node[below] at (4,-\radius - 0.2) {$N_2$}; 

  \foreach \x/\y in {
    0/0,
    -0.5/0,
    0.5/0
  } {
    \fill (\x,\y) circle (0.07);
  }

  \foreach \x/\y in {
    3.5/0,
    4.5/0
  } {
    \fill (\x,\y) circle (0.07);
  }

  \draw[thick] (0.57,0) -- (3.43,0);
  \node[below] at (2, - 0.2) {$L$};
  \node[below] at (2, -\radius){$P \overset{\pi}{\rightarrow} \mathbb{T}^*$};
  
\end{tikzpicture}

\caption{Approximate solution traversing a pair of gluing regions.} \label{fig:approxsoln}
\end{figure}
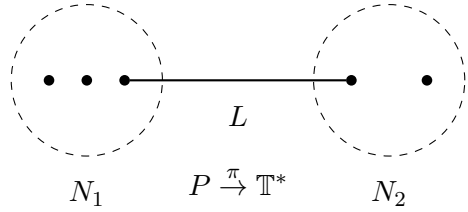

  \subsubsection{Special Lagrangian cylinder}
Next, we look at the special Lagrangian equation on the cylindrical part of $L_\epsilon$. Recall that before taking the quotient by $\tau$ we have a pair of cylinders $L_{\rm cyl} \subseteq P$, and coordinates $(x ,e^{i \psi})$ are chosen such that the tangent space is
\[
T_p L_{\rm cyl} = {\rm span} \, \bigg\{ {\partial \over \partial x_1}, {\partial \over \partial \psi} \bigg\}.
\]
From the hyperk\"ahler structure $\underline{\boldsymbol{\omega}}_\epsilon^{\rm bulk}$ on $P$, we select the Calabi-Yau structure $(\omega,\Omega)$ which appears in this coordinate chart as
\begin{equation} \label{cy-approx}
\omega^{\rm bulk}_\epsilon = \epsilon dx_3 \wedge \theta + h_\epsilon dx_1 \wedge d x_2, \quad {\rm Im} \, \Omega^{\rm bulk}_\epsilon = \epsilon dx_2 \wedge \theta + h_\epsilon dx_3 \wedge dx_1,
\end{equation}
so that the special Lagrangian equations
\begin{equation} \label{bulk-slag}
\omega^{\rm bulk}_\epsilon|_{L_{\rm cyl}} = 0, \quad {\rm Im} \, \Omega^{\rm bulk}_\epsilon|_{L_{\rm cyl}} = 0
\end{equation}
hold over the cylinder $L_{\rm cyl} \subseteq P$.

\subsubsection{Special Lagrangian cigar} Next, we consider the special Lagrangian equation on the part of $L_\epsilon$ inside the gluing region $N_i \overset{\pi}{\rightarrow} U_{i,\epsilon}$. Recall that for example on $N_1$, we have
\[
L_{{\rm cig},\epsilon} \subseteq N_1, \quad  L_{{\rm cig},\epsilon} = \pi^{-1}( \{ \epsilon y_1 + \lambda e_1 : \lambda \geq 0 \} ).
\]
In other words, $L_{{\rm cig},\epsilon}$ is the pre-image of a ray emanating from $\epsilon y_1$ to infinity in direction $e_1$. As noted in Section \ref{sec:multiTAUB}, the $S^1$-fibers collapse as we approach the marked point $\epsilon y_1$, and the submanifold $L_{\rm cig} \subseteq N_1$ takes the form of a cigar as its geometry at infinity will look like $\mathbb{R} \times S^1 \subseteq \mathbb{R}^3 \times S^1$. We select from the hyperk\"ahler structure $\underline{\boldsymbol{\omega}}_\epsilon^{\rm tn}$ on $N_1$ the following Calabi-Yau structure:
\[
\omega^{\rm tn}_\epsilon = \epsilon dx_3 \wedge \theta + h^{\rm tn}_\epsilon dx_1 \wedge d x_2, \quad {\rm Im} \, \Omega^{\rm tn}_\epsilon = \epsilon dx_2 \wedge \theta + h^{\rm tn}_\epsilon dx_3 \wedge dx_1.
\]
Then the submanifold $L_{{\rm cig},\epsilon} \subseteq (N_i,\omega^{\rm tn}_\epsilon,\Omega^{\rm tn}_\epsilon)$ is a special Lagrangian submanifold since its tangent space (except at the tip) is spanned by $\partial_{x_1}$, $\partial_\psi$.


\subsubsection{Summary} In summary, we glue the two caps $L_{\rm cig} \subseteq N_i$ to the cylinder $L_{\rm cyl} \subseteq P$ and obtain a closed submanifold $L_\epsilon \subseteq M_\epsilon$; see Figure \ref{fig:approxsoln}. The gluing occurs in the region
\[
\{ R_0 \epsilon < \rho < \rho_0 \} \cap L_{\rm cyl} \leftrightarrow \{ R_0 \epsilon < \rho < \rho_0 \} \cap L_{\rm cig}.
\]
We select, from the reference triple $\underline{\boldsymbol{\omega}}_\epsilon$ on the closed four-manifold $(M_\epsilon,g_\epsilon)$, the following almost Calabi-Yau structure:
\begin{equation} \label{ref-cy}
\omega_\epsilon = \underline{\boldsymbol{\omega}}_\epsilon \cdot \boldsymbol{e}_3, \quad \Omega_\epsilon = (\underline{\boldsymbol{\omega}}_\epsilon \cdot \boldsymbol{e}_1 )+ i (\underline{\boldsymbol{\omega}}_\epsilon \cdot \boldsymbol{e}_2).
\end{equation}
We also define the tangent bundle endomorphism
\[
J_\epsilon = g_\epsilon^{-1} \omega_\epsilon.
\]
We say ``almost'' here since $J_\epsilon$ does not necessarily square to $-{\rm id}$ on the collar $\{ \epsilon^{2/5} \leq \rho \leq 2 \epsilon^{2/5} \}$. Outside this $\epsilon^{2/5}$-collar, then $J_\epsilon$ is a complex structure and the pair $(\omega,\Omega)$ does define a genuine Calabi-Yau structure; see the definition \eqref{bg-metric-regions} of $ \underline{\boldsymbol{\omega}}_\epsilon$. In fact
\[
(\omega_\epsilon,\Omega_\epsilon)|_{\{ \rho \geq 2 \epsilon^{2/5} \}} = (\omega_\epsilon^{\rm bulk},\Omega_\epsilon^{\rm bulk}), \quad (\omega_\epsilon,\Omega_\epsilon)|_{ \{ \rho \leq \epsilon^{2/5} \}} = (\omega_\epsilon^{\rm tn},\Omega_\epsilon^{\rm tn})
\]
and $L_\epsilon$ solves the special Lagrangian equations with respect to this global reference structure on its cylindrical part and on its cigar part. This reference structure $(\omega_\epsilon,\Omega_\epsilon)$ will next be perturbed to a genuine Calabi-Yau structure $(\tilde{\omega}_\epsilon,\tilde{\Omega}_\epsilon)$, and the task will be to deform $L_\epsilon$ to a genuine special Lagrangian submanifold.

\subsection{Realignment of the hyperk\"ahler structure} \label{sec:realign}

\subsubsection{Selecting the ambient Calabi-Yau structure}
Let 
\begin{equation} \label{genuine-ansatz2}
\underline{\boldsymbol{\tilde{\omega}}}_\epsilon = \underline{\boldsymbol{\omega}}_\epsilon + d \underline{\boldsymbol{a}}_\epsilon + B \cdot \underline{\boldsymbol{\omega}}_\epsilon
\end{equation}
be the genuine hyperk\"ahler triple on the $K3$ surface $M_\epsilon$ as described in \eqref{genuine-ansatz}. From the triple $\underline{\boldsymbol{\tilde{\omega}}}_\epsilon$, we must select a Calabi-Yau structure $(\tilde{\omega}_\epsilon,\tilde{\Omega}_\epsilon)$ about which to solve the special Lagrangian equation; this pair must satisfy
\begin{equation} \label{int-cond}
\int_{L_\epsilon} \tilde{\omega}_\epsilon = 0, \quad  {\rm Im} \, \int_{L_\epsilon} \tilde{\Omega}_\epsilon = 0,
\end{equation}
as this is the necessary topological condition to start deforming the approximate special Lagrangian submanifold $L_\epsilon$ to a genuine special Lagrangian submanifold $\tilde{L}_\epsilon$. 

To start, let $L_\epsilon \subseteq M_\epsilon$ be the approximate solution from Section \ref{sec:approxsoln} traversing two gluing regions. The direction selected in \eqref{ref-cy} from the reference structure $\underline{\boldsymbol{\omega}}_\epsilon$ satisfies
\[
\underline{\boldsymbol{\omega}}_\epsilon = (\omega_{\epsilon,1},\omega_{\epsilon,2},\omega_{\epsilon,3}) = ( {\rm Re} \, \Omega_\epsilon, {\rm Im} \, \Omega_\epsilon, \omega_\epsilon)
\]
and
\[
\int_{L_\epsilon} {\rm Re} \, \Omega_\epsilon = {\rm vol}(L_\epsilon), \quad  \int_{L_\epsilon} {\rm Im} \, \Omega_\epsilon = 0, \quad \int_{L_\epsilon} \omega_\epsilon =0, 
\]
since $L_\epsilon$ is special Lagrangian with respect to $(\omega_\epsilon^{\rm bulk},\Omega_\epsilon^{\rm bulk})$ and $ (\omega_\epsilon^{\rm tn},\Omega_\epsilon^{\rm tn})$, and the error term over the gluing region in \eqref{bg-metric-regions} is exact. Then the genuine hyperk\"ahler triple $\underline{\boldsymbol{\tilde{\omega}}}_\epsilon$ given in \eqref{genuine-ansatz2} integrates to
\[
\int_{L_\epsilon} \underline{\boldsymbol{\tilde{\omega}}}_\epsilon = \bigg[ {\rm vol}(L_\epsilon) \bigg] (1, 0,0)  + \bigg[ {\rm vol}(L_\epsilon) \bigg]  (B^1{}_1, B^1{}_2, B^1{}_3) .
\]
We look for a unit vector $\boldsymbol{v}$ such that
\[
\int_{L_\epsilon} \underline{\boldsymbol{\tilde{\omega}}}_\epsilon \cdot \boldsymbol{v} =0.
\]
We will explain in a moment why this gives a solution to \eqref{int-cond}, but for now we solve the equation
\begin{align}
&   (1+B^1{}_1) v_1 + B^1{}_2 v_2 + B^1{}_3 v_3 = 0 \\
  & \ v_1^2 +v_2^2+v_3^2 = 1
\end{align}
for given $B$ small; recall from \eqref{B-est} that $|B| \leq C \epsilon^{\frac{17-4\delta}{10}}$. This is to find $\boldsymbol{v} \in S^2$ such that
\begin{equation} \label{integral-cond}
\langle \boldsymbol{v}, \begin{bmatrix} 1+ B^1{}_1 \\ B^1{}_2 \\ B^1{}_3 \end{bmatrix} \rangle = 0.
\end{equation}
There is an $S^1$-family of solutions $\boldsymbol{v}$. Choose a solution $\boldsymbol{v}$ nearby $\boldsymbol{e}_3$. Choose another solution $\boldsymbol{v}^\perp$ to \eqref{integral-cond} which is perpendicular to $\boldsymbol{v}$ and nearby $\boldsymbol{e}_2$. Complete this set to an orthogonal basis $\{ \boldsymbol{v}, \boldsymbol{v}^\perp, \boldsymbol{u} \}$ of unit vectors in $\mathbb{R}^3$ with $\boldsymbol{u}$ nearby $\boldsymbol{e}_1$.
\begin{equation} \label{yoyo}
| (\boldsymbol{u}, \boldsymbol{v}^\perp,\boldsymbol{v}) -  (\boldsymbol{e}_1, \boldsymbol{e}_2,\boldsymbol{e}_3) | \leq C |B| \leq C \epsilon^{\frac{17-4\delta}{10}}. 
\end{equation}
We can then take
\begin{equation} \label{genuine-cy}
\tilde{\omega}_\epsilon = \underline{\boldsymbol{\tilde{\omega}}}_\epsilon \cdot \boldsymbol{v}, \quad \tilde{\Omega}_\epsilon = (\underline{\boldsymbol{\tilde{\omega}}}_\epsilon \cdot \boldsymbol{u} )+ i (\underline{\boldsymbol{\tilde{\omega}}}_\epsilon \cdot \boldsymbol{v}^\perp),
\end{equation}
to solve \eqref{int-cond}.

\subsubsection{Ambient error estimates}
With this choice, the approximate Calabi-Yau structure $(\omega_\epsilon,\Omega_\epsilon)$ and the genuine Calabi-Yau structure $(\tilde{\omega}_\epsilon,\tilde{\Omega}_\epsilon)$ differ by
\[
|\tilde{\omega}_\epsilon - \omega_\epsilon|_{g_\epsilon} \leq |\underline{\boldsymbol{\tilde{\omega}}}_\epsilon \cdot \boldsymbol{v} -\underline{\boldsymbol{\omega}}_\epsilon\cdot \boldsymbol{v} |_{g_\epsilon} +  |\underline{\boldsymbol{\omega}}_\epsilon \cdot \boldsymbol{v} -\underline{\boldsymbol{\omega}}_\epsilon\cdot \boldsymbol{e}_3 |_{g_\epsilon}
\]
and so, by \eqref{genuine-est3} and \eqref{yoyo}, after keeping the worst power of $\epsilon$ we derive
\begin{equation} \label{tildeCY-nontilde2}
|\tilde{\omega}_\epsilon - \omega_\epsilon |_{g_\epsilon}  + \rho |\nabla(\tilde{\omega}_\epsilon - \omega_\epsilon) |_{g_\epsilon} + \rho^2 |\nabla^2(\tilde{\omega}_\epsilon - \omega_\epsilon) |_{g_\epsilon} \leq C \epsilon^{1+\frac{1}{6}}.
\end{equation}
Similar estimates hold for $\Omega$, since $\boldsymbol{u}$ is near $\boldsymbol{e}_1$ and $\boldsymbol{v}^\perp$ is near $\boldsymbol{e}_2$.

\subsubsection{Error estimates along $L_\epsilon$}
When restricted to $L_\epsilon$, the symplectic form $\tilde{\omega}_\epsilon$ becomes
\[
\tilde{\omega}_\epsilon|_{L_\epsilon} = (\underline{\boldsymbol{\tilde{\omega}}}_\epsilon \cdot \boldsymbol{v})|_{L_\epsilon} = (d \underline{\boldsymbol{a}}_\epsilon  \cdot \boldsymbol{v})|_{L_\epsilon} + (\omega_{\epsilon,k} v^k + B^i{}_k \omega_{\epsilon,i} v^k)|_{L_\epsilon}.
  \]
  We start from the bulk region where $L_\epsilon$ solves the special Lagrangian equations with respect to $(\omega_\epsilon,\Omega_\epsilon)$ by \eqref{bulk-slag}.
  \[
\underline{\boldsymbol{\omega}}_\epsilon|_{L_\epsilon}  = (d {\rm vol},0,0), \quad {\rm on } \, \{ \rho \geq 2 \epsilon^{2/5} \}.
  \]
Therefore on the bulk region there holds
\[
\tilde{\omega}_\epsilon|_{L_\epsilon} = (d \underline{\boldsymbol{a}}_\epsilon  \cdot \boldsymbol{v})|_{L_\epsilon} + ( v^1 + B^1{}_k v^k) d {\rm vol}_{L_\epsilon} =  (d \underline{\boldsymbol{a}}_\epsilon  \cdot \boldsymbol{v})|_{L_\epsilon} 
  \]  
  by \eqref{integral-cond}. Hence
  \begin{equation} \label{tildeCY-nontilde-bulk}
\sup_{ \{ \rho \geq 2 \epsilon^{2/5} \}} \rho |(\tilde{\omega}_\epsilon-\omega_\epsilon)|_{L_\epsilon}|_{g_\epsilon} \leq \rho | d \underline{\boldsymbol{a}}_\epsilon |_{g_\epsilon} \leq C  \epsilon^{2+\frac{1}{6}} 
  \end{equation}
by \eqref{da-est} and the fact that $\omega_\epsilon|_{L_\epsilon}=0$ on this region. Next, applying \eqref{genuine-est2} gives
\[
  \sup_{ \{ \rho \leq 2 \epsilon^{2/5} \}} \rho |(\tilde{\omega}_\epsilon-\omega_\epsilon)|_{L_\epsilon}|_{g_\epsilon} \leq C  \epsilon^{2+\frac{1}{6}} (1+ \epsilon^{-\frac{1}{2}} \epsilon^{\frac{2}{5}}) = C\epsilon^{2+ \frac{1}{15}}.
\]
We can also estimate the derivative of $\tilde{\omega}_\epsilon-\omega_\epsilon$ in a similar way, and after combining both regions the result is
\begin{equation} \label{tildeCY-nontilde}
\rho \left|(\tilde{\omega}_\epsilon-\omega_\epsilon)|_{L_\epsilon} \right|_{g_\epsilon}  + \rho^2 \left| \nabla^{g_\epsilon|_{L_\epsilon}} (\tilde{\omega}_\epsilon-\omega_\epsilon)|_{L_\epsilon} \right|_{g_\epsilon}  \leq C\epsilon^{2+ \frac{1}{15}}.
\end{equation}
When bounding the derivative, we applied the submanifold estimate
\[
\left| \nabla^{g_\epsilon|_{L_\epsilon}}  (\tilde{\omega}_\epsilon-\omega_\epsilon)|_{L_\epsilon} \right|_{g_\epsilon} \leq C | \nabla^{g_\epsilon} (\tilde{\omega}_\epsilon-\omega_\epsilon) |_{g_\epsilon} + C |A|_{g_\epsilon} | (\tilde{\omega}_\epsilon-\omega_\epsilon)|_{L_\epsilon}|_{g_\epsilon},
\]
together with \eqref{genuine-est0} and the bounds \eqref{secondfundform} on the second fundamental form $A$ which are to be derived in the next section.

\subsection{Second fundamental form}
We will need estimates on the second fundamental form of $L_\epsilon \subseteq (M_\epsilon ,g_\epsilon)$.
\[
A(X,Y) = (\nabla_X Y)^\perp, \quad A \in \Gamma(L,T^*L^{\otimes 2} \otimes NL).
\]
We obtain estimates on the cylindrical region $L_{\rm cyl}$, the cigar region $L_{\rm cig}$ and the overlap gluing region.

\subsubsection{Cylindrical region} On the cylindrical part $L_{\rm cyl} \subseteq P$, namely where $\rho \geq 2 \epsilon^{2/5}$, the metric on $M_\epsilon$ can locally be written as
\[
g_\epsilon =\sum_{k=1}^3 h_\epsilon dx_k \otimes d x_k + \epsilon^2 h_\epsilon^{-1} \theta \otimes \theta, \quad \theta= d \psi - i \pi^* \mathcal{A}.
\]
Choose local coordinates $(x_1,x_2,x_3,\psi)$ near a point such that $\mathcal{A}(p)=0$ and
\[
 T_pL =  {\rm span} \bigg\{ \frac{\partial}{\partial x_1}, \frac{\partial}{\partial \psi} \bigg\}, \quad N_pL = {\rm span} \bigg\{ \frac{\partial}{\partial x_2}, \frac{\partial}{\partial x_3} \bigg\}.
\]
We compute
\begin{align*}
A_{x_1 x_1}|_p &= \Gamma_{x_1 x_1}^{x_2} \partial_{x_2} + \Gamma_{x_1 x_1}^{x_3} \partial_{x_3} = - \frac{1}{2} \bigg[ \partial_{x_2} \log h_\epsilon \, \partial_{x_2} + \partial_{x_3} \log h_\epsilon \, \partial_{x_3} \bigg] ,\\
A_{\psi \psi}|_p &= \Gamma_{\psi \psi}^{x_2} \partial_{x_2} + \Gamma_{\psi \psi}^{x_3} \partial_{x_3} = \frac{\epsilon^2}{2} h^{-2} \bigg[ \partial_{x_2} \log h_\epsilon \, \partial_{x_2} +  \partial_{x_3} \log h_\epsilon \, \partial_{x_3} \bigg],\\
A_{\psi x_1}|_p & = \Gamma_{\psi x_1}^{x_2} \partial_{x_2} + \Gamma_{\psi x_1}^{x_3} \partial_{x_3} = \frac{ \epsilon^2}{2} h^{-2} \bigg[ -iF_{12} \partial_{x_2} -i F_{13} \partial_{x_3} \bigg],
\end{align*}
with $-iF=\star_{g_T} dh$. The norm can be bounded by
\begin{equation} \label{secondfundcyl}
|A|^2_g \leq C h^{-3} |\partial h|^2 .
\end{equation}
Applying estimates on $h$ \eqref{h-estimates} implies
\[
|A|_{g_\epsilon} \leq C \epsilon^{1/5} \quad \text{on }  \{ \rho \geq 2 \epsilon^{2/5} \}.
\]
\subsubsection{Cigar region} This is the region where $g_\epsilon = g_\epsilon^{\rm tn}$ is given by rescalings of the Taub-NUT metric. On $\{ \rho \leq \epsilon R_0 \}$, then $\epsilon^{-2} g_\epsilon = S_{\epsilon^{-1}}^* g^{\rm tn}$. We obtain
\[
\sup_{ \{ \rho \leq R_0 \}} |A_{\rm cig}|_{g^{\rm tn}} \leq C \implies \sup_{ \{ \rho \leq \epsilon R_0 \}} \rho |A|_{g_\epsilon} \leq C
\]
by pulling back a compactness estimate on the bounded region on the cigar. Next, on $\{ \epsilon R_0 \leq \rho \leq \epsilon^{2/5} \}$, the same calculation as \eqref{secondfundcyl} leads to
\[
|A|^2_{g_\epsilon} \leq C (h_\epsilon^{\rm tn})^{-3} |\partial h_\epsilon^{\rm tn}|^2.
\]
Applying \eqref{bdd-htn}, we obtain
\[
|A|_{g_\epsilon} \leq C \epsilon \rho^{-2}, \quad \text{on }  \{ \epsilon R_0 \leq \rho \leq \epsilon^{2/5} \}.
\]

\subsubsection{Overlap region} It remains to control the region $\{ \epsilon^{2/5} \leq \rho \leq 2 \epsilon^{2/5} \}$. The error terms in the metric coming from $d( \zeta_\epsilon u_\epsilon)$ are so small by \eqref{glue-collar} that the previous bound with respect to the $g^{\rm tn}_\epsilon$ metric still holds:
\[
|A|_{g_\epsilon} \leq C \epsilon \rho^{-2}, \quad \text{on }  \{ \epsilon^{2/5} \leq \rho \leq 2 \epsilon^{2/5} \}.
\]

\subsubsection{Summary}
Altogether
\begin{equation}\label{secondfundform}
|A|_{g_\epsilon} \leq
\begin{cases}
C \epsilon^{1/5}, & \text{if } \rho \geq 2 \epsilon^{2/5},\\
C \epsilon \rho^{-2},     & \text{if } \rho \leq 2 \epsilon^{2/5}.
\end{cases}
\end{equation}
Here we used $C^{-1} \leq (\epsilon \rho^{-1})$ on $\{ \rho \leq \epsilon R_0 \}$ to unify the statement of the estimate on the cigar region with the overlap region. As an application, we can restrict the cylindrical estimates \eqref{bulk2cyl} to the submanifold $L_\epsilon$. Let $g_{L_\epsilon}$ and $g^{\rm cyl}_{L_\epsilon}$ be the restrictions of $g_\epsilon$ and $g^{\rm cyl}_\epsilon$ to $L_\epsilon$. Using
\[
 |\nabla^{\rm cyl}_{L_\epsilon} g_{L_\epsilon} |_{g^{\rm cyl}_{L_\epsilon}} \leq  |\nabla^{\rm cyl}_\epsilon g_\epsilon |_{g^{\rm cyl}_\epsilon}  + C |g_\epsilon|_{g^{\rm cyl}_\epsilon} |A|_{g^{\rm cyl}_\epsilon},
\]
we obtain
\begin{equation} \label{bulk2cyl-L}
\begin{aligned} 
\sup_{\{ \rho \geq 2 \epsilon^{2/5} \}} \left| g_{L_\epsilon} - g^{\rm cyl}_{L_\epsilon} \right|_{g_{L_\epsilon}^{\rm cyl}} &\leq C \epsilon^{\frac{3}{5}}, \\
  \sup_{\{ \rho \geq 2 \epsilon^{2/5} \}} \left| \nabla_{g_{L_\epsilon}^{\rm cyl}} \, g_{L_\epsilon} \right|_{g_{L_\epsilon}^{\rm cyl}} &\leq C \epsilon^{\frac{1}{5}}. 
\end{aligned}
\end{equation}
Similarly, estimate \eqref{geps2cyl} restricted to $L_\epsilon$ becomes
\begin{equation} \label{geps2cyl-L}
  \begin{aligned}
\sup_{\{ \rho \geq \epsilon R_0 \}}  \left| g_{L_\epsilon} - g_{L_\epsilon}^{\rm cyl} \right|_{g_{L_\epsilon}^{\rm cyl}} &\leq \frac{1}{2},\\
  \sup_{\{ \rho \geq \epsilon R_0 \}}   \rho \left| \nabla_{g_{L_\epsilon}^{\rm cyl}} \, g_{L_\epsilon}\right|_{g_{L_\epsilon}^{\rm cyl}} &\leq C.
    \end{aligned}
\end{equation}
Finally, we record an estimate on $A$ akin to \eqref{long-region-cyl}. For $0<\tau<1$, there holds
\begin{equation} \label{long-region-A}
\sup_{\{ \epsilon R^{(1+\tau)/2} \le \rho \le \rho_0 \}}
    |A|_{g_\epsilon}
    \le C \bigl(\epsilon^{\frac{1}{5}} + \epsilon^{-1} R^{-(1+\tau)}\bigr).
\end{equation}

\subsection{Fixed point theorem} \label{sec:fixedpoint}
Let $f_0: S^2 \rightarrow M_\epsilon$ be a smooth map which parametrizes the submanifold $L_\epsilon \subseteq (M_\epsilon, g_\epsilon)$ serving as approximate solution. Given any $a \in \Omega^1(S^2)$, deform $f$ by
\[
  f_a: S^2 \rightarrow M_\epsilon, \quad f_a = \exp_{f_0} \xi
\]
where $\xi = J_\epsilon a^\sharp$ and the index of $a$ is raised with respect to the induced metric on $L_\epsilon$ from $(M_\epsilon, g_\epsilon)$. Let
\begin{equation} \label{F-defn}
\mathcal{F}(a) = - f_a^* \tilde{\omega}_\epsilon - \star f_a^* {\rm Im} \, \tilde{\Omega}_\epsilon,
\end{equation}
so that a special Lagrangian deformation of $L_\epsilon$ solves $\mathcal{F}(a)=0$. The Hodge star on $S^2$ is with respect to the metric induced on $L_\epsilon \subseteq (M_\epsilon,g_\epsilon)$, i.e. $f_0^* g_\epsilon$. First, we notice that the image of $\mathcal{F}$ is in the image of $d+d^\dagger$, where $d^\dagger$ is with respect to the metric induced on $L_\epsilon \subseteq (M_\epsilon,g_\epsilon)$. Indeed,
\begin{equation} \label{F-image}
\mathcal{F}: \Lambda^1(S^2) \rightarrow (d+d^\dagger)\Lambda^1(S^2),
\end{equation}
is equivalent to
\[
\int_{S^2} f_a^* \tilde{\omega}_\epsilon = 0, \quad \int_{S^2} f_a^* {\rm Im} \, \tilde{\Omega}_\epsilon = 0,
\]
and this follows from \eqref{int-cond}, since $f_a$ and $f_0$ are homotopic and hence $[f_a^* \tau] = [f_0^* \tau]$ for any closed $k$-form $\tau$.

Next, we set up the contraction mapping problem. By \eqref{F-image}, we may define the quadratic operator
\begin{equation} \label{Q-defn}
\begin{aligned}
  &\ \mathcal{Q}: \Lambda^1(S^2) \rightarrow  (d+d^\dagger)\Lambda^1(S^2) \\
  &\ \mathcal{Q}(a) = \mathcal{F}(a) - \mathcal{F}(0) - (d+d^\dagger) a 
\end{aligned}
\end{equation}
where $d^\dagger$ is with respect to the metric induced by the reference geometry on $L_\epsilon \subseteq (M_\epsilon, g_\epsilon)$, so that
\[
\mathcal{F}(a) = \mathcal{F}(0) + (d+d^\dagger)a + \mathcal{Q}(a).
\]
We also define
\begin{align*}
  &\ \mathcal{N}: \Lambda^1(S^2) \rightarrow \Lambda^1(S^2) \\
  &\ \mathcal{N}(a) = - (d+d^\dagger)^{-1} (\mathcal{F}(0) + \mathcal{Q}(a)).
\end{align*}
We note that due to $H^1(S^2)=0$, the operation
\[
  (d+d^\dagger)^{-1} : (d+d^\dagger) \Lambda^{1}(S^2) \rightarrow \Lambda^{1}(S^2)
\]
is well-defined. We take the fixed point strategy for constructing solutions to $\mathcal{F}(a)=0$:
\[
\mathcal{N}(a) = a \quad \Rightarrow \quad \mathcal{F}(a) = 0.
\]
To let these operators act on Banach spaces, we introduce weighted $C^k$-norms
\[
\| a \|_{C^k_\beta(L_\epsilon)} = \sum_{i=0}^k \sup_L \rho^{-\beta} \rho^i |\nabla^i a|_{g_\epsilon|_{L_\epsilon}}, \quad a \in \Omega^p(S^2)
\]
and the semi-norm
\[
[ a ]_{C^{0,\alpha}_\beta} = \sup_{ \{ d_\epsilon(x,y)< {\rm inj}_{g_\epsilon} \}} \bigg[ {\rm min} (\rho(x), \rho(y))^{-\beta} \frac{|a(x)-a(y)|_{g_\epsilon}}{d_\epsilon(x,y)^\alpha} \bigg],
\]
where $a(x)-a(y)$ is understood by parallel transport along a geodesic connecting the points. The weighted H\"older norms are then
\[
\| a \|_{C^{k,\alpha}_\beta} = \| a \|_{C^k_\beta} + [ \nabla^k h ]_{C^{0,\alpha}_{\beta-k-\alpha}}.
\]
We take our Banach space to be
\[
\mathcal{B} = \bigg( \Lambda^{1}(S^2), C^{1,\alpha}_\beta \bigg), \quad \mathcal{C} = \bigg( (d+d^\dagger)\Lambda^{1}(S^2), C^{0,\alpha}_{\beta-1} \bigg).
\]
The weight $\beta$ will taken later to be a small negative number.

To prove $\mathcal{N}$ is a contraction mapping, we establish the existence of constants $C >0$, $\kappa>0$ and $\iota \geq \frac{1}{2}+ \frac{1}{50}$ independent of $\epsilon>0$ such that:
\begin{enumerate}
\item There is a degenerate bound for the inverse: $\| (d+d^\dagger)^{-1} \|_{\mathcal{C},\mathcal{B}} \leq C \epsilon^{-\frac{1}{2} - \frac{1}{100}}$.
\item There is a quadratic estimate
  \[
    \| Q(u)-Q(v) \|_{\mathcal{C}} \leq  \epsilon^\iota \| u-v \|_{\mathcal{B}}
  \]
  for all $u,v \in \mathcal{U}$, where
  \[
\mathcal{U} = \{ a \in \mathcal{B}: \| a \|_{\mathcal{B}}  < \epsilon^\kappa \}.
\]
\item The errors in the approximate solution are small: $\| \mathcal{F}(0) \|_{\mathcal{C}} \leq \epsilon^{\kappa+\iota}$.
\end{enumerate}

The triangle inequality
\[
\| \mathcal{N}(u) \|_{\mathcal{B}} \leq \| \mathcal{N}(u) - \mathcal{N}(0) \|_{\mathcal{B}}  + \| \mathcal{N}(0) \|_{\mathcal{B}} 
\]
combined with (1), (2), (3) implies that
\[
\| \mathcal{N}(u) \|_{\mathcal{B}} \leq C \epsilon^{-\frac{1}{2} - \frac{1}{100}} \epsilon^\iota \epsilon^\kappa, \quad u \in \mathcal{U}.
\]
This is where the condition $\iota \geq \frac{1}{2}+ \frac{1}{50}$ enters the argument: it is to absorb the $\epsilon^{-\frac{1}{2} - \frac{1}{100}}$ coming from the inverse of the linearized operator. Thus
\[
\mathcal{N}: \mathcal{U} \rightarrow \mathcal{U}
\]
and similarly we can show that $\mathcal{N}$ is a contraction mapping once $\epsilon$ is small enough. By the Banach fixed point theorem, there exists a unique fixed point $a \in \mathcal{U}$, namely a solution to $\mathcal{N}(a)=a$. Thus have a found a 1-form solving $\mathcal{F}(a)=0$ and deformed $L_\epsilon$ into a special Lagrangian submanifold. More precisely, the fixed point argument produces a $C^{1,\alpha}$ regular submanifold, but this must be a smooth submanifold by a standard argument using elliptic regularity (alternatively, since the submanifold is minimal, it is smooth by Morrey's regularity theorem \cite{Morrey}).

Thus it remains to prove (1), (2) and (3). These will be proved in Section \ref{sec:quad-est}. 

\subsection{Simplified model I: spheres inside a bubble region}
Let us set aside for the moment the 2-sphere $L_\epsilon$ connecting two points $p_i$. In this subsection, we stay inside a Taub-NUT bubble and show that 2-spheres connecting a pair of monopole points in the bubble can be deformed into a special Lagrangian submanifold on the ambient $K3$. This is essentially already known and follows from the local model geometry \cite{LotayOli}, Foscolo's estimates \cite{FoscoloK3}, and a straight-forward perturbation argument. We give two arguments for this perturbation argument.

\begin{figure} 
  \centering
\begin{tikzpicture}[scale=2]
  \def\radius{0.05}

  \fill (0,0) circle (\radius);
  \fill (1,0) circle (\radius);
  \fill (0,1) circle (\radius);

  \node[below left] at (0,0) {\(y_1\)};
  \node[below right] at (1,0) {\(y_2\)};
  \node[above left] at (0,1) {\(y_3\)};

  \draw (0,0) -- (1,0) node[midway, below] {\(S^2\)};
  \draw (0,0) -- (0,1) node[midway, left] {\(S^2\)};
  \draw (1,0) -- (0,1) node[midway, above right] {\(S^2\)};

\end{tikzpicture}
\caption{A chain of 2-spheres in Taub-NUT space $N$.} 
\label{fig-2chainz}
\end{figure}
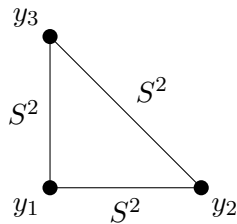

\subsubsection{Minimal surface}

Take a 2-sphere connecting two points where the circle action collapses in multi-Taub-NUT bubble; for this argument we can place the points $y_a$ within the bubble in any configuration such that $\sum y_a = 0$, see e.g. Figure \ref{fig-2chainz}. Inside the bubble, Foscolo's $K3$ metric $\tilde{g}$ becomes after rescaling a small perturbation of the multi-Taub-NUT metric $g^{\rm tn}$. By \cite{Trinca22}, such spheres are the unique area (with respect to $g^{\rm tn}$) minimisers representing that homology class. Therefore this minimal submanifold does not have nullity, and we may apply White's deformation theorem \cite{White} for $\epsilon>0$ small enough to obtain a minimal sphere (with respect to $\tilde{g}$) representing that same homology class. By Webster's formula (see e.g. \cite{FoscoloTrinca}) this minimal sphere must have Gauss map of degree zero and hence be holomorphic for some holomorphic structure.

\subsubsection{McLean's theorem}
The above argument does not keep track of which ambient Calabi-Yau structure makes the 2-sphere special Lagrangian, and for our application the ambient Calabi-Yau structure is fixed to $(\tilde{\omega}_\epsilon,\tilde{\Omega}_\epsilon)$. We revisit the argument and remove the freedom of arbitrary hyperk\"ahler rotation. Recall that the point $p_1 \in M_\epsilon^*$ was replaced by a multi Taub-NUT space $N_\epsilon \overset{\pi}{\rightarrow} U_\epsilon$ with $U = \mathbb{R}^3 \backslash \{ \epsilon y_k \}$, and the points $\{y_k \}$ are all taken to be along the $x$-axis. Choose two neighboring points denoted $\epsilon y_1$ and $\epsilon y_2$ and let $S_\epsilon$ be the 2-sphere lifting the straight line in direction $e_1$ between them. We wish to find a nearby 2-sphere which is special Lagrangian.

In Section \ref{sec:realign} we fixed a Calabi-Yau structure $ (\tilde{\omega}_\epsilon,  \tilde{\Omega}_\epsilon)$ and this pair solves
\[
\int_{S_\epsilon} \tilde{\omega}_\epsilon = 0, \quad \int_{S_\epsilon} {\rm Im} \, \tilde{\Omega}_\epsilon = 0.
\]
In the bubble $\{ \rho \leq R_0 \epsilon \}$, we have the identities $\omega_\epsilon = \epsilon^2 S_{\epsilon^{-1}}^* \underline{\boldsymbol{\omega}}^{\rm tn}  \cdot e_3$, $\rho \equiv R_0 \epsilon$, $g_\epsilon = \epsilon^2 S_{\epsilon^{-1}}^* g^{\rm tn}$. Pullback by $S_\epsilon$ to work on the multi Taub-NUT space $N \overset{\pi}{\rightarrow} U$ with $U = \mathbb{R}^3 \backslash \{ y_k \}$, where estimate \eqref{tildeCY-nontilde2} becomes
\[
| \bar{\omega}_\epsilon - \omega|_{g}  +  |\nabla( \bar{\omega}_\epsilon - \omega) |_{g} +  |\nabla^2( \bar{\omega}_\epsilon - \omega) |_{g} \leq C \epsilon^{1+\frac{1}{6}},
\]
where we write $g=g^{\rm tn}$ and $\omega = \underline{\boldsymbol{\omega}}^{\rm tn} \cdot e_3$ and $\bar{\omega}_\epsilon = \epsilon^{-2} S_\epsilon^* \tilde{\omega}_\epsilon $. Similarly, we write $\Omega = (\underline{\boldsymbol{\omega}}^{\rm tn} \cdot e_1) + i (\underline{\boldsymbol{\omega}}^{\rm tn} \cdot e_2)$ and $\bar{\Omega}_\epsilon = \epsilon^{-2} S_\epsilon^* \tilde{\Omega}_\epsilon$. In sum, we have two nearby Calabi-Yau structures $(\omega,\Omega)$ and $(\bar{\omega}_\epsilon,\bar{\Omega}_\epsilon)$ and
\[
\omega|_{S} = 0, \quad {\rm Im} \, \Omega|_S=0
  \]
where $S \subseteq N$ lifts the straight line from $y_1$ to $y_2$. It follows from McLean's deformation theorem \cite{JoyceSurvey, Marshall} that $S$ can be deformed to a nearby there exists a special Lagrangian 2-sphere with respect to $(\bar{\omega}_\epsilon, \bar{\Omega}_\epsilon)$.

\section{Simplified model II: two-torus in the bulk} \label{sec:2torus}
We will resume in Section \ref{sec:quad-est} the proof of Theorem \ref{mainthm} on elongated spheres joining a pair of bubbles. Before that, in this section, we prove Theorem \ref{mainthm2} on thin tori avoiding the bubble regions. The proof is similar in form but the analysis is simpler compared to the estimates in Section \ref{sec:quad-est}.

Let $B = \mathbb{T}^3 / \mathbb{Z}_2$ be given datum $\mathcal{D}$ as before. Let $L_0 \subseteq B$ be an embedded straight line which avoids the involution fixed points $q_j$ and also avoids the marked points $p_i \in B_{\rm reg}$. In this section, we discuss how $L_0$ can be lifted to a sequence of special Lagrangian 2-tori $\Sigma_\epsilon$ on a family of $K3$ surfaces $M_\epsilon$ collapsing to $B$.

\[
\begin{tikzcd}[column sep=large, row sep=large]
\Sigma_\epsilon  
  \arrow[r, "\epsilon \to 0"] 
  \arrow[d, hook] 
& 
L_0 
  \arrow[d, hook] \\
M_\epsilon 
  \arrow[r, "\epsilon \to 0"] 
& 
B
\end{tikzcd}
\]


\subsection{The approximate solution} Take $\epsilon > 0$ small enough such that we may remove $\epsilon$-neighborhoods around the $q_j, p_i \in B$ without removing any part of $L_0$. Build the circle bundle $P \rightarrow \mathbb{T}^*$ as in Section \ref{sec:foscolosummarized} and lift $L_0$ to two disjoint embedded tori
\[
\tilde{\Sigma}_{\epsilon} = \pi^{-1}(L_0) \subseteq P.
\]
Take the quotient and compactify $P$ to $M_\epsilon$ by Foscolo's construction (see Section \ref{sec:foscolosummarized} for a summary) and obtain a 2-torus $\Sigma_{\epsilon} \subseteq M_\epsilon$. Rotate coordinates such that the affine line points in direction $e_1$, and select the bulk Calabi-Yau structure
\begin{align*}
\omega_\epsilon^{\rm bulk} 
  &= \epsilon\, dx_3 \wedge \theta + h_\epsilon\, dx_1 \wedge dx_2, 
  & \omega_\epsilon^{\rm bulk}\big|_{\Sigma_\epsilon} &= 0, \\
\operatorname{Im}\,\Omega_\epsilon^{\rm bulk} 
  &= \epsilon\, dx_2 \wedge \theta + h_\epsilon\, dx_3 \wedge dx_1, 
  & \operatorname{Im}\,\Omega_\epsilon^{\rm bulk}\big|_{\Sigma_\epsilon} &= 0.
\end{align*}
The local model $(\omega_\epsilon^{\rm bulk}, \Omega_\epsilon^{\rm bulk})$ is only defined on $M_\epsilon^*$. We next select the global compact Calabi-Yau structure $(M_\epsilon, \tilde{\omega}_\epsilon, \tilde{\Omega}_\epsilon)$ and use it to solve the special Lagrangian equation for $\Sigma_\epsilon$.

\subsection{The global Calabi-Yau structure} 
The genuine hyperk\"ahler structure on $M_\epsilon$ is of the form
\[
\underline{\boldsymbol{\tilde{\omega}}}_\epsilon = \underline{\boldsymbol{\omega}}_\epsilon + d \underline{\boldsymbol{a}}_\epsilon + B \cdot \underline{\boldsymbol{\omega}}_\epsilon.
\]
As explained in Section \ref{sec:realign}, it is possible to select a basis of unit vector $\{\boldsymbol{u}, \boldsymbol{v}^\perp, \boldsymbol{v} \}$ nearby $(e_1,e_2,e_3)$ such that
\[
\tilde{\omega}_\epsilon = \underline{\boldsymbol{\tilde{\omega}}}_\epsilon \cdot \boldsymbol{v}, \quad \tilde{\Omega}_\epsilon = \underline{\boldsymbol{\tilde{\omega}}}_\epsilon \cdot \boldsymbol{u} + i (\underline{\boldsymbol{\tilde{\omega}}}_\epsilon \cdot \boldsymbol{v}^\perp)
\]
is a Calabi-Yau structure on $M_\epsilon$ aligned to solve
\begin{equation} \label{torus-alignment}
\int_{\Sigma_\epsilon} \tilde{\omega}_\epsilon = 0, \quad \int_{\Sigma_\epsilon} {\rm Im} \, \tilde{\Omega}_\epsilon =0.
\end{equation}
The embedded torus $\Sigma_\epsilon \subseteq (M_\epsilon, \tilde{\omega}_\epsilon, \tilde{\Omega}_\epsilon)$ is to be deformed to solve the special Lagrangian equations.

\subsection{The fixed point argument} 
Let $f_0: T^2 \rightarrow M_\epsilon$ be a smooth map parametrizing $\Sigma_\epsilon \subseteq M_\epsilon$. We deform $f_0$ by the formula
\[
f_a: T^2 \rightarrow M_\epsilon, \quad f_a = \exp_{f_0} J a^\sharp, \quad a \in \Omega^1(\Sigma_\epsilon)
\]
and look for solutions to $\mathcal{F}(a)=0$ with
\[
\mathcal{F}(a) = - f_a^* \tilde{\omega}_\epsilon - \star f_a^* {\rm Im} \, \tilde{\Omega}_\epsilon.
\]
We setup the domain and codomain as
\[
\mathcal{F}: C^{1,\alpha}(\mathcal{H}^1{}^\perp) \rightarrow C^{0,\alpha}( (d+d^\dagger)\Lambda^1),
\]
where $\mathcal{H}^1{}^\perp$ denotes 1-forms $L^2$-orthogonal to harmonic forms, and H\"older norms and $d^\dagger$ are with respect to $g_\epsilon|_{\Sigma_\epsilon}$, which is
\[
g_\epsilon|_{\Sigma_\epsilon} = h_\epsilon(x) dx \otimes dx + \epsilon^2 h_\epsilon(x)^{-1} \theta \otimes \theta,
\]
so that the harmonic forms are
\[
\mathcal{H}^1 = {\rm span} \{ h_\epsilon(x) dx, \theta \}.
  \]
The image of $\mathcal{F}$ is indeed in $(d+d^\dagger)\Lambda^1$ since
\[
\int_{T^2}  f_a^* \tilde{\omega}_\epsilon =0, \quad \int_{T^2} f_a^* {\rm Im} \, \tilde{\Omega}_\epsilon =0
\]
by \eqref{torus-alignment}. Introduce as before the quadratic operator $\mathcal{Q}$ and contraction operator $\mathcal{N}$.
\[
\mathcal{F}(a) = \mathcal{F}(0) + (d+d^\dagger)a + \mathcal{Q}(a), \quad \mathcal{N}(a) = (d+d^\dagger)^{-1}( \mathcal{F}(0) + \mathcal{Q}(a)).
\]
Here $(d+d^\dagger)^{-1}(\psi)$ is defined to be the unique 1-form $a \in \mathcal{H}^1{}^\perp$ such that $(d+d^\dagger)a = \psi$. The domain and codomain of the contraction operator is
\[
\mathcal{N}: C^{1,\alpha}(\mathcal{H}^1{}^\perp) \rightarrow C^{1,\alpha}(\mathcal{H}^1{}^\perp),
\]
and we aim to solve $\mathcal{N}(a)=a$ by the Banach fixed point theorem. For this, we need to estimate $\| \mathcal{F}(0) \|$, $\| (d+d^\dagger)^{-1} \|$, and $\| \mathcal{Q}(u)-\mathcal{Q}(v) \|$.

\subsection{Estimating the approximate solution} \label{sec:2torusapprox} We estimate
\[
\| \mathcal{F}(0) \|_{C^{0,\alpha}} \leq | (\tilde{\omega}_\epsilon |_{\Sigma_\epsilon}) | +  | \nabla^{\Sigma_\epsilon} (\tilde{\omega}_\epsilon |_{\Sigma_\epsilon}) | + | ({\rm Im} \, \tilde{\Omega}_\epsilon |_{\Sigma_\epsilon}) | +  | \nabla^{\Sigma_\epsilon} {\rm Im} \, (\tilde{\Omega}_\epsilon |_{\Sigma_\epsilon}) |.
\]
Since the reference structure $(\omega_\epsilon, {\rm Im} \, \Omega_\epsilon)$ vanishes on $\Sigma_\epsilon$, this amounts to estimating the error in going from $\tilde{\omega}_\epsilon$ to $\omega_\epsilon$. By \eqref{tildeCY-nontilde}
\[
\| \mathcal{F}(0) \|_{C^{0,\alpha}} \leq \epsilon^2,
\]
and this gives the smallness of the approximate solution. We note that \eqref{tildeCY-nontilde} is only stated for $\tilde{\omega}_\epsilon$ but the same estimate (and proof) holds for $\tilde{\Omega}_\epsilon$, and estimate \eqref{tildeCY-nontilde} is stated for $L_\epsilon$ but also holds for $\Sigma_\epsilon \subseteq \{ \rho \geq 2 \epsilon^{2/5} \}$; in fact only \eqref{tildeCY-nontilde-bulk} is needed as $\Sigma_\epsilon$ does not enter the bubble $\{ \rho \leq 2 \epsilon^{2/5} \}$.

\subsection{Estimating the inverse of the linearized operator} There is a constant $C>1$ such that for all $\epsilon>0$ there holds
\begin{equation} \label{2torusLinverse}
\| a \|_{C^{1,\alpha}(g^{\rm cyl}_\epsilon)} \leq C \| (d+d^\dagger) a \|_{C^{0,\alpha}(g^{\rm cyl}_\epsilon)}, \quad a \in \mathcal{H}^1{}^\perp(T^2)
\end{equation}
where $g^{\rm cyl}_\epsilon = g^{\rm cyl}_\epsilon|_{\Sigma_\epsilon}$. Here we drop the restriction notation from the ambient $M_\epsilon$ to the submanifold $\Sigma_\epsilon$, so that $g^{\rm cyl}_\epsilon = dx^2 + \epsilon^2 \theta^2$ with $\theta = d \psi -i \mathcal{A}_1(x) dx$. Suppose not. Then there exists a sequence $a_i$ such that
\[
\| a_i \|_{C^{1,\alpha}(g^{\rm cyl}_\epsilon)} = 1, \quad \| (d+d^\dagger) a_i \|_{C^{0,\alpha}(g^{\rm cyl}_\epsilon)} \rightarrow 0.
\]
Write
\begin{equation} \label{torus-1forms}
a = a_x( e^{ix}, e^{i\psi} ) d x + \epsilon a_\psi ( e^{ix}, e^{i\psi}) d \psi.
\end{equation}
We use the following lemma on collapsing tori.

\begin{lem} \label{lem:toricollapse}
  Let $T^2 = S^1 \times S^1$ be a torus with coordinates $(e^{i x} ,e^{i \psi})$ equipped with the metric $g^{\rm cyl}_\epsilon = d x \otimes d x + \epsilon^2  \theta \otimes \theta$. Let $a^\epsilon \in \Omega^1(T^2)$ be a sequence of 1-forms such that
  \[
\| a^\epsilon \|_{C^{1,\alpha}(g^{\rm cyl}_\epsilon)} \leq C, \quad \| (d+d^\dagger) a^\epsilon \|_{C^{0}(g^{\rm cyl}_\epsilon)} \rightarrow 0,
\]
as $\epsilon \rightarrow 0$. Then there exists a subsequence such that $(a^\epsilon_x, a^\epsilon_\psi)$, using notation \eqref{torus-1forms}, converges uniformly to
  \[
(a^\infty_x, a^\infty_\psi) = (c_1,c_2)
  \]
  where the $c_i$ are constants.
\end{lem}

We will prove Lemma \ref{lem:toricollapse} later on in Section \ref{sec:cylcollapse}. Using this lemma, we conclude that a subsequence $((a_i)_x,(a_i)_\psi)$ converges uniformly to a pair of constants $(c_1,c_2)$. The condition $a \in \mathcal{H}^1{}^\perp$, which is
\[
\langle a, h_\epsilon dx \rangle_{L^2} = 0, \quad \langle a, \theta \rangle_{L^2} = 0,
\]
is equivalent to
\[
\int_{T^2} (h_\epsilon a_\psi) \, dx d \psi = 0, \quad \int_{T^2} (a_x + i \epsilon \mathcal{A}_1 a_\psi) \, dx d \psi = 0,
\]
and implies $(c_1,c_2)=(0,0)$ since $h_\epsilon \rightarrow 1$ as $\epsilon \rightarrow 0$. On the other hand, the argument in Lemma \ref{lem:elliptic-est}, to be stated and proven below, gives the uniform estimate
\[ 
\| a \|_{C^{1,\alpha}(g^{\rm cyl}_\epsilon)} \leq C \bigg( \| (d+d^\dagger) a \|_{C^{0,\alpha}(g^{\rm cyl}_\epsilon)} + \| a \|_{C^0(g^{\rm cyl}_\epsilon)} \bigg).
\]
To derive this, we unravel the collapsing circle on the universal cover and apply interior Schauder estimates on $S^1 \times \mathbb{R}$. Therefore along our sequence
\[
\frac{1}{2C} \leq \| a_i \|_{C^0(g^{\rm cyl}_\epsilon)} \leq 2 \bigg( |(a_i)_x|^2 + |(a_i)_\psi|^2 \bigg).
\]
This contradicts that a subsequence of $((a_i)_x,(a_i)_\psi)$ converges to zero. This proves \eqref{2torusLinverse}. In fact, \eqref{2torusLinverse} implies
\begin{equation}  \label{2torusLinverse2}
\| a \|_{C^{1,\alpha}(g_\epsilon)} \leq C \| (d+d^\dagger) a \|_{C^{0,\alpha}(g_\epsilon)}, \quad a \in \mathcal{H}^1{}^\perp(T^2)
\end{equation}
once $\epsilon$ is small enough. Here $g_\epsilon= g_\epsilon|_{\Sigma_\epsilon}$. Indeed, we can convert metrics in $d+d^\dagger$ via the general formula
\begin{align} \label{ddag2cyl}
  | (d+d^\dagger_{g^{\rm cyl}_\epsilon} ) a_i |_{g^{\rm cyl}_\epsilon} &\leq  C | (d+d^\dagger_{g_\epsilon}) a_i |_{g_\epsilon}\nonumber\\
  &+ C \bigg( |g_\epsilon-g^{\rm cyl}_\epsilon|_{g^{\rm cyl}_\epsilon} |\nabla a_i |_{g_\epsilon} + |\nabla^{\rm cyl}_\epsilon (g_\epsilon-g^{\rm cyl}_\epsilon)|_{g^{\rm cyl}_\epsilon} | a_i|_{g_\epsilon}\bigg) .
\end{align}
From here, we can use \eqref{bulk2cyl} to convert $g_\epsilon^{\rm cyl}$ to $g_\epsilon$, while noting the appearance of the second fundamental form
\[
 | \nabla^{\rm cyl}_{L_\epsilon} g_{L_\epsilon} |_{g^{\rm cyl}_\epsilon} \leq  |\nabla^{\rm cyl}_{M_\epsilon} g_{M_\epsilon} |_{g^{\rm cyl}_{\epsilon}}  + C |g_\epsilon|_{g^{\rm cyl}_\epsilon} |A|_{g^{\rm cyl}_\epsilon},
\]
which is controlled in a similar way to \eqref{secondfundcyl} so that
\[
|A|_{g^{\rm cyl}_\epsilon} \leq C \epsilon^{\frac{1}{5}}.
\]
We obtain an estimate of the form
\[
 \| (d+d^\dagger_{g^{\rm cyl}_\epsilon} ) a_i \|_{C^{0,\alpha}(g^{\rm cyl}_\epsilon)} \leq C \| (d+d^\dagger_{g_\epsilon}) a_i \|_{C^{0,\alpha}(g_\epsilon)} + C \epsilon^{\frac{1}{5}} \| a_i \|_{C^{1,\alpha}(g_\epsilon)}.
\]
The $C^1$ norms is similarly converted (using e.g. \eqref{nabla2nablacyl}), and the result is
\[
\| a \|_{C^{1,\alpha}(g_\epsilon)} \leq C \| a \|_{C^{1,\alpha}(g^{\rm cyl}_\epsilon)}
\]
for $\epsilon$ small. Applying these conversions between $g^{\rm cyl}_\epsilon$ and $g_\epsilon$ to \eqref{2torusLinverse} proves \eqref{2torusLinverse2} once $\epsilon$ is small enough.

\subsection{Estimating the quadratic operator}
Let $p \in \Sigma_\epsilon \subseteq M_\epsilon$. Choose uniform coordinates on the local universal cover of $M_\epsilon$ near $p$ by unravelling the collapsing circle fiber as in Lemma \ref{lemma-metric-scale}, region (1). On such a chart $B_1 \times \mathbb{R}$, the metric tensor $g_\epsilon$ appears as a matrix $\bar{g}_{ij}$ with uniform bounds compared to the Euclidean metric and bounded partial derivatives. The forms $\tilde{\omega}_\epsilon$ and $\tilde{\Omega}_\epsilon$ are also uniformly bounded in this uniform chart.

It follows that in this chart, $\mathcal{Q}(x,a,\partial a)$ is a smooth function of the jet $(a,\partial a)$ with $\mathcal{Q}(0)=0$ and $D\mathcal{Q}(0)=0$. For any such function, we have the property that for all $u,v \in \Omega^1(\Sigma_\epsilon)$ with $\| u \|_{C^{1,\alpha}} \leq 1$ and $\| v \|_{C^{1,\alpha}} \leq 1$, then
\begin{equation} \label{torus-quad-est}
\| \mathcal{Q}(u) - \mathcal{Q}(v) \|_{C^{0,\alpha}} \leq C (\| u \|_{C^{1,\alpha}}  + \| v \|_{C^{1,\alpha}} ) \| u - v \|_{C^{1,\alpha}}.
\end{equation}
Cover $\Sigma_\epsilon$ by such charts and descend from the covering space to obtain \eqref{torus-quad-est} as a global estimate on $\mathcal{Q}: \Omega^1(\Sigma_\epsilon) \rightarrow \Omega^*(\Sigma_\epsilon)$.

\subsection{Completing the fixed point argument}
Using the bound \eqref{2torusLinverse2} for the inverse of $d+d^\dagger$ and the quadratic estimate \eqref{torus-quad-est}, we estimate
\begin{align*}
  \| \mathcal{N}(u) \|_{C^{1,\alpha}} &\leq \| \mathcal{N}(u) - \mathcal{N}(0)\|_{C^{1,\alpha}} + \| \mathcal{N}(0) \|_{C^{1,\alpha}} \\
                       &\leq  C \| \mathcal{Q}(u) - \mathcal{Q}(0)\|_{C^{0,\alpha}} + C \| \mathcal{F}(0) \|_{C^{0,\alpha}} \\
  &\leq C \bigg( \| u  \|_{C^{1,\alpha}}^2 + \| \mathcal{F}(0) \|_{C^{0,\alpha}}  \bigg).
\end{align*}
By the smallness of the approximate solution \S \ref{sec:2torusapprox}, we conclude that
\[
\mathcal{N} : \{ \| u \|_{C^{1,\alpha}} < \epsilon \} \rightarrow \{ \| u \|_{C^{1,\alpha}} < \epsilon \}
\]
once $\epsilon$ is small enough, and similarly $\| \mathcal{N}(u) - \mathcal{N}(v) \| < \|u-v\|$ on this set. Therefore $\mathcal{N}$ has a fixed point and so $\mathcal{F}(a)=0$ can be solved. This completes the proof of lifting a line in the bulk region to a special Lagrangian 2-torus nearby the explicit submanifold $\Sigma_\epsilon$.

\section{Analysis in weighted spaces} \label{sec:quad-est}

We return to the elongated sphere connecting two bubble regions. In Section \S \ref{sec:fixedpoint}, we reduced the proof of Theorem \ref{mainthm} to three estimates. The outline for this section is:
\begin{itemize}
\item First, we prove the smallness bound for the approximate solution.
  \begin{equation} \label{step0}
 \| \mathcal{F}(0) \|_{C^{0,\alpha}_{\beta-1}}  < \epsilon^{\kappa+\iota}
  \end{equation}
  where $0<\kappa = 1 - \beta + \iota$,  $0< \iota \leq \frac{1}{2} + \frac{1}{50}$, and $\beta < 0$ is small: $|\beta| \ll 1$.
  
\item Second, we prove that the inverse operator of $d+d^\dagger$ admits the following degenerating bound.
\begin{equation} \label{step1}
  \| (d+d^\dagger)^{-1} \|_{\mathcal{C},\mathcal{B}} \leq C \epsilon^{- \frac{1}{2}-\frac{1}{100}}.
\end{equation}
This is where the weights $\beta<0$ in the H\"older norms are used.

\item Third, we prove the quadratic estimate 
\begin{equation} \label{step2}
\| \mathcal{Q}(u) - \mathcal{Q}(v) \|_{C^{0,\alpha}_{\beta-1}} \leq C \epsilon^\iota \| u - v \|_{C^{1,\alpha}_\beta}, \quad u,v \in \mathcal{U}.
\end{equation}
This amounts to controlling the error terms in our construction which differ from the local model.
\end{itemize}

Once \eqref{step0}, \eqref{step1} and \eqref{step2} are established, the proof of the main theorem is complete.

\subsection{Smallness of approximate solution}
In this subsection, we will prove
\begin{equation} \label{small-approx-slag}
  \begin{aligned}
    & \| \mathcal{F}(0) \|_{C^{0,\alpha}_{\beta-1}}  < \epsilon^{\kappa+\iota}, \\
    & 0<\kappa = 1 - \beta + \iota, \quad 0< \iota \leq \frac{1}{2} + \frac{1}{50}, \quad \beta < 0, \quad |\beta| \ll 1.
    \end{aligned}
  \end{equation}
When this estimate is used in Section \S \ref{sec:fixedpoint}, we take $\iota = \frac{1}{2} + \frac{1}{50}$. We introduce $\kappa$ defined this way to be consistent with constraints which will come later during the quadratic estimate in \S \ref{sec:est-errors}. Let us drop $\epsilon$ subscripts and write e.g. $g=g_\epsilon$,  $\tilde{\omega} = \tilde{\omega}_\epsilon$, etc, in this subsection for ease of notation. We do not bother with H\"older norms, and instead bound the stronger weighted derivative norm. The task is then to estimate
\begin{equation} \label{smallness0}
  \| \mathcal{F}(0) \|_{C^{1}_{\beta-1}}  \leq \rho \rho^{-\beta} \bigg( | (\tilde{\omega}|_L) |_{g_L}  + \rho | \nabla^{g_L} (\tilde{\omega}|_L)|_{g_L} +  | ({\rm Im} \tilde{\Omega}|_L) |_{g_L}  + \rho | \nabla^{g_L} ({\rm Im} \tilde{\Omega}|_L)|_{g_L} \bigg).
\end{equation}
We start by studying the untilded reference geometry $(\omega,\Omega)$. The submanifold $L$ is special Lagrangian with respect to $(\omega,\Omega)$ on the regions $\{ \rho \geq 2 \epsilon^{2/5} \}$ and $\{ \rho \leq \epsilon^{2/5} \}$. Therefore
\[
\mathcal{E}_1 :=  \rho \rho^{-\beta} \bigg( \left| (\omega|_L) \right|_{g_L}  + \rho \left| \nabla^{g_L} (\omega|_L) \right|_{g_L} \bigg) 
\]
has the property that
\[
{\rm supp} \, \mathcal{E}_1 \subseteq  \{ \epsilon^{2/5} \leq \rho \leq 2 \epsilon^{2/5} \}.
\]
Since $\omega^{\rm tn}|_L=0$, this term is equal to
\[
\mathcal{E}_1=  \rho \rho^{-\beta} \bigg( \left| (\omega- \omega^{\rm tn}) |_L \, \right|_{g_L}  + \rho \left| \nabla^{g_L} (\omega-\omega^{\rm tn})|_L \, \right|_{g_L} \bigg)
\]
and by \eqref{collar-error} and \eqref{secondfundform} this is bounded by
\begin{align*}
  \mathcal{E}_1 &\leq \rho \rho^{-\beta} \bigg( |\omega-\omega^{\rm tn}|_g + \rho |\nabla^g (\omega-\omega^{\rm tn})|_g + C \rho |A|_g |\omega-\omega^{\rm tn}|_g \bigg) \\
                &\leq C \rho^{-\beta} \epsilon^{2+\frac{1}{5}}.
\end{align*}
We also estimate, using \eqref{tildeCY-nontilde} for differences $\tilde{\omega}-\omega$, the following quantity
\begin{align*}
  \mathcal{E}_2 &:=  \rho \rho^{-\beta} \bigg( \left| (\tilde{\omega}-\omega)|_L \, \right|_{g_L}  + \rho \left| \nabla^{g_L} (\tilde{\omega}-\omega)|_L \, \right|_{g_L} \bigg) \\
  &\leq  C \rho^{-\beta} \epsilon^{2+\frac{1}{15}}.
\end{align*}
Combining estimates for $\mathcal{E}_1$ and $\mathcal{E}_2$ by the triangle inequality, we obtain
\begin{equation} \label{smallness1}
\rho \rho^{-\beta} \bigg( | (\tilde{\omega}|_L) |_{g_L}  + \rho | \nabla^{g_L} (\tilde{\omega}|_L)|_{g_L} \bigg) \leq C \epsilon^{2+\frac{1}{15}}.
\end{equation}
Since $\beta<0$, we used $\rho^{-\beta} \leq 1$. The argument for ${\rm Im} \, \tilde{\Omega}_\epsilon$ is analogous. Later on in \eqref{set-kappa}, we will set the parameter to be $\kappa = 1 + |\beta| + \iota$. Combining \eqref{smallness0} with \eqref{smallness1} and introducing $\epsilon^\kappa \epsilon^{-\kappa}$ gives
\[
\| \mathcal{F}(0) \|_{C^{0,\alpha}_{\beta-1}} \leq C \bigg[ \epsilon^{1+|\beta| + \iota} \bigg] \bigg[\epsilon^{-1 - |\beta| - \iota}  \epsilon^{2+\frac{1}{15}}  \bigg]
\]
and so
\begin{equation}
\| \mathcal{F}(0) \|_{C^{0,\alpha}_{\beta-1}} \leq C \epsilon^{\kappa} \epsilon^{1+\frac{1}{15} -|\beta|-\iota} \leq \epsilon^\kappa \epsilon^\iota
\end{equation}
once $\epsilon$ is small enough provided $\iota < \frac{1}{2} + \frac{1}{30} - \frac{|\beta|}{2}$. Taking $|\beta| \ll 1$ small gives \eqref{small-approx-slag}.

\subsection{Weighted estimates}
Before proving \eqref{step1}, we establish some lemmas on uniform coordinates as the fibers degenerate. Let $L_\epsilon \subseteq (M_\epsilon, g_\epsilon)$ be the approximate special Lagrangian submanifold inside the reference geometry. According to our setup, $L_\epsilon \subseteq M_\epsilon$ is an embedded submanifold which is an $S^1$-bundle away from two points $p_1$ and $p_2$ where the circles collapse. We denote a tubular neighborhood of $L_\epsilon$ by $X_\epsilon$,
\[
L_\epsilon \subseteq X_\epsilon \subseteq M_\epsilon
\]
so that $X_\epsilon \backslash \{ p_1, p_2 \}$ is an $S^1$ fibration over a three-dimensional affine base with singular fibers at $p_1$, $p_2$. Let $z_0 \in X_\epsilon \backslash \{ p_1, p_2 \}$ and choose a local bundle trivialization $z_0 \in U \times S^1$ with $U \subseteq \mathbb{R}^3$. View the circle as $S^1 = \mathbb{R} / \mathbb{Z}$ with coordinate $t \in \mathbb{R}$ on its universal cover. Local functions on $X_\epsilon$ near $z_0$ are understood as
\[
f(x,t) \in C^\infty(U \times \mathbb{R})
\]
which are periodic under $t \mapsto t + n$. 

When bounding various quantities independently of $\epsilon$ in weighted spaces, we will often use the following lemma on uniform coordinates.

\begin{lem} \label{lemma-metric-scale}
  Let $z_0 \in X_\epsilon \cap \{ \rho \geq R_0 \epsilon \}$ and denote $\hat{\rho}=\rho(z_0)$. There exists a local bundle trivialization $U_{z_0} \times S^1 \subseteq X_\epsilon$ around $z_0$ with a covering map
  \[
\varphi: B_1 \times \mathbb{R} \rightarrow U_{z_0} \times S^1, \quad \varphi(\cdot, t): B_1 \xrightarrow{\ \cong\ } U_{z_0}
  \]
  where $B_1 \subseteq \mathbb{R}^3$ is the Euclidean ball of radius one centred at the origin, with the following property. Let
  \[
\bar{g}_{ij} = \hat{\rho}^{-2} [g_\epsilon]_{ij}
  \]
where $[\bar{g}_\epsilon]_{ij}$ denotes the components of the metric tensor over $B_1 \times \mathbb{R}$, i.e. $\bar{g}_{ij}$ are the matrix entries of $\bar{g} = \hat{\rho}^{-2} \varphi^* g_\epsilon$. Then
  \begin{equation} \label{metric-scale}
C^{-1} \delta_{ij} \leq \bar{g}_{ij} \leq C \delta_{ij},
  \end{equation}
  over $B_1 \times \mathbb{R}$, and
  \begin{equation} \label{metric-scale2}
|\bar{g}_{ij} |_{C^k(B_1 \times \mathbb{R})} \leq C_k,
  \end{equation}
  where constants $C>0$ and $C_k>0$ are independent of $z_0$ and $\epsilon>0$.
\end{lem}

\begin{proof}
 Take $z_0 \in X_\epsilon$.  

 \begin{enumerate}[itemsep=0.3em, leftmargin=2.5em]
 \item[\textbf{(1)}] \textbf{The bulk region.}
   Suppose $z_0 \in \{ \rho \geq \rho_0 \}$ and choose an arbitrary local trivialization $z_0 \in B_1 \times \mathbb{R}$ on the universal cover. Here $\rho$ is uniformly equivalent to 1, and furthermore the metric $g_\epsilon$, by \eqref{geps2cyl} and $\mathcal{A} = O(\rho^{-1})$, is uniformly equivalent to
  \[
g_\epsilon^{\rm mod} = g_{\mathbb{R}^3} + \epsilon^2 dt \otimes dt.
\]
Consider the rescaling $T_{\epsilon^{-1}}(x,t) = (x,\epsilon^{-1} t)$, so that $g_1^{\rm mod} = T_{\epsilon^{-1}}^* g_\epsilon^{\rm mod}$ and $T_{\epsilon^{-1}}$ unravels the circle fibers. Then
\[
(B_1 \times \mathbb{R}, \ g_1^{\rm mod}) \overset{T_{\epsilon^{-1}}}{\rightarrow} (B_1 \times \mathbb{R}, \ g_\epsilon^{\rm mod})
\]
is an isometry which realizes \eqref{metric-scale}, with $\rho(z_0)$ being irrelevant in this region. Said otherwise, pullback estimate \eqref{geps2cyl} by $T_{\epsilon^{-1}}$ to obtain \eqref{metric-scale}, \eqref{metric-scale2} in coordinates given by $T_{\epsilon^{-1}}$.

\item[\textbf{(2)}] \textbf{The long region.}
  Suppose $z_0 \in \{ R_0 \epsilon \leq \rho \leq \rho_0 \}$ and write $\hat{\rho}= \rho(z_0)$. We take a local trivialization on the universal cover
\[
 z_0 \in B_{\hat{\rho}/2}(x_0) \times \mathbb{R} \subseteq \left\{ \frac{1}{2} \hat{\rho} \leq  \rho \leq \frac{3}{2} \hat{\rho} \right\} \subseteq X_\epsilon.
\]
We use the scaling map $S_\epsilon$ to convert this region to the long-range region of the Taub-NUT space. After pulling back by $S_\epsilon: x \mapsto \epsilon x$, the geometry is given by
\[
\left( \left\{ \frac{1}{2} \hat{\rho} \epsilon^{-1} \leq  |x| \leq \frac{3}{2} \hat{\rho} \epsilon^{-1} \right\}, \ S_\epsilon^* g_\epsilon \right)
\]
and in this model $\rho=|x|$ and $\{ |x| \geq R_0/2 \}$ since $R_0 \epsilon \leq \hat{\rho} \leq \rho_0$.

We notice that $S_\epsilon^* [\rho^{-2} g_\epsilon]$ is close to the cylinder metric on this region, as can be seen from pulling back \eqref{geps2cyl} via $S_\epsilon$ and using $S_\epsilon^* g_\epsilon^{\rm mod} = \epsilon^2 g_1^{\rm mod}$ and $S_\epsilon^* \rho = \epsilon \rho$. So we have
\begin{equation*}
\begin{aligned}
C^{-1} |x|^{-2} g^{\rm mod}_1 
&\leq S_\epsilon^* [\rho^{-2} g_\epsilon] 
\leq C |x|^{-2} g^{\rm mod}_1, \\
&\text{on } 
\left\{ \tfrac{1}{2} \hat{\rho} \epsilon^{-1} 
\leq |x| \leq 
\tfrac{3}{2} \hat{\rho} \epsilon^{-1} \right\}.
\end{aligned}
\end{equation*}
Next, we pullback this estimate by the rescaling of both the base and the circle given by
\begin{equation*}
\begin{aligned}
\phi: B_1 \times \mathbb{R} 
&\longrightarrow B_{\frac{1}{2} \hat{\rho} \epsilon^{-1}}(\epsilon^{-1} x_0) \times \mathbb{R}, \\
(x,t) 
&\longmapsto \bigl(\epsilon^{-1} x_0 + \hat{\rho} \epsilon^{-1} x,\ \hat{\rho} \epsilon^{-1} t \bigr),
\end{aligned}
\end{equation*}
so that $\phi^* g^{\rm mod}_1 = \hat{\rho}^2 \epsilon^{-2} g^{\rm mod}$. It follows that
\[
C^{-1} g_1^{\rm mod} \leq  \hat{\rho}^{-2} \phi^* S_\epsilon^* g_\epsilon \leq C g_1^{\rm mod} \quad {\rm on} \quad \left\{ \frac{1}{2}  \leq |x| \leq \frac{3}{2} \right\},
\]
which proves \eqref{metric-scale} in the coordinates given by $S_\epsilon \circ \phi$. Higher estimates \eqref{metric-scale2} are similar: it all follows from rescaling coordinates in the estimate \eqref{geps2cyl}.
\end{enumerate}
\end{proof}

We will also need coordinates at the collapsed circles $p_i$. 

\begin{lem} \label{lemma-metric-scale2}
  Let $z_i \in X_\epsilon \cap \{ \rho \leq R_0 \epsilon \}$ so that $\hat{\rho} = \rho(z_i) = R_0 \epsilon$. There exists an open set $z_i \in U \subseteq X_\epsilon$ and a coordinate chart
  \[
\varphi: B_1^{\mathbb{R}^4} \rightarrow U
  \]
where $B_1 \subseteq \mathbb{R}^4$ is the Euclidean ball of radius one centred at the origin with the following property. Let $[g_\epsilon]_{ij}$ denote the components of the metric tensor in this chart, i.e. the entries of $\varphi^* g_\epsilon$, and let $\bar{g}_{ij} = \hat{\rho}^{-2} [g_\epsilon]_{ij}$. Then
  \begin{equation} \label{metric-scale3}
C^{-1} \delta_{ij} \leq \bar{g}_{ij} \leq C \delta_{ij}, \quad |\bar{g}_{ij}|_{C^k} \leq C_k
  \end{equation}
over $B_1^{\mathbb{R}^4}$, where $C>0$ and $C_k>0$ are independent of $z_0$ and $\epsilon>0$.
\end{lem}

\begin{proof}
We work on the compact gluing region. Suppose $z_0 \in \{ \rho \leq R_0 \epsilon \}$. By \eqref{bg-metric-regions} the metric here is $g_\epsilon = \epsilon^2 S_{\epsilon^{-1}}^* g^{\rm tn}$ with $S_\epsilon(x,t) = (\epsilon x, t)$ where $S_\epsilon$ rescales the base $\mathbb{R}^3$. We have an isometry
\[
( \{ |x| \leq R_0 \} , g^{\rm tn}) \overset{S_{\epsilon}}{\rightarrow} (\{ \rho \leq R_0 \epsilon \}, \epsilon^{-2} g_\epsilon).
\]
By compactness
\[
  C^{-1} g^{\rm euc} \leq g^{\rm tn} \leq C g^{\rm euc} \quad |g^{\rm tn}|_{C^k(g^{\rm euc})} \leq C_k
\]
on $\{ |x| \leq R_0 \}$. This gives the desired estimate for $\epsilon^{-2} g_\epsilon$ in the chart provided by $S_\epsilon$, and since $\rho \sim \epsilon$ this proves the estimate \eqref{metric-scale3} for $\rho^{-2} g_\epsilon$.
\end{proof}

We have constructed above uniform coordinates on the ambient $(M_\epsilon, g_\epsilon)$. We will also use uniform coordinates on $(L_\epsilon, g_\epsilon|_{L_\epsilon})$. The same proof goes through verbatim along the submanifold $L_\epsilon \subseteq (M_\epsilon, g_\epsilon)$. Namely:

\begin{itemize}
\item Near each point $z_0 \in L_\epsilon$ with $\rho \geq R_0 \epsilon$, there exists coordinates $(x,t)$ on $L_\epsilon$ near $z_0$ on the local universal cover $(-1,1) \times \mathbb{R}$ such that $\bar{g} = \rho^{-2} g_\epsilon|_{L_\epsilon}$ satisfies
\begin{equation} \label{metricscale3}
C^{-1} \delta_{ij} \leq \bar{g}_{ij} \leq C \delta_{ij}, \quad |\bar{g}_{ij}|_{C^k(g_{\rm euc})} \leq C_k
\end{equation}
on $(-1,1) \times \mathbb{R}$.

\item Near each point $z_0 \in L_\epsilon$ with $ \rho \leq R_0 \epsilon$, there exists coordinates $(x,\psi)$ on $L_\epsilon$ near $z_0$ such that $\bar{g} = \rho^{-2} g_\epsilon|_{L_\epsilon}$ satisfies \eqref{metricscale3} on a coordinate ball $B_2 \subseteq \mathbb{R}^2$.
\end{itemize}

We will now obtain estimates for $d+d^\dagger$ using these weighted norms. Similar results can be found in Foscolo \cite{FoscoloK3}.

\begin{lem} \label{lem:elliptic-est}
 Let $L_\epsilon \subseteq (M_\epsilon,g_\epsilon)$ be the submanifold serving as approximate solution. For any $a \in \Omega^1(L_\epsilon)$, there holds
  \begin{equation} \label{elliptic-est}
\| a \|_{C^{1,\alpha}_\beta} \leq C \bigg( \| a \|_{C^0_\beta} + \| (d+d^\dagger) a \|_{C^{0,\alpha}_{\beta-1}} \bigg)
  \end{equation}
  where $C$ is independent of $\epsilon>0$. 
\end{lem}

\begin{proof} For any $z_0 \in L_\epsilon \cap \{ \rho \geq R_0 \epsilon \}$, we take a coordinate chart around it on the universal cover as in \eqref{metricscale3}, so that near $z_0$ we work locally in coordinates $(x,t) \in (-1,1) \times \mathbb{R}$ translated such that $z_0 = (0,0)$. The 1-form $a$ appears as
  \[
a \in \Omega^1((-1,1) \times \mathbb{R})
  \]
invariant under the $\mathbb{Z}$-action of $t \mapsto t+n$. Apply the interior Schauder estimate for elliptic systems in this coordinate chart:
  \[
\| a \|_{C^{1,\alpha}([-\frac{1}{4}, \frac{1}{4}]^2)} \leq C \bigg( \| a \|_{C^0([-\frac{1}{2}, \frac{1}{2}]^2)} + \| (d+d^\dagger_{\rho^{-2} g_\epsilon} ) a \|_{C^{0,\alpha}([-\frac{1}{2},\frac{1}{2}]^2)} \bigg).
  \]
  Norms are with respect to the Euclidean metric, and uniform ellipticity of $d+d^\dagger$ follows from \eqref{metricscale3}. We convert the norms from $g^{\rm euc}$ to $\rho(x_0)^{-2} g_\epsilon$ using \eqref{metricscale3}, and this is when the weights appear with each derivative, e.g.
  \[
| a |_{C^1(\rho^{-2} g_\epsilon)} = \rho |a|_{g_\epsilon} + \rho^2 |\nabla a|_{g_\epsilon}.
  \]
  We can multiply both sides of this local estimate by $\rho(x_0)^{-\beta}$ to introduce the $\beta$-factor. The construction of the uniform coordinate charts is such that the function $\rho$ is uniformly equivalent to $\rho(x_0)$ on the chosen chart. The weight shift of $\beta \mapsto \beta-1$ comes from
  \[
\rho^{-\beta} |(d+d^\dagger_{\rho^{-2} g_\epsilon})a |_{\rho^{-2} g_\epsilon} = \rho^{-\beta} \rho^2 |(d+d^\dagger_{g_\epsilon}) a|_{g_\epsilon}.
  \]
  On the other hand, if $z_0 \in L_\epsilon \cap \{ \rho \leq R_0 \epsilon \}$, then we apply \eqref{metricscale3} and a similar argument via local Schauder estimate on the local 1-form $a \in \Omega^1(B_1^{\mathbb{R}^2})$.
  
\end{proof}

\subsection{Inverting the linearized operator} \label{sec:invertL}

We would like to upgrade the standard elliptic estimate in Lemma \ref{lem:elliptic-est} by removing the $ \| a \|_{C^0_\beta}$ term on the right-hand side. This is typically acheived by a blow-up argument, however in our setup one of the possible rescalings ends up on the cylinder inside the torus where we face elements of
\begin{equation} \label{K-defn}
\mathcal{H} = {\rm span} \{ \kappa_1, \, \kappa_2 \}, \quad \kappa_1 = \zeta dx, \quad \kappa_2 = \zeta \theta
\end{equation}
and these 1-forms derail the contradiction argument. Here $\zeta(x)$ is a fixed cutoff function supported in the middle of the cylindrical part in the torus bulk region
\[
{\rm supp} \, \zeta \subseteq \{ \rho = 1 \}.
\]
Thus $\{ {\rm supp} \, \zeta \} \cap L_\epsilon$ is diffeomorphic to a cylinder with coordinates $(x,e^{i \psi})$ and where the 1-forms $dx$ and $\theta$ are defined (see Figure \ref{fig:bump}). Intuitively, a special Lagrangian cylinder in a four-torus has translational symmetry in two normal directions. After converting these normal vector fields to 1-forms and applying the ambient complex structure, these correspond to the directions $dx$ and $\theta$ tangent to the cylinder and in the kernel of our linearized operator $d+d^\dagger$ with cylindrical metric. We should gauge fix to remove these symmetries. Practically, we restrict to the following subspace of 1-forms
\[
\mathcal{H}^\perp = \left\{ a \in \Omega^1(L_\epsilon) : \int_{L_\epsilon} a \wedge \star \kappa_i = 0 \quad {\rm for} \ i=1,2 \right\}
\]
and obtain bounds uniform in $\epsilon$. Note that on the support of $\kappa_i$, the reference metric is $g_\epsilon|_{L_\epsilon} = h_\epsilon dx^2 + \epsilon^2 h^{-1} \theta^2$ and
\begin{equation} \label{kappa-star}
\star \kappa_1 = \epsilon \zeta h_\epsilon^{-1} \theta, \quad \star \kappa_2 = - \epsilon^{-1} \zeta h_\epsilon dx.
  \end{equation}
The uniform inverse bound is:

\begin{lem} \label{lem:key-estimate}
 Let $L_\epsilon \subseteq (M_\epsilon,g_\epsilon)$ be the 2-sphere connecting a pair of $p_i$ serving as approximate solution and let $\beta<0$. There exists $\epsilon_0>0$ and $C >1$ such that 
  \begin{equation}
\| a \|_{C^{1,\alpha}_\beta} \leq C \| (d+d^\dagger) a \|_{C^{0,\alpha}_{\beta-1}}
\end{equation}
for all $a \in \mathcal{H}^\perp$ and $0<\epsilon<\epsilon_0$.
\end{lem}

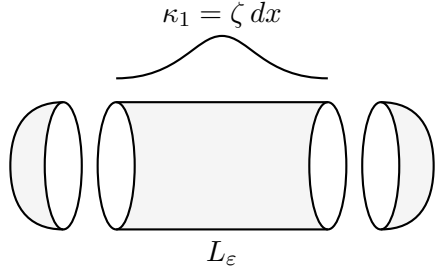
\begin{figure}

\begin{tikzpicture}[scale=0.7]

  \pgfmathsetmacro{\R}{1.2}
  \pgfmathsetmacro{\L}{4.0}
  \pgfmathsetmacro{\cap}{1.0}
  \pgfmathsetmacro{\gap}{1.0}
  \pgfmathsetmacro{\e}{0.35}

  \pgfmathsetmacro{\cylstart}{\gap}
  \pgfmathsetmacro{\cylend}{\gap+\L}
  \pgfmathsetmacro{\rightcap}{\gap+\L+\gap}
  \pgfmathsetmacro{\center}{\gap + \L/2}


\fill[gray!8]
  (\cylstart,\R)
    arc (90:-90:{\e} and {\R})      
  -- (\cylend,-\R)
    arc (-90:-270:{\e} and {\R})    
  -- cycle;
  
  \fill[gray!8]
    (0,-\R)
      .. controls ({-0.7*\cap},-\R) and ({-\cap},{-0.6*\R}) .. ({-\cap},0)
      .. controls ({-\cap},{0.6*\R}) and ({-0.7*\cap},\R) .. (0,\R)
      arc (90:270:{\e} and {\R})
    -- cycle;

  \fill[gray!8]
    (\rightcap,-\R)
      .. controls ({\rightcap+0.7*\cap},-\R) and ({\rightcap+\cap},{-0.6*\R}) .. ({\rightcap+\cap},0)
      .. controls ({\rightcap+\cap},{0.6*\R}) and ({\rightcap+0.7*\cap},\R) .. (\rightcap,\R)
      arc (90:-90:{\e} and {\R})
    -- cycle;


  \draw[thick]
    (0,-\R)
      .. controls ({-0.7*\cap},-\R) and ({-\cap},{-0.6*\R}) .. ({-\cap},0)
      .. controls ({-\cap},{0.6*\R}) and ({-0.7*\cap},\R) .. (0,\R);
  \draw[thick] (0,0) ellipse ({\e} and {\R});

  \draw[thick] (\cylstart,-\R) -- (\cylend,-\R);
  \draw[thick] (\cylstart, \R) -- (\cylend, \R);
  \draw[thick] (\cylstart,0) ellipse ({\e} and {\R});
  \draw[thick] (\cylend,0)   ellipse ({\e} and {\R});

  \draw[thick]
    (\rightcap,-\R)
      .. controls ({\rightcap+0.7*\cap},-\R) and ({\rightcap+\cap},{-0.6*\R}) .. ({\rightcap+\cap},0)
      .. controls ({\rightcap+\cap},{0.6*\R}) and ({\rightcap+0.7*\cap},\R) .. (\rightcap,\R);
  \draw[thick] (\rightcap,0) ellipse ({\e} and {\R});

  \node at (\center,-\R-0.45) {$L_\epsilon$};

  \draw[thick]
    (\cylstart,{\R+0.45})
      .. controls ({\center-0.8},{\R+0.45})
      and ({\center-0.5},{\R+1.25})
      .. ({\center},{\R+1.25})
      .. controls ({\center+0.5},{\R+1.25})
      and ({\center+0.8},{\R+0.45})
      .. (\cylend,{\R+0.45});

  \node at (\center,{\R+1.65}) {$\kappa_1 = \zeta \, dx$};

\end{tikzpicture}
\caption{One-form localized to the cylindrical part of $L_\epsilon$.} \label{fig:bump}
\end{figure}

\begin{proof}
  Suppose by contradiction that there exists a sequence $a_i \in \Omega^1(L_{\epsilon_i})$ and $\epsilon_i \in (0,\epsilon_0)$ such that $\epsilon_i \rightarrow 0$ and
  \[
\| a_i \|_{C^{1,\alpha}_\beta} \geq D_i \| (d+d^\dagger) a_i \|_{C^{0,\alpha}_{\beta-1}}, \quad D_i \rightarrow \infty.
  \]
  We rescale the sequence such that
  \begin{equation} \label{ai-est}
\| a_i \|_{C^{1,\alpha}_\beta}=1, \quad \| (d+d^\dagger) a_i \|_{C^{0,\alpha}_{\beta-1}} \leq D_i^{-1}.
\end{equation}
The uniform elliptic estimate of Lemma \ref{lem:elliptic-est} implies
\[
1 \leq C (\| a_i \|_{C^0_\beta} + D_i^{-1} ),
\]
and so there exists a sequence $x_i \in L_{\epsilon_i}$ with
\begin{equation} \label{key-lowerbdd}
  |a_i(x_i)|_{g_{\epsilon_i}} \geq \frac{1}{2} C^{-1} \rho_{\epsilon_i}^\beta(x_i).
  \end{equation}
 From now on, we simply write $\rho = \rho_{\epsilon_i}$ and suppress its dependence on $\epsilon_i$, as done in \eqref{rho-defn}. We also simply write $\epsilon$ instead of $\epsilon_i$. To derive a contradiction, we must tame the three regions.
  
\vspace{1ex}
\paragraph{\bf Case 1.} The bulk region: suppose $\lim \inf \rho(x_i) > 0$. Then there exists a subsequence along which $x_i \rightarrow x_\infty$ with $\rho(x_\infty)>0$. Let $K \subseteq L_\epsilon$ be a fixed compact set disjoint from $\{ \rho = 0 \}$ with $x_\infty \in K$. Once $\epsilon$ is small enough, then $K \subseteq \{ \rho > 2 \epsilon^{\frac{2}{5}} \}$. Recall that here the reference metric $g_\epsilon|_{L_\epsilon}$ is equivalent to the cylinder model metric $g^{\rm cyl}_\epsilon|_{L_\epsilon}$. To ease notation, we drop the restriction to the submanifold and simply write $g_\epsilon = g_\epsilon|_{L_\epsilon}$  and $g^{\rm cyl}_\epsilon = dx^2 + \epsilon^2 \theta^2$, where $(x,e^{i \psi})$ are the coordinates on the cylinder and $\theta = d \psi - i \mathcal{A}_1(x) dx$. Recall that by \eqref{bulk-A}, a choice of gauge was made to ensure
\[
\sup_{ \{ \rho \geq 2 \epsilon^{2/5} \} }  |\mathcal{A}|_{g_{\mathbb{T}}} \leq C \epsilon^{-2/5}, \quad \sup_{ \{ \rho \geq 2 \epsilon^{2/5} \} }  |\partial \mathcal{A}|_{g_{\mathbb{T}}} \leq C \epsilon^{-4/5}.
\]
We may use \eqref{ddag2cyl} and \eqref{bulk2cyl-L} to convert $g_\epsilon$ to $g_\epsilon^{\rm cyl}$ in estimate \eqref{ai-est}.
\begin{equation} \label{cylcollapse0}
\sup_K  | (d+d^\dagger_{g^{\rm cyl}_\epsilon} ) a_i |_{g^{\rm cyl}_\epsilon} \leq C_K (D_i^{-1} + \epsilon^{\frac{1}{5}}).
\end{equation}
We will also need
\[
\| a_i \|_{C^{1,\alpha}(K,g^{\rm cyl}_\epsilon)} \leq C_K,
\]
which similarly follows from estimate \eqref{ai-est} and using \eqref{bulk2cyl-L} to convert $g_\epsilon$ to $g_\epsilon^{\rm cyl}$. Next, we note the following lemma on collapsing cylinders.

\begin{lem} \label{lem:cylcollapse}
  Let $[a,b] \times S^1$ be a cylinder with coordinates $(x,e^{i \psi})$ equipped with the metric $g^{\rm cyl}_\epsilon = dx \otimes dx + \epsilon^2 \theta \otimes \theta$, where $\theta = d \psi - i \mathcal{A}_1(x) \, dx$. Assume the bounds
  \[
|\mathcal{A}_1| \leq C \epsilon^{-2/5}, \quad |\mathcal{A}_1'| \leq C \epsilon^{-4/5}.
  \]
  Let $a^\epsilon \in \Omega^1([a,b] \times S^1)$ a sequence of 1-forms such that
  \[
\| a^\epsilon \|_{C^{1,\alpha}(g^{\rm cyl}_\epsilon)} \leq C, \quad \| (d+d^\dagger) a^\epsilon \|_{C^{0}(g^{\rm cyl}_\epsilon)} \rightarrow 0,
\]
as $\epsilon \rightarrow 0$. Write
\begin{equation} \label{cyl-1forms}
a^\epsilon = a^\epsilon_x(x,e^{i \psi}) dx + \epsilon a^\epsilon_\psi(x,e^{i \psi}) d \psi.
  \end{equation}
  Then there exists a subsequence such that $(a^\epsilon_x, a^\epsilon_\psi)$ converges uniformly to
  \[
(a^\infty_x, a^\infty_\psi) = (c_1,c_2)
  \]
  where the $c_i$ are constants.
\end{lem}

We postpone the proof of Lemma \ref{lem:cylcollapse} and continue the main argument. Applying Lemma \ref{lem:cylcollapse} to \eqref{cylcollapse0}, we conclude that $(a^{\epsilon_i}_x, a^{\epsilon_i}_\psi)$ subconverges uniformly on $K$ to a pair of constants. The condition that $a_i \in \mathcal{H}^\perp$, using \eqref{kappa-star} and $h_\epsilon \rightarrow 1$ over the support of $\zeta$, implies that these constants must be zero. Therefore $(a^{\epsilon_i}_x, a^{\epsilon_i}_\psi) \rightarrow 0$ on $K$. This contradicts \eqref{key-lowerbdd}, which gives the lower bound
\[
|a_i(x_i)|_{g^{\rm cyl}_\epsilon} \geq C^{-1}
\]
and rewriting this in the notation \eqref{cyl-1forms} is 
\[
|a^{\epsilon_i}_x(x_i)|+ |a^{\epsilon_i}_\psi(x_i)| \geq C^{-1},
\]
which gives the contradiction.
  \vspace{1ex}
  \paragraph{\bf Case 2.} Enter the bubble. If Case 1 does not hold, then $\lim \inf \rho(x_i) = 0$ and there is a subsequence along which $\rho(x_i) \rightarrow 0$. We may assume all $x_i \in \{ \rho < \rho_0 \}$ are near $p_1$. We restrict $a_i$ to the caps to define a sequence of 1-forms over $\{ \rho < \rho_0 \} \cap L_{\rm cig}$. We rescale this region by pullback via $S_\epsilon$, so that our analysis will take place on
  \[
L_{i} = \{ \rho < \rho_0 \epsilon_i^{-1} \} \cap L_{\rm cig}, 
  \]
  with sequence of 1-forms given by
  \[
\tilde{a}_i = \epsilon_i^{-\beta-1} S^*_{\epsilon_i} a_i, \quad  \tilde{a}_i \in \Omega^1(L_{i}).
  \]
The scale of $\epsilon_i^{-\beta-1}$ is included for the estimates \eqref{ai-est-tilde} below. On the rescaled spaces $(L_i,g_i)$ we use the reference metric $g_i = \epsilon_i^{-2} S_{\epsilon_i}^* g_{\epsilon_i}|_{L_{\epsilon_i}}$, which by its definition \eqref{bg-metric-regions} satisfies 
  \[
g_i = g^{{\rm cig},i} \quad \text{on } \{ \rho \leq \epsilon_i^{-3/5} \},
\]
where the cigar metric $g^{{\rm cig},i}$ is induced by a harmonic function of the form
\[
h = 1 + \epsilon_i + \sum_{a=1}^k \frac{1}{2 |x-y_a|} + \epsilon_i^2 \ell(x).
  \]
Pulling back \eqref{ai-est} gives
 \begin{equation} \label{ai-est-tilde}
\| \tilde{a}_i \|_{C^{1,\alpha}_\beta (g_i)} \leq 1, \quad \| (d+d^\dagger) \tilde{a}_i \|_{C^{0,\alpha}_{\beta-1}(g_i)} \leq D_i^{-1}.
\end{equation}
To add a bit more detail, using $S_\epsilon^* \rho = \epsilon \rho$ and $S^* g_\epsilon = \epsilon^2 g_i$, the pullback of
\[
\sup_{ \{ \rho \leq \rho_0 \}} \bigg[ \rho^{-\beta} | a_i|_{g_\epsilon} + \rho \rho^{-\beta} |\nabla a_i|_{g_\epsilon} \bigg] \leq 1 
\]
by $S_\epsilon$ is
\[
\sup_{ \{ \rho \leq \rho_0 \epsilon^{-1} \} } \epsilon^{-\beta} \bigg[ \rho^{-\beta} | S_\epsilon^* a_i|_{\epsilon^2 g_i} + \epsilon \rho \rho^{-\beta} |\nabla S_\epsilon^* a_i|_{\epsilon^2 g_i} \bigg] \leq 1 
\]
and this type of rescaling argument, when also including rescalings of the H\"older norms, gives \eqref{ai-est-tilde}. The pullback of \eqref{key-lowerbdd} is
\begin{equation} \label{key-lowerbdd-tilde}
|\tilde{a}_i( \tilde{x}_i)|_{g_i} \geq \frac{1}{2} C^{-1} \rho^\beta( \tilde{x}_i), \quad \tilde{x}_i = \epsilon_i^{-1} x_i \in L_i.
\end{equation}
We now split the analysis into two subcases.
\vspace{1ex}
\paragraph{\bf Case 2a.} $\lim \sup \rho( \tilde{x}_i) < \infty$; the cigar subcase. Let $K \subseteq L_{\rm cig}$ be a fixed compact set englobing all of the points $\tilde{x}_i \in K$. For all $i$ far enough along the sequence, we have
\[
  K \subseteq \{ \rho \leq \epsilon_i^{-3/5} \} \cap L_i, \quad g_i|_K = g^{{\rm cig},i}.
\]
The cigar metrics $g^{{\rm cig},i}$ over $K$ approach uniformly as $\epsilon_i \rightarrow 0$ the cigar metric $g^{\rm cig}$ inside multi-Taub-NUT with harmonic function
\[
h = 1 + \sum_{a=1}^k \frac{1}{2 |x-y_a|} .
  \]
By \eqref{ai-est-tilde} and the Arzela-Ascoli theorem, we may extract a subsequential limit of $\tilde{a}_i$ on $(K, g^{\rm cig})$. By exhausting with such compact sets $K$, we obtain a limit $\tilde{a}_\infty$ defined on the cigar
\[
L_{\rm cig} \subseteq (N,g^{\rm tn}) \quad (d+d^\dagger) \tilde{a}_\infty = 0.
\]
Since $|\tilde{a}_\infty| \leq \rho^\beta$ with $\beta<0$, we apply Proposition \ref{cig-no-slow-forms} to conclude $\tilde{a}_\infty \equiv 0$. This contradicts \eqref{key-lowerbdd-tilde}.

\vspace{1ex}
\paragraph{\bf Case 2b.} $\lim \sup \rho(\tilde{x}_i) = \infty$; the cone at infinity subcase. In this case let $R_i = \rho(\tilde{x}_i)$ and take a sequence $R_i \rightarrow \infty$. Rescale by $\hat{x}_i = S_{R_i^{-1}} (\tilde{x}_i)$ so that $\rho(\hat{x}_i)=1$ and assume $\hat{x}_i \rightarrow \hat{x}_\infty$. We pull-back our geometric structures by
\[
S_{R_i} : \hat{L}_i \rightarrow L_i, \quad \hat{L}_i = \{ \rho < \rho_0 (R_i \epsilon_i)^{-1} \} \cap L_{\rm cig}
\]
so that
\[
\hat{a}_i = R_i^{-\beta-1} S^*_{R_i} \tilde{a}_i, \quad  \hat{a}_i \in \Omega^1( \hat{L}_i ).
\]
We know $R_i \epsilon_i \rightarrow 0$ since $\rho(x_i) \rightarrow 0$, and so $\hat{L}_i$ expands to all of $L_{\rm cig}$. The geometry on $\hat{L}_i$ will be given by $\hat{g}_i = R_i^{-2} S^*_{R_i} g_i$, so that pulling back \eqref{ai-est-tilde} gives
\begin{equation} \label{case2b1}
\| \hat{a}_i \|_{C^{1,\alpha}_\beta(\hat{g}_i)} =1, \quad \| (d + d^\dagger) \hat{a}_i \|_{C^{0,\alpha}_{\beta-1}(\hat{g}_i)} \leq D_i^{-1}, \quad |\hat{a}_i( \hat{x}_i)|_{\hat{g}_i} \geq C^{-1}.
\end{equation}
The metric $\hat{g}_i$ is modelled on a cylinder. Indeed, using notation
\[
\hat{g}_i = (R_i \epsilon_i)^{-2} S_{R_i \epsilon_i}^* g_{\epsilon_i}|_{L_{\epsilon_i}}, \quad g^{\rm cyl}_{i} = dx^2 + R_i^{-2} d \theta^2,
\]
pulling back \eqref{long-region-cyl} by $S_{R_i \epsilon_i}$ and restricting to the submanifold where second fundamental form terms are controlled by \eqref{long-region-A}, gives the estimates
\begin{equation} \label{long-region-cyl2}
\begin{aligned}
\sup_{\{  R_i^{(-1+\tau)/2} \le \rho \le  (R_i \epsilon_i)^{-1} \rho_0 \}}
    |\hat{g}_i - g_i^{\rm cyl}|_{g_i^{\rm cyl}}
    &\le C \bigl(\epsilon_i^{\frac{3}{5}} + R_i^{-(1+\tau)/2}\bigr), \\[6pt]
\sup_{\{  R_i^{(-1+\tau)/2} \le \rho \le (R_i \epsilon_i)^{-1} \rho_0 \}}
    |\nabla_i^{\rm cyl} \hat{g}_i|_{g_i^{\rm cyl}}
    &\le C \bigl( \epsilon_i^{\frac{1}{5}} (\epsilon_i R_i) +  R_i^{-\tau}\bigr).
\end{aligned}
\end{equation}
Using these, we may convert \eqref{case2b1} to cylindrical geometry. Let $K \subseteq \hat{L}_i \backslash \{ o \}$ be a compact set, where $o=y_1$ is the tip of the cigar. For $0<\tau<1$, then eventually along the sequence the bounds \eqref{long-region-cyl2} apply on $K$, and in a similar way to \eqref{cylcollapse0} we can show
\begin{equation} \label{case2b2}
\| \hat{a}_i \|_{C^{1,\alpha}(K,g^{\rm cyl}_i)} \leq C_K, \quad \| (d + d^\dagger) \hat{a}_i \|_{C^0(K,g^{\rm cyl}_i)} \rightarrow 0.
\end{equation}
By Lemma \ref{lem:cylcollapse}, the sequence uniformly subconverges on $K$ according to
\[
  \hat{a}_i = a^i_x dx + R_i^{-2} a^i_\psi d \psi, \quad (a^i_x,  a^i_\psi ) \rightarrow (c_1,c_2).
\]
By the decay $|\hat{a}_i|_{g^{\rm cyl}_{i}} \leq C \rho^\beta$ with $\beta<0$, we have
\[
|a^i_x|+|a^i_\psi| \leq C \rho^\beta.
\]
Letting $K$ exhaust $L_{\rm cig} \backslash \{ o \}$ implies $c_1=c_2=0$. This contradicts \eqref{case2b1}. 

\end{proof}

\subsection{Cylindrical collapse} \label{sec:cylcollapse} In this section, we give the proof of Lemma \ref{lem:cylcollapse}. Essentially the same proof applies to Lemma \ref{lem:toricollapse}, and these lemmas are so similar that we only give the full details for Lemma \ref{lem:cylcollapse}.

\subsubsection{Setup}
Recall the setup. Let $K=[a,b] \times S^1$ be a cylinder with coordinates $(x,e^{i \psi})$ equipped with the metric
\[
g^{\rm cyl}_\epsilon = \begin{bmatrix} g_{xx} & g_{x \psi} \\ g_{\psi x} & g_{\psi \psi} \end{bmatrix} = \epsilon^2 \begin{bmatrix} \epsilon^{-2} + b^2 & b \\  b & 1
  \end{bmatrix}
\]
with $b(x)$ a real function satisfying
\[
\epsilon b(x) \rightarrow 0, \quad \epsilon b'(x) \rightarrow 0
\]
uniformly. The inverse metric is
\[
[g^{\rm cyl}_\epsilon]{}^{-1} =  \begin{bmatrix} 1 & -b \\ -b & \epsilon^{-2} + b^2\end{bmatrix}.
\]
We are given a sequence of 1-forms $a^\epsilon \in \Omega^1(K)$ satisfying
\begin{equation} \label{ai-cyl}
\| a^\epsilon \|_{C^{1,\alpha}(K, g^{\rm cyl}_{\epsilon})} \leq C, \quad \| (d+d^\dagger) a^\epsilon \|_{C^{0}(K,g^{\rm cyl}_{\epsilon})} \rightarrow 0.
\end{equation}
We write a 1-form $a \in \Omega^1(L_{\rm cyl})$ as
\[
a = a_x(x,e^{i \psi}) dx + \epsilon a_\psi(x,e^{i \psi}) d \psi,
\]
where the factor of $\epsilon$ is introduced so that
\begin{align*}
  |a|^2_{g^{\rm cyl}_\epsilon} &=   |a_x|^2 +  (1+\epsilon^2 b^2)|a_\psi|^2 - 2 \epsilon b a_x a_\psi\\
  &\geq \frac{1}{2} |a_x|^2 + \frac{1}{2} |a_\psi|^2
\end{align*}
once $\epsilon$ is small enough. The goal is to find a subsequence of $(a_x^\epsilon,a_\psi^\epsilon)$ converging to a pair of constants.

\subsubsection{Extracting a subsequence} With these conventions, we compute $d^\dagger a = - \star d \star a$ and $da$.
\begin{align}
  d^\dagger a &= - \epsilon^{-1} \partial_\psi a_\psi - (\partial_x - b \partial_\psi) (a_x+\epsilon b a_\psi),  \label{d-dag-a} \\
  |d a|^2_{g_\epsilon} &= (1 + \epsilon^2 b^2)( \partial_x a_\psi - \epsilon^{-1} \partial_\psi a_x)^2 . \label{d-a}
\end{align}
Next, expanding the norm of $\nabla a$ gives the following estimate on partial derivatives of components:
\begin{equation} \label{osc-est}
 |\partial_\psi a_x|^2 + |\partial_\psi a_\psi |^2 \leq  2 \epsilon^2  |\nabla a|^2_{g^{\rm cyl}_\epsilon},
\end{equation}
for small $\epsilon$. Indeed,
 \begin{align*}
 \epsilon^2  g^{ij} g^{\psi \psi} \nabla_\psi a_i \nabla_\psi a_j &= (1 + \epsilon^2 b^2) |\partial_\psi a_x|^2 + (1+\epsilon^2 b^2)^2 |\partial_\psi a_\psi|^2 \\
   &- 2  (b + \epsilon^2 b^3) \epsilon \partial_\psi a_x \partial_\psi a_\psi \\
   &\geq \frac{3}{4} |\partial_\psi a_x|^2 + \frac{3}{4}|\partial_\psi a_\psi|^2 - \epsilon |b + \epsilon^2 b^3| \left(|\partial_\psi a_x|^2 + |\partial_\psi a_\psi|^2 \right)
\end{align*}
once $\epsilon$ is small enough and the full $\epsilon^2  |\nabla a|^2_{g^{\rm cyl}_\epsilon}$ contains this term and more.

By estimate \eqref{ai-cyl} and the Arzela-Ascoli theorem, after taking a subsequence, $(a_x^\epsilon, a_\psi^\epsilon)$ converges uniformly on $K$ to a limiting pair $(a_x^\infty, a_\psi^\infty)$. By estimates \eqref{osc-est} and \eqref{ai-cyl},
\[
\sup_K \bigg[ |\partial_\psi a^\epsilon_x| + |\partial_\psi a^\epsilon_\psi| \bigg] \leq C \epsilon
\]
and sending $\epsilon \rightarrow 0$ proves that the limiting components $a_x^\infty(x)$ and $a_\psi^\infty(x)$ do not depend on the circle coordinates $e^{i \psi}$. Next, we integrate both sides of \eqref{d-dag-a} over the circle:
\[
\int_0^{2 \pi} d^\dagger a^\epsilon \, d \psi = \int_0^{2 \pi} -\partial_x (a^\epsilon_x + \epsilon b a_\psi) \, d \psi.
\]
Take the limit of this integral and use $d^\dagger a^\epsilon \rightarrow 0$.
\[
0 = \int_0^{2 \pi} \partial_x a^\infty_x \, d \psi.
\]
Since $a_x^\infty$ does not depend on $e^{i \psi}$, this identity is
\[
0 = {d \over dx} a_x^\infty.
\]
Similarly, \eqref{d-a} implies
\[
\int_0^{2 \pi} ( \partial_x a^\epsilon_\psi - \epsilon^{-1} \partial_\psi a^\epsilon_x) d \psi \rightarrow 0
\]
and so
\[
0 = {d \over dx} a_\psi^\infty.
\]
Therefore the limit of component functions is
\[
(a^\infty_x, a^\infty_\psi) = (c_1,c_2)
\]
for constants $c_1$, $c_2$. This proves Lemma \ref{lem:cylcollapse}.


\subsection{Control of the bad directions}

 In Lemma \ref{lem:key-estimate}, we derived
\[
\| a^\perp \|_{C^{1,\alpha}_\beta} \leq C \| (d+d^\dagger) a^\perp \|_{C^{0,\alpha}_{\beta-1}}, \quad a^\perp \in \mathcal{H}^\perp
\]
over the elongated sphere $(L_\epsilon,g_\epsilon)$, where $L_\epsilon$ is the 2-sphere connecting a pair of Taub-NUT bubbles and $g_\epsilon = g_\epsilon|_{L_\epsilon}$ is the induced metric from the ambient reference geometry $(M_\epsilon,g_\epsilon)$. Here $\mathcal{H}$ is the space of 1-forms spanned by $\kappa_1 = \zeta(x) dx$ and $\kappa_2 = \zeta(x) \theta$, where $\theta = d \psi - i \mathcal{A}_1(x) dx$, perdendicular is with respect to the $L^2(g_\epsilon)$ pairing, and $\zeta(x)$ is a bump function such that
 \[
      {\rm supp} \, \zeta \subseteq \{ \rho = 1 \},
    \]
    so that $\zeta$ localizes us to the middle of the cylindrical part of $L_\epsilon$. Our task for this section is to bound the degenerating constants $C(\epsilon)$ in the elliptic estimate
\[
\| a \|_{C^{1,\alpha}_\beta} \leq C(\epsilon) \| (d+d^\dagger) a \|_{C^{0,\alpha}_{\beta-1}}, \quad a \in \Omega^1(L_\epsilon)
\]
as $\epsilon \rightarrow 0$. That is, we now consider all $a \in \Omega^1(L_\epsilon)$ rather than only the subspace $\mathcal{H}^\perp$. To work with these, we decompose an arbitrary element $a \in \Omega^1(L_\epsilon)$ orthogonally as
\begin{equation} \label{a-ortho}
a = a^\perp + \lambda_1 \kappa_1 + \lambda_2 \kappa_2, \quad \lambda_i = \frac{ \langle a , \kappa_i \rangle_{L^2}}{\| \kappa_i \|^2_{L^2}}.
\end{equation}
We can apply the estimate of Lemma \ref{lem:key-estimate}.
\[
\| a \|_{C^{1,\alpha}_\beta} \leq C \| (d+d^\dagger) a^\perp \|_{C^{0,\alpha}_{\beta-1}}  + |\lambda| \|\kappa \|_{C^{1,\alpha}}.
\]
Here we write $\lambda=(\lambda_1,\lambda_2)$, $\kappa=(\kappa_1,\kappa_2)$, and remark that weights play no role in estimates for $\kappa$. Applying the de Rham operator operator to the decomposition \eqref{a-ortho} of $a$ gives
\begin{equation} \label{ddagaperp}
(d+d^\dagger)a = (d+d^\dagger) a^\perp + \lambda_i (d+d^\dagger) \kappa_i,
  \end{equation}
and so
\begin{equation} \label{last-step}
\| a \|_{C^{1,\alpha}_\beta} \leq  C \bigg[ \| (d+d^\dagger) a \|_{C^{0,\alpha}_{\beta-1}}  + |\lambda| \|\kappa \|_{C^{1,\alpha}} \bigg].
  \end{equation}
The remaining step is to bound $|\lambda_i|$. For this, we solve the equation
  \begin{equation} 
d^\dagger \phi = \kappa_2, \quad \phi \in \Omega^2(L_\epsilon) \cap {\rm Im} \, d.
\end{equation}
which is possible since $H^1(L_\epsilon)=0$ and $d^\dagger \kappa_2 = 0$. We note that $\kappa_2= \zeta(x) \theta$ is co-closed with respect to $g_\epsilon$ since over the support of $\kappa_i$, the reference metric is
    \[
      g_\epsilon = h_\epsilon(x) dx \otimes dx + \epsilon^2 h_\epsilon(x)^{-1} \theta \otimes \theta,
    \]
so coclosedness can be verified by explicit calculation \eqref{kappa-star}. Pairing \eqref{ddagaperp} with $\phi$ gives
\begin{equation} \label{kappa-potential}
\langle da, \phi \rangle_{L^2} = \lambda_2 \| \kappa_2 \|^2_{L^2}
\end{equation}
since $a^\perp \perp \kappa_2$. Similarly, we solve
\begin{equation}
d \varphi = \kappa_1, \quad \varphi \in C^\infty(L_\epsilon), \quad \int_{L_\epsilon} \varphi \, d {\rm vol}_{g_\epsilon} = 0,
\end{equation}
and derive
\[
\langle d^\dagger a, \varphi \rangle_{L^2} = \lambda_1 \| \kappa_1 \|^2_{L^2}.
  \]
Therefore
\[
|\lambda_1| \leq \frac{1}{\| \kappa_1 \|^2_{L^2}} \int_{L_\epsilon} |d^\dagger a| |\varphi|, \quad |\lambda_2| \leq \frac{1}{\| \kappa_2 \|^2_{L^2}} \int_{L_\epsilon} |da| |\phi|.
\]
In this subsection we omit the volume form $d {\rm vol}_{g_\epsilon}$ under integrals. We use these estimates to bound $|\lambda_i|$. We start with the bound for $\lambda_1$. We estimate by H\"older's inequality
\begin{align*}
\| \kappa_1 \|^2_{L^2} |\lambda_1|  &\leq \| d^\dagger a \|_{C^0_{\beta-1}} \int_{L_\epsilon} \rho^{-1} \rho^{\beta}  |\varphi| \\
                                                       &\leq \| (d+d^\dagger)a \|_{C^0_{\beta-1}} \bigg( \int_{L_\epsilon} \rho^{-1+b} |\varphi|^2 \bigg)^{\frac{1}{2}} \bigg( \int_{L_\epsilon} \rho^{-1-b+2 \beta} \bigg)^{\frac{1}{2}} \\
  &\leq C \epsilon^{-\frac{b}{2}+\beta} \| (d+d^\dagger)a \|_{C^0_{\beta-1}} \bigg( \int_{L_\epsilon} \rho^{-1+b} |\varphi|^2 \bigg)^{\frac{1}{2}},
\end{align*}
since ${\rm vol}(L_\epsilon) = O(\epsilon)$ and $\rho \geq \epsilon$. Next, we apply the weighted Poincar\'e inequality
  \begin{equation} \label{weighted-poincare}
\int_{L_\epsilon} \rho^{-1+b}  |\varphi|^2 \leq C \int_{L_\epsilon} \rho^{1+b} |d \varphi|^2, \quad  0<b \ll 1,
  \end{equation}
 which holds for all $\varphi \in C^\infty(L_\epsilon)$ with $\int_{L_\epsilon} \varphi \, d {\rm vol} = 0$. Assuming \eqref{weighted-poincare}, which will be proved in the following section, we use $d \varphi = \kappa_1$ to conclude
  \[
\| \kappa_1 \|^2_{L^2}  |\lambda_1| \leq C \epsilon^{-\frac{b}{2}+ \beta} \| (d+d^\dagger)a \|_{C^0_{\beta-1}} \bigg( \int_{L_\epsilon} \rho^{1+b} |\kappa_1|^2 \bigg)^{\frac{1}{2}}.
  \]
  and hence, since $\rho \leq C$,
  \[
 |\lambda_1|  \leq \frac{C}{\| \kappa_1 \|_{L^2}} \epsilon^{-\frac{b}{2}+ \beta} \| (d+d^\dagger)a \|_{C^0_{\beta-1}}.
\]
A completely analogous argument, using in place of \eqref{weighted-poincare} the equivalent inequality in dimension two
\begin{equation}
\int_{L_\epsilon} \rho^{-1+b}  |\phi|^2 \leq C \int_{L_\epsilon} \rho^{1+b} |d^\dagger \phi|^2, \quad  0<b \ll 1,
  \end{equation}
 which holds for all $\phi \in \Omega^2(L_\epsilon)$ with $\int_{L_\epsilon} \phi = 0$, leads to
    \[
 |\lambda_2|  \leq \frac{C}{\| \kappa_2 \|_{L^2}} \epsilon^{-\frac{b}{2}+ \beta} \| (d+d^\dagger)a \|_{C^0_{\beta-1}}.
\]
We now substitute these bounds for $|\lambda|$ into \eqref{last-step}.
\[
\| a \|_{C^{1,\alpha}_\beta} \leq  C \bigg[1  +  \frac{\|\kappa_i \|_{C^{1,\alpha}}}{\| \kappa_i \|_{L^2}} \epsilon^{-\frac{b}{2}+ \beta} \bigg] \| (d+d^\dagger) a \|_{C^{0,\alpha}_{\beta-1}} .
\]
Over the support of $\kappa_i$, the reference metric $g_\epsilon$ is by \eqref{geps2cyl} uniformly equivalent to $g^{\rm cyl}_\epsilon = dx^2 + \epsilon^2 \theta^2$. Since the $\kappa_i$ are fixed compactly supported forms in a region with volume collapsing as $d {\rm vol}_{g_\epsilon} = \epsilon dx \wedge d \psi$, the ratio degenerates as
\[
\frac{\|\kappa_i \|_{C^{1,\alpha}}}{\| \kappa_i \|_{L^2}} \leq C \epsilon^{-\frac{1}{2}}.
\]
We conclude that 
\[
\| a \|_{C^{1,\alpha}_\beta} \leq  C \epsilon^{-\frac{1}{2}-\frac{b}{2}+ \beta} \| (d+d^\dagger) a \|_{C^{0,\alpha}_{\beta-1}}.
\]
for all $a \in \Omega^1(L_\epsilon)$. We take $0< |\beta|, b \ll 1$ so that the exponent on $\epsilon$ is a small perturbation of minus one half.

\subsection{Weighted Poincar\'e inequality}
In this subsection, we prove the existence of a constant $C>1$ such that
\begin{equation} \label{weighted-poincare2}
\int_{L_\epsilon} \rho^{-1} \rho^{b} |f|^2 \, d {\rm vol}_{g_\epsilon} \leq C \int_{L_\epsilon} \rho^{1+b} |\nabla f|^2 \, d {\rm vol}_{g_\epsilon}, \quad 0<b \ll 1,
\end{equation}
for all $\epsilon>0$ small enough and all scalar functions $f \in C^\infty(L_\epsilon)$ such that $\int_{L_\epsilon} f \, d {\rm vol}_{g_\epsilon}= 0$. We prove this by contradiction. Suppose \eqref{weighted-poincare2} fails and there exists a sequence $f_i \in C^\infty(L_{\epsilon_i})$ such that
\begin{equation} \label{weighted-poin-contra}
\int_{L_i} \rho^{-1+b} |f_i|^2 \, d {\rm vol}_{g_\epsilon} = 1, \quad \int_{L_i} \rho^{1+b} |\nabla f_i|^2 \, d {\rm vol}_{g_\epsilon} \rightarrow 0.
\end{equation}
We want to obtain a contradiction from here.

\subsubsection{Extracting a limit on compact sets}
To take limits, first we fix a compact set $K = \{ \rho \geq \delta \} \subseteq L_\epsilon$. Once $\epsilon$ is small enough, we view $K$ as a cylinder $K = [a,b] \times S^1$ with coordinates $(x,e^{i \psi})$. The metric $g_\epsilon$, according to \eqref{geps2cyl}, is uniformly equivalent to a collapsing cylinder:
\begin{equation} \label{geps2cylcyl}
 C_\delta^{-1} ( dx^2 + \epsilon^2 d \psi^2) \leq g_\epsilon \leq  C_\delta ( dx^2 + \epsilon^2 d \psi^2), \quad {\rm on} \ \{ \rho \geq \delta \},
\end{equation}
and $d {\rm vol}_{g_\epsilon}$ is $\epsilon dx d \psi$. To take a limit of the $f_i$ on $K$ while the circle fibers collapse, we will work on the base $[a,b]$ and pushforward $f_i$ to the base.
\[
\bar{f}_i : [a,b] \rightarrow \mathbb{R}, \quad \bar{f}_i(x) = \int_{S^1} \sqrt{\epsilon} f_i(x,e^{i \psi}) d \psi.
\]
The contradiction will come from showing that $\bar{f}_i$ converges to zero on any such compact set $K$, while there exists a compact set $K$ where $\bar{f}_i$ cannot go to zero. Let us start by estimating the $W^{1,2}$ norm of $\bar{f}_i$. First, H\"older's inequality
\[
|\partial_x \bar{f}_i|^2 = \bigg| \int_{S^1} \sqrt{\epsilon} \partial_x f d \psi \bigg|^2 \leq 2 \pi \int_{S^1} \epsilon |\partial_x f|^2 d \psi,
\]
implies
\[
\int_{[a,b]} | \partial_x \bar{f}_i|^2 dx \leq 2 \pi \int_{[a,b] \times S^1} \bigg[ |\partial_x f|^2 + \epsilon^{-2} |\partial_\psi f|^2 \bigg] \epsilon d x d \psi
\]
and so, after adding the metric via \eqref{geps2cylcyl}, we obtain
\[
\int_{[a,b]} | \partial_x \bar{f}_i |^2 dx \leq C_K \int_{L_i} \rho^{1+b} |\nabla f_i|^2_{g_\epsilon} d {\rm vol}_{g_\epsilon} 
\]
since $\{\rho \geq \delta \}$ on $K$. It follows from \eqref{weighted-poin-contra} that
\begin{equation} \label{L2-grad}
\int_{[a,b]} | \partial_x \bar{f}_i |^2 dx \rightarrow 0, \quad i \rightarrow \infty.
\end{equation}
We also have $L^2$ bounds on $\bar{f}_i$, since
\[
\int_{[a,b]} | \bar{f}_i |^2 dx \leq \int_{[a,b] \times S^1}  |f_i|^2 \epsilon \, dx d \psi \leq C_K \int_{L_i} \rho^{-1} \rho^{b} |f_i|^2 \, d {\rm vol}_{g_\epsilon} = C_K
\]
by \eqref{weighted-poin-contra}. It follows from \eqref{L2-grad} and this $L^2$ bound that after a subsequence $\bar{f}_i$ converges to a constant $c \in \mathbb{R}$ in $L^2([a,b])$.

\subsubsection{Showing the limit is zero}
Next, we use the condition $\int f_i = 0$ to show $c=0$. Starting from 
\[
\int_{ \{ \rho \geq \delta \} } f_i \, d {\rm vol}_{g_\epsilon} + \int_{ \{ \rho < \delta \}} f_i \, d {\rm vol}_{g_\epsilon} = 0
\]
we can estimate
\begin{align}
  \bigg| \int_{ \{ \rho \geq \delta \} } \epsilon^{-1/2} f_i d {\rm vol}_{g_\epsilon} \bigg| &\leq \int_{ \{ \rho < \delta \}} \epsilon^{-1/2} |f_i| d {\rm vol}_{g_\epsilon}  \nonumber \\
                                                                          &\leq \epsilon^{-\frac{1}{2}} \bigg( \int_{ \{ \rho < \delta \}} |f_i|^2 \rho^{-1+b} d {\rm vol}_{g_\epsilon} \bigg)^{\frac{1}{2}}  \bigg( \int_{ \{ \rho < \delta \} }  \rho^{1-b}  d {\rm vol}_{g_\epsilon} \bigg)^{\frac{1}{2}} \nonumber \\
  &\leq C \delta^{\frac{1-b}{2}}. \label{poincare-int-zero}
\end{align}
We now keep track $\delta>0$ in the definition of $\bar{f}_i$ and write $\bar{f}_{i ,\delta}: [a_\delta,b_\delta] \rightarrow \mathbb{R}$, $\{ \rho \geq \delta \} = [a_\delta,b_\delta] \times S^1$ where $[a_\delta,b_\delta] \subseteq [-L,L]$ and $[a_{\delta_1}, b_{\delta_1}] \subset [a_{\delta_2},b_{\delta_2}]$ when $\delta_2 < \delta_1$. After extracting a diagonal subsequence, there exists a constant $c$ such that for any $\delta>0$, then $\bar{f}_{i,\delta} \rightarrow c$ in $L^2$ as $i \rightarrow \infty$. We also know from \eqref{poincare-int-zero} that
\[
\bigg| \int_{[a_\delta,b_\delta]} \bar{f}_{i,\delta} dx \bigg| \leq C \delta^{\frac{1-b}{2}}.
\]
Send $i \rightarrow \infty$ followed by $\delta \rightarrow 0$ to see that $c=0$. Therefore for any given $\delta>0$, we have that $\bar{f}_{i,\delta} \rightarrow 0$ in $L^2([a_\delta,b_\delta])$ as $i \rightarrow \infty$.

\subsubsection{Contradiction}
The main step of the proof is to contradict $\bar{f}_i \rightarrow 0$ and \eqref{weighted-poin-contra}. For this purpose, the key estimate is the following: there exists $\iota>0$ such that
\begin{equation} \label{no-mass-lost}
\limsup_{i \rightarrow \infty} \int_{ \{ \rho < \delta \}} |f_i|^2 \rho^{-1+b}  d {\rm vol}_{g_\epsilon} \leq C \delta^\iota.
\end{equation}
We postpone the proof of \eqref{no-mass-lost} and obtain a contradiction from here. Combining \eqref{weighted-poin-contra} and \eqref{no-mass-lost} implies the no-mass-lost type estimate
\begin{equation} \label{no-mass-lost1}
\liminf_{i \rightarrow \infty} \int_{ \{ \rho \geq \delta \}} |f_i|^2 \rho^{-1+b} d {\rm vol}_{g_\epsilon} \geq 1- C \delta^\iota.
\end{equation}
We use this to contradict $\bar{f}_i \rightarrow 0$. Let $K=\{ \rho \geq \delta \} =[a,b] \times S^1$ and
\[
\mu_i = \frac{1}{|K|} \int_{[a,b]} \bar{f}_i \, dx.
\]
Since $\bar{f}_i \overset{L^2}{\rightarrow} 0$, we know that $\mu_i \rightarrow 0$ as $i \rightarrow \infty$. We have
\begin{align}
  \int_{ \{ \rho \geq \delta \}} |f_i|^2 \rho^{-1+b} d {\rm vol}_{g_\epsilon} &\leq C_\delta \int_K |f_i|^2 \, \epsilon dx d \psi \nonumber\\
                                                                                    &= C_\delta \bigg( \int_K |\sqrt{\epsilon} f_i-\mu_i|^2 \,  dx d \psi + \int_K |\mu_i|^2 \,  dx d \psi   \bigg).  
\end{align}
The usual Poincar\'e inequality on $K$ states
\begin{align}
  \int_{K} |\sqrt{\epsilon} f_i - \mu_i|^2 &\leq C_K \int_K \bigg[ |\partial_x f_i|^2 + |\partial_\psi f_i|^2 \bigg] \, \epsilon dx d \psi \nonumber\\
                                            &\leq C_K \int_K \bigg[ |\partial_x f_i|^2 + \epsilon^{-2} |\partial_\psi f_i|^2 \bigg] \, \epsilon dx d \psi \nonumber\\
  &\leq C_K \int_{L_\epsilon} \rho^{1+b} |\nabla f|^2_{g_\epsilon} d {\rm vol}_{g_\epsilon}.
\end{align}
This goes to zero by \eqref{weighted-poin-contra}. Therefore, for each fixed $\delta>0$, letting $i \rightarrow \infty$ gives
\[
\lim_{i \rightarrow \infty} \int_{ \{ \rho \geq \delta \}} |f_i|^2 \rho^{-1+b}  d {\rm vol}_{g_\epsilon} =0 .
\]
This contradicts \eqref{no-mass-lost1} once $\delta>0$ is chosen small enough.

\subsubsection{No-mass-lost estimate}
The Poincar\'e inequality therefore follows from the no-mass-lost estimate \eqref{no-mass-lost}. Each component of the set $\{ \rho < \delta \}$ when $0< \delta< \rho_0$ may be identified as within the local model $L_{\rm cig}$. To prove \eqref{no-mass-lost}, we will use the following estimate:
\begin{equation} \label{Poincare-ALF0}
\int_{L_{{\rm cig},\epsilon}} | \varphi |^2 \rho^{-1+b} d {\rm vol}_{g_\epsilon} \leq C \int_{L_{{\rm cig},\epsilon}} |\nabla \varphi|^2_{g_\epsilon} \rho^{1+b} d {\rm vol}_{g_\epsilon}
\end{equation}
for all $\varphi \in C^\infty(L_{{\rm cig},\epsilon})$ such that
\[
 {\rm supp} \, \varphi \subseteq \{ \rho < \rho_0 \} \cap L_{{\rm cig},\epsilon}.
\]
We assume \eqref{Poincare-ALF0}, which will be proved in the following subsection, and now prove \eqref{no-mass-lost}. Let $\eta: \mathbb{R} \rightarrow [0,1]$ be a fixed cutoff function with $\eta|_{\{x \leq 1 \}} \equiv 1$ and $\eta|_{ \{x \geq 2 \}} \equiv 0$ and $\eta_\delta(x) = \eta(\delta^{-1} \rho)$. Then applying \eqref{Poincare-ALF0} to $\eta_\delta f_i$ gives
\begin{align}
  \int_{L_{\rm cig}} |f_i|^2 \eta_\delta^2 \rho^{-1+b} \, d {\rm vol}_{g_\epsilon} &\leq C \int_{L_{\rm cig}} |\nabla (\eta_\delta f)|^2 \rho^{1+b} \, d {\rm vol}_{g_\epsilon} \nonumber\\
  &\leq C \int_{L_{\rm cig}} |f_i|^2 |\nabla \eta_\delta|^2 \rho^{1+b} \, d {\rm vol}_{g_\epsilon} + C \int_{L_{\rm cig}} \eta_\delta^2 | \nabla f_i|^2 \rho^{1+b} \, d {\rm vol}_{g_\epsilon} .
\end{align}
Since $|\nabla \eta_\delta|= \delta^{-2} |\eta'| |\nabla \rho|$ and $|\nabla \rho|_{g_\epsilon} \leq C$ independently of $\epsilon$, we have
\begin{align} \label{decay-setup}
  & \ \int_{\{ \rho \leq \delta \}} |f_i|^2 \rho^{-1+b} \, d {\rm vol}_{g_\epsilon} \nonumber\\
  &\leq C \int_{\{\delta \leq \rho \leq 2 \delta \}} |f_i|^2 \rho^{-1+b} \, d {\rm vol}_{g_\epsilon} +  C \int_{\{ \rho \leq 2 \delta \}} |\nabla f_i|^2 \rho^{1+b} \, d {\rm vol}_{g_\epsilon}. 
\end{align}
Next, we use the following standard iteration lemma.

\begin{lem} [Lemma 4.19 in \cite{HanLin}]
  Let $\omega$ and $\sigma$ be nondecreasing functions on $(0,1]$, and let $0< \gamma < 1$. Suppose that for all $r \in (0, \frac{1}{2}]$, then
\[
\omega(r) \leq \gamma \omega(2 r) + \sigma(2 r).
\]
Then for all $0< r \leq \frac{1}{2}$, we have
\[
\omega(r) \leq C \big( r^\iota \omega(1) + \sigma(1) \big)
\]
for $C>0$ and $\iota>0$ depending only on $\gamma$.
\end{lem}

We will apply the iteration lemma to
\[
\omega(r) = \int_{\{ \rho \leq r \}} |f_i|^2 \rho^{-1+b} \, d {\rm vol}_{g_\epsilon} , \quad \sigma(r) = \int_{\{ \rho \leq r \}} |\nabla f_i|^2 \rho^{1+b} \, d {\rm vol}_{g_\epsilon},
\]
so that \eqref{decay-setup} is written as
\[
\omega(\delta) \leq C \omega(2 \delta) - C \omega(\delta)+ C \sigma(2 \delta)
\]
or
\[
\omega(r) \leq \frac{C}{C+1} \omega(2r) + C \sigma(2r), \quad 0<r< \frac{1}{2}.
\]
By the iteration lemma and the bound $\omega(1) \leq C$, we conclude
\[
\int_{\{ \rho \leq \delta \}} |f_i|^2 \rho^{-1+b} \, d {\rm vol}_{g_\epsilon} \leq C \delta^\iota + \int_{L_\epsilon} |\nabla f_i|^2 \rho^{1+b} \, d {\rm vol}_{g_\epsilon}.
\]
Sending $i \rightarrow \infty$, the second term drops off \eqref{weighted-poin-contra} and we obtain \eqref{no-mass-lost}.

\subsubsection{Weighted Poincar\'e on cigars}

It remains to prove the local model Poincar\'e inequality \eqref{Poincare-ALF0}. Pullback the integrals in \eqref{Poincare-ALF0} by $S_\epsilon$, let $\epsilon^2 g_i = S^*_\epsilon g_\epsilon |_{L_\epsilon}$ and replace $\varphi$ with $\varphi \circ S_\epsilon$. The shift in weights $\rho^{-1+b} \mapsto \rho^{1+b}$ is there to cancel all $\epsilon$ factors. Then the estimate that we are after is this: there exists $C>1$ such that
\begin{equation} \label{no-mass-lost2}
  \int_{L_{\rm cig}} |\varphi|^2 \rho^{-1+b}  d {\rm vol}_{g_i} \leq C \int_{L_{\rm cig}} |\nabla \varphi|^2_{g_i} \rho^{1+b} d {\rm vol}_{g_i} ,
\end{equation}
for all $\varphi \in C^\infty(L_{\rm cig})$ such that
\[
{\rm supp} \, \varphi \subseteq \{ \rho < \rho_0 \epsilon^{-1} \} \cap L_{\rm cig}.
  \]
In this setup, $g_i$ is asymptotically equivalent to a cylindrical metric in the sense that
\[
|g_i - g_1^{\rm cyl} |_{g_1^{\rm cyl}} \leq \frac{1}{2} \quad \{ \rho \geq R_0 \},
\]
which follows from pulling back \eqref{geps2cyl} by $S_\epsilon$. In this subsection, we omit the volume form $d {\rm vol}_{g_i}$ under integrals as these are uniformly bounded.

Before proving \eqref{no-mass-lost2}, consider first a function $\varphi \in C^\infty(L_{\rm cig})$ which is supported on the long-range region $\{R_0 < \rho < \rho_0 \epsilon^{-1} \}$. This region can be identified with $(R_0, \rho_0\epsilon^{-1}) \times S^1$ with coordinates $(x,e^{i \psi})$ and $\rho=|x|$, and there holds a long-range Poincar\'e inequality:
\begin{equation} \label{longrangepoincare}
\int_{L_{\rm cig}} |\varphi|^2 \rho^{-1+b}  \leq C \int_{L_{\rm cig}} |\nabla \varphi|^2_{g_i} \rho^{1+b}, \quad \varphi \in C^\infty_c(\{R_0 < \rho < \rho_0 \epsilon^{-1} \}).
\end{equation}
Estimate \eqref{longrangepoincare} follows from applying Hardy's inequality
\[
\int_0^\infty g^2 |x|^{-1+b} \, dx \leq \frac{4}{b^2} \int_0^\infty (\partial_x g)^2 |x|^{1+b} \, dx, \quad g(x) \in C^\infty_c(0,\infty),
\]
to $g(x) = \varphi(x,e^{i\psi})$ with fixed circle angle, integration along the circle, and converting the cylindrical metric to $g_i$.

From a more general $\varphi \in C^\infty(L_{\rm cig})$ compactly supported in $\{ \rho < \rho_0 \epsilon^{-1} \}$, we break it into two pieces.
\[
\varphi = \eta \varphi + (1-\eta) \varphi, \quad \eta|_{\rho \leq R_0} \equiv 0, \quad \eta|_{\rho \geq 2 R_0} \equiv 1.
\]
We can then estimate
\[
\int_{L_{\rm cig}} |\varphi|^2 \rho^{-1+b} \leq C \int_{L_{\rm cig}} |\eta \varphi|^2 \rho^{-1+b} + C \int_{L_{\rm cig}} |(1-\eta) \varphi|^2 \rho^{-1+b}. 
\]
The first piece is estimated by the long-range Poincar\'e inequality \eqref{longrangepoincare}, and we obtain the estimate
\begin{equation} \label{cigcigpoincare}
\int_{L_{\rm cig}} |\varphi|^2 \rho^{-1+b}  \leq C \int_{L_{\rm cig}} |\nabla \varphi|^2 \rho^{1+b} + C \int_{ \{ \rho \leq 2 R_0 \} } |\varphi|^2 \rho^{-1+b} ,
\end{equation}
for all functions $\varphi$ supported in $\{ \rho < \rho_0 \epsilon^{-1} \}$.

We now prove \eqref{no-mass-lost2}. Suppose by contradiction that \eqref{no-mass-lost2} fails. Then there exists a subsequence $i \rightarrow \infty$ with
\begin{equation} \label{cigacigarcubacubar}
\int_{L_{\rm cig}} |\varphi_i|^2 \rho^{-1+b}  =1, \quad  \int_{L_{\rm cig}} |\nabla \varphi_i|^2 \rho^{1+b} \rightarrow 0.
\end{equation}
Apply Rellich's lemma to extract a subsequence $\{ \varphi_n \}$ which converges in $L^2$ on the compact set $\{ \rho \leq 2 R_0 \}$. Applying \eqref{cigcigpoincare} gives
\[
\int_{L_{\rm cig}} |\varphi_m-\varphi_n|^2 \rho^{-1+b}   \leq C \int_{L_{\rm cig}} |\nabla (\varphi_m-\varphi_n)|^2 \rho^{1+b} + C \int_{ \{ \rho \leq 2 R_0 \} } |\varphi_m-\varphi_n|^2 \rho^{-1+b}  .
\]
Hence
\[
\int_{L_{\rm cig}} |\varphi_m-\varphi_n|^2 \rho^{-1+b} \rightarrow 0.
\]
Therefore $\varphi_n \rightarrow \varphi_\infty$ in $L^2_{loc}$ with $\nabla \varphi_\infty \equiv 0$. Hence $\varphi_\infty=c_\infty$ is a constant. Combining
\[
\int_{ \{ \rho \leq N \}} \rho^{-1+b}  \geq C^{-1} \int_{x=R_0}^{x=N} |x|^{-1+b} dx \geq C^{-1} N^b , \quad b>0
  \]
with \eqref{cigacigarcubacubar} leads to the contradiction
\[
c_\infty C^{-1} N^b \leq \lim \int_{ \{ \rho \leq N \}} |\varphi_i|^2 \rho^{-1+b} \leq 1
\]
once $N \rightarrow \infty$. This proves \eqref{no-mass-lost2}.

\subsection{Estimating error terms} \label{sec:est-errors}
To finish the proof as outlined at the start of Section \ref{sec:quad-est}, it only remains to prove the quadratic estimate \eqref{step2}, which reads
\begin{equation} \label{contraction2} 
\begin{aligned} 
 & \| \mathcal{Q}(u) - \mathcal{Q}(v) \|_{C^{0,\alpha}_{\beta-1}} \leq \epsilon^\iota \| u - v \|_{C^{1,\alpha}_\beta}, \\
 & u,v \in \mathcal{U} = \{ a \in \Omega^1(S^2) : \| a \|_{C^{1,\alpha}_\beta}  < \epsilon^\kappa \}. 
\end{aligned}
\end{equation}
After substituting the definition of $\mathcal{Q}$ \eqref{Q-defn}, the task then is to estimate
\[
  \mathcal{Q}(u) - \mathcal{Q}(v) = \int_0^1 {d \over ds} \bigg[ \mathcal{F}(u_s) - (d+d^\dagger) (u_s) \bigg] \, ds,
\]
where we let $u_s = su + (1-s)v$. This is bounded by
\[
\|  \mathcal{Q}(u) - \mathcal{Q}(v) \| \leq \int_0^1 \bigg\| \delta \mathcal{F}|_{u_s} (u-v) - (d+d^\dagger) (u-v) \bigg\| \, ds,
\]
and so to prove \eqref{contraction2} we will show:
\begin{equation} \label{eq:linear-estimate}
  \left\| \left( \delta \mathcal{F}|_u - (d + d^\dagger) \right) a \right\|_{C^{0,\alpha}_{\beta-1}} 
  \leq \epsilon^\iota \left\| a \right\|_{C^{1,\alpha}_\beta},
\end{equation}
for all $u,a \in \Omega^1(S^2)$ with $\| u \|_{C^{1,\alpha}_\beta} < \epsilon^\kappa$. For this, we decompose the difference into three groupings:
\begin{align*}
  \delta \mathcal{F}|_u - (d+d^\dagger) &= \bigg[ \delta \mathcal{F}|_u - \delta \mathcal{F}|_0  \bigg] + \bigg[ \delta \mathcal{F}|_0 - \delta \mathcal{F}^{\rm ref}|_0\bigg] + \bigg[ \delta \mathcal{F}^{\rm ref}|_0 - (d+d^\dagger) \bigg] \\
  &= ({\rm I}) + ({\rm II}) + ({\rm III}),
\end{align*}
where we recall $d+d^\dagger$ is induced by the reference geometry $(M_\epsilon, g_\epsilon)$ and we introduce
\[
\mathcal{F}^{\rm ref} (a) = - f_a^* \omega_\epsilon - \star f_a^* {\rm Im} \, \Omega_\epsilon
\]
which is also with respect to the reference structure $(\omega_\epsilon, \Omega_\epsilon)$ rather than the genuine Calabi-Yau structure $(\tilde{\omega}_\epsilon, \tilde{\Omega}_\epsilon)$. We now estimate each error term $({\rm I})$, $({\rm II})$, $({\rm III})$; see also e.g. \cite{ChiuLin, CGPY} for this procedure in different geometric settings.

\subsubsection{Term I} $(\delta \mathcal{F}|_u - \delta \mathcal{F}|_0)$. Recall the setup. We start from a map $f: S^2 \rightarrow M_\epsilon$ parametrizing the approximate solution $L_\epsilon \subseteq M_\epsilon$. A 1-form $u \in \Omega^1(S^2)$ produces a deformed map $f_u : S^2 \rightarrow M_\epsilon$ given by
\[
f_u(q) = \exp^{g_\epsilon}_{f(q)} \left( J \, df_q \, (h^{-1} u) \right)
\]
where $h = f^* g_\epsilon$ is the induced metric on $S^2$. Fix $q \in S^2$, take a local trivialization around $f(q) \in M_\epsilon$ which looks like $B_1 \times S^1$, and work on a local covering space $B_1 \times \mathbb{R}$. Let $\rho = \rho(f(q))$ be the weight function at this point. By Lemma \ref{lemma-metric-scale}, $\bar{g}=\rho^{-2} g_\epsilon$ is uniformly equivalent to the Euclidean metric on the local covering space. To ensure constants are independent of $\epsilon$, we will write $\mathcal{F}$ in terms of the uniform geometry given by $\bar{g} = \rho^{-2} g_\epsilon$, $\bar{h} = \rho^{-2} f^* g_\epsilon$. For this, first we introduce the notation $E_q: T_q^*\, S^2 \rightarrow M_\epsilon$
\[
 E_q( u) = \exp^{\bar{g}}_{f(q)} \left(  J df_q \, (\bar{h}^{-1} u) \right)
\]
so that
\[
f_u(q) = E_q(\rho^{-2}u).
\]
Since $\bar{g}_{ij}$ is uniformly equivalent to the Euclidean metric and its derivatives are bounded, we have uniform bounds on derivatives of the exponential map $\exp^{\bar{g}}{}_{x}(v)$ over $|v|_{\bar{g}} \leq C$. Thus 
\begin{equation} \label{bound:expmap}
  \begin{aligned}
    | D E_q(u) \cdot \xi |_{\bar{g}} &\leq C |\xi|_{\bar{h}}, \\
    |D^2 E_q(u) \cdot (\xi, \eta) |_{\bar{g}} &\leq C | \xi|_{\bar{h}} |\eta|_{\bar{h}}
    \end{aligned}
\end{equation}
for all $|u|_{\bar{h}} \leq C$. 
Now, let us rewrite $\mathcal{F}$ \eqref{F-defn} as 
\[
\mathcal{F}(u) = - \rho^2 E_q(\rho^{-2} u)^* \bar{\omega} - \star_{\bar{h}} E_q(\rho^{-2}u)^* {\rm Im} \, \bar{\Omega}_\epsilon
\]
where $\bar{\omega} = \rho^{-2} \tilde{\omega}_\epsilon$ and $\bar{\Omega} = \rho^{-2} \tilde{\Omega}_\epsilon$. By the fundamental theorem of calculus and the uniform bounds \eqref{bound:expmap},
\[
|(\delta \mathcal{F}|_u - \delta \mathcal{F}|_0){}^{\text{2-form}} a|_{\bar{g}} \leq C \rho^{-2}  |a|_{C^1(\bar{g})}  |u|_{C^1(\bar{g})}.
\]
Here $\rho^{-2}$ appears upon differentiating $E_q(\rho^{-2} u)$ twice. We undo the scaling $\bar{g} = \rho^{-2} g_\epsilon$, noting that the left-hand side is a 2-form while $a$ and $u$ are 1-forms while using $\rho \leq 1$ to obtain
\[
\| ({\rm I}) a \|_{C^0_{\beta-1}(g_\epsilon)} \leq C \rho^{-1}   \| u \|_{C^1(g_\epsilon)} \| a \|_{C^1_\beta(g_\epsilon)} .
\]
Here we also included the estimate for the 0-form part of $\mathcal{F}$ as this follows in a similar way. The H\"older norm can be estimated in a similar way. We now use that $u \in \mathcal{U}$ and conclude
\begin{equation} \label{bound-termI}
\| ({\rm I}) a \|_{C^{0,\alpha}_{\beta-1}} \leq C \epsilon^{\beta} \epsilon^\kappa \epsilon^{-1}  \|a \|_{C^{1,\alpha}_\beta}  
\end{equation}
since $\rho \geq \epsilon$.


\subsubsection{Term II} $( \delta \mathcal{F}|_0 - \delta \mathcal{F}^{\rm ref}|_0)$. We first expand using the definition of the Lie derivative.
\begin{equation} \label{liederiv}
\begin{aligned}
  \delta \mathcal{F}|_0 a &= \left. \frac{d}{ds} \right|_{s=0} 
                 \left( - f_{sa}^* \tilde{\omega}_\epsilon - \star f_{sa}^* \operatorname{Im} \tilde{\Omega}_\epsilon\right) \\
  &= - f^* L_{J a^\sharp} \tilde{\omega}_\epsilon - \star f^* L_{J a^\sharp} {\rm Im} \, \tilde{\Omega}_\epsilon.
\end{aligned}
\end{equation}
Therefore
\[
({\rm II}) a = - f^* L_{J a^\sharp} (\tilde{\omega}_\epsilon -\omega_\epsilon) - \star f^* L_{J a^\sharp} {\rm Im} \, (\tilde{\Omega}_\epsilon - \Omega_\epsilon).
\]
The Lie derivative of a tensor $\chi$ is schematically of the form
  \[
L_V \chi = \nabla V * \chi + V * \nabla \chi,
  \]
  and so we estimate
  \[
| ({\rm II}) a |_{g_\epsilon} \leq C | a |_{g_\epsilon} \, | \nabla (\tilde{\omega}_\epsilon -\omega_\epsilon ) |_{g_\epsilon} + C | \nabla^{g_\epsilon} (Ja^\sharp) |_{g_\epsilon} \, | \tilde{\omega}_\epsilon - \omega_\epsilon  |_{g_\epsilon}.
\]
By the estimate on the second fundamental form \eqref{secondfundform} and converting the extrinsic to the intrinsic via
\[
| \nabla^{g_\epsilon} (Ja^\sharp) |_{g_\epsilon} \leq C |\nabla^{g_\epsilon|_{L_\epsilon}} a|_{g_\epsilon} + C |A|_{g_\epsilon} |a|_{g_\epsilon},
\]
and using the estimate \eqref{tildeCY-nontilde2} for $\tilde{\omega}-\omega$, we obtain
\begin{equation} \label{bound-termII}
\| ({\rm II}) a \|_{C^{0,\alpha}_{\beta-1}} \leq C \epsilon^{1+\frac{1}{6}} \| a \|_{C^{1,\alpha}_\beta}.
\end{equation}
The bound on H\"older norms of $\nabla (\underline{\boldsymbol{\tilde{\omega}}}_\epsilon - \underline{\boldsymbol{\omega}}_\epsilon)$ is because if $\rho |\nabla T| + \rho^2 |\nabla^2 T| \leq C$, then $\| \nabla T \|_{C_{-1}^{0,\alpha}} \leq C$. 
  
\subsubsection{Term III} $(\delta \mathcal{F}^{\rm ref}|_0 - (d+d^\dagger) )$. We recall a classic calculation on the linearization of the special Lagrangian equation (see e.g. \cite{Marshall} for full details).

\begin{lem}
Let $(M,g,\omega,\Omega)$ be a Calabi-Yau structure and let $f: L \rightarrow M$ be a special Lagrangian submanifold with induced metric $g|_L$. Then:
\begin{equation} \label{slag-linearized}
- f^* L_{J a^\sharp} \omega - \star f^* L_{J a^\sharp} {\rm Im} \, \Omega
= (d + d^\dagger_{g|_L}) a,
\end{equation}
for all $a \in \Omega^1(L)$.
\end{lem}

We will use this identity to understand Term (III). The same calculation as \eqref{liederiv} gives
\[
\delta \mathcal{F}^{\rm ref}|_0 a = - f^* L_{J a^\sharp} \omega_\epsilon - \star f^* L_{J a^\sharp} {\rm Im} \, \Omega_\epsilon,
\]
and it follows from \eqref{slag-linearized} that
\[
{\rm supp} \, ({\rm III} a) \subseteq \{ \epsilon^{2/5} \leq \rho \leq 2 \epsilon^{2/5} \}.
\]
This is because the reference structure \eqref{glue-collar} is genuinely Calabi-Yau outside this annulus, and the approximate solution $L$ is genuinely a special Lagrangian submanifold there (see Section \ref{sec:approxsoln}). On this collar region, we bound the error term $({\rm III})a$ by
\[
\| - f^* L_{J a^\sharp} \omega_\epsilon - \star f^* L_{J a^\sharp} {\rm Im} \, \Omega_\epsilon - (d+d^\dagger_{g_\epsilon^{\rm cig}} ) a \| + \|  (d+d^\dagger_{g_\epsilon^{\rm cig}})a -  (d+d^\dagger_{g_\epsilon|L_\epsilon})a \|.
\]
We apply \eqref{slag-linearized} to $d+d^\dagger_{g_\epsilon^{\rm cig}}$ and obtain the bound
\[
  \left|  ({\rm III}) a \right| \leq C \bigg( (|\nabla a| + |A| |a|) |\underline{\boldsymbol{\omega}}_\epsilon - \underline{\boldsymbol{\omega}}^{\rm tn}_\epsilon| + |a| |\nabla (\underline{\boldsymbol{\omega}}_\epsilon - \underline{\boldsymbol{\omega}}^{\rm tn}_\epsilon)| \bigg).
\]
The error estimate \eqref{collar-error} between the model and the global metrics on the collar region implies
\[
  \left|  ({\rm III}) a \right| \leq C  \bigg( \epsilon^{2- \frac{3}{5}}|a|  + \epsilon^{2- \frac{1}{5}}  | \nabla a |  \bigg).
\]
Since $\rho \sim \epsilon^{2/5}$, after including the H\"older norms we can obtain from here the estimate
\begin{equation} \label{bound-termIII}
\left\|  ({\rm III}) a \right\|_{C^{0,\alpha}_{\beta-1}} \leq  C  \epsilon^{2- \frac{1}{5}} \| a \|_{C^{1,\alpha}_\beta}.
\end{equation}
The bound on the H\"older norm of $\underline{\boldsymbol{\omega}}_\epsilon -  \underline{\boldsymbol{\omega}}^{\rm tn}_\epsilon$ follows from \eqref{collar-error} and the remark in Term $({\rm II})$.

\subsubsection{Tuning kappa and iota} Combining \eqref{bound-termI}, \eqref{bound-termII}, \eqref{bound-termIII}, we obtain 
\begin{equation} \label{set-kappa}
  \left\| \left( \delta \mathcal{F}|_u - (d + d^\dagger) \right) a \right\|_{C^{0,\alpha}_{\beta-1}} 
  \leq C (\epsilon^{\kappa -1+\beta}+ \epsilon^{1 + \frac{1}{6}} + \epsilon^{2 - \frac{1}{5}}) \left\| a \right\|_{C^{1,\alpha}_\beta},
\end{equation}
for all $u,a \in \Omega^1(S^2)$ with $\| u \|_{C^{1,\alpha}_\beta} < \epsilon^\kappa$. Let
\[
\kappa = 1 + |\beta| + \iota
\]
and take $0< \iota < 1+ \frac{1}{6}$. Then estimate \eqref{eq:linear-estimate} holds.

\subsection{Gromov-Hausdorff convergence}

\subsubsection{Summary of the fixed point argument}
From a map $f_0: S^2 \rightarrow M_\epsilon$ parametrizing the 2-sphere $L_\epsilon \subseteq M_\epsilon$ connecting a pair of bubbles by straight line, we have deformed $f_0$ to a nearby map
\[
f_a = \exp_{f_0} J a^\sharp
\]
where $a \in \Omega^1(S^2)$ solves
\[
\mathcal{F}(a) = 0, \quad \| a \|_{C^{1,\alpha}_\beta} < \epsilon^{\frac{3}{2}+\frac{1}{50}},
\]
where $\beta$ is negative and $|\beta| \ll 1$. We have replaced the condition $u \in \mathcal{U}$, which is $\| u \|_{\mathcal{B}} < \epsilon^\kappa$, with the explicit exponent $\frac{3}{2}+\frac{1}{50}$ consistent with the definition of $\kappa$ \eqref{small-approx-slag}. Using that $\rho \geq \epsilon$ and $|\beta| \ll 1$, we derive the smallness condition
\begin{equation} \label{small-a}
|a|_{g_\epsilon} < \epsilon^{\frac{3}{2}}, \quad |\nabla a|_{g_\epsilon} < \epsilon^{\frac{1}{2}}.
\end{equation}
Denote the image of the special Lagrangian map $f_a$ by $\tilde{L}_\epsilon$. We will show that $(\tilde{L}_\epsilon, \tilde{g}_\epsilon) \rightarrow (L_0,g_E)$ in the Gromov-Hausdorff sense, where $L_0$ is a straight line between the bubble points $p_1$ and $p_2$ on the affine base $\mathbb{T}^3/\mathbb{Z}_2$.

\subsubsection{Gromov-Hausdorff distance}
We recall the definition of the Gromov-Hausdorff distance.

\begin{defn}
  Let $(X,d_X)$ and $(Y,d_Y)$ be a pair of compact metric spaces. We write
 \[
d_{\rm GH}(X,Y)
=
\inf \left\{
\epsilon>0 \,\middle|\,
\exists \text{ an $\epsilon$-isometry } f:X\to Y
\right\}.
\]
A map $f: X \to Y$ is an $\epsilon$-isometry if for all $x_1,x_2 \in X$ then
\[
\bigl| d_X(x_1,x_2) - d_Y(f(x_1),f(x_2)) \bigr| < \epsilon,
\]
and for every $y \in Y$, there exists $x \in X$ such that
\[
d_Y (y, f(x)) < \epsilon.
\]
\end{defn}

The geometry $(\tilde{L}_\epsilon, \tilde{g}_\epsilon)$ is obtained by implicit function theorem and is not explicit, but these spaces are nearby the model $(L_\epsilon,g_\epsilon)$. Indeed, we note that by \eqref{genuine-est2} and \eqref{small-a} there holds
\[
| f_a^* \tilde{g}_\epsilon - f_0^* g_\epsilon|_{f_0^* g_\epsilon} \leq \delta_\epsilon,
\]
with $\delta_\epsilon \rightarrow 0$ as $\epsilon \rightarrow 0$. Hence if $(L_\epsilon,g_\epsilon)$ $d_{\rm GH}$-converges to $(L_0,g_E)$, then so does $(\tilde{L}_\epsilon, \tilde{g}_\epsilon)$.

\subsubsection{Convergence of the straight line spheres} Finally, we note that our construction readily implies
\[
(L_\epsilon,g_\epsilon) \overset{\epsilon \rightarrow 0}{\rightarrow} (L_0,g_E).
\]
The $\epsilon$-isometry $f: L_\epsilon \rightarrow L_0$ will be projection to the base. Let $x \in [0,1]$ be a coordinate parametrizing $L_0$.

\begin{enumerate}
\item
Bulk cylindrical region: $\{ \rho \geq 2 \epsilon^{\frac{2}{5}} \}$. Here we have coordinates $(x,e^{i \psi})$ and define
\[
f(x,e^{i \psi}) = x.
\]
The estimate \eqref{bulk2cyl} and $|\mathcal{A}| \leq C \epsilon^{-2/5}$ implies that
\[
g_\epsilon|_{L_\epsilon} = dx^2 + \epsilon^2 d \psi^2 + O(\epsilon^{3/5})
\]
holds over this region.

\item Cigar cylindrical region: $\{ \epsilon R_0 \leq \rho \leq \epsilon^{\frac{2}{5}} \}$. We again define
  \[
f(x,e^{i \psi}) = x.
  \]
  On this region $\rho=x$ and by \eqref{geps2cyl} the geometry is uniformly equivalent to
  \[
dx^2 + \epsilon^2 d \psi^2,
  \]
  so the diameter of this region of is of order $\epsilon^{\frac{2}{5}}$.
  
\item Cigar caps: $\{ \rho \leq \epsilon R_0 \}$. Here we simply define $f \equiv 0$. Since
  \[
    g_\epsilon = \epsilon^2 S^*_{\epsilon^{-1}} g^{\rm tn},
  \]
 the diameter of this region is of order $\epsilon$.
  
  \end{enumerate}

Using these properties it is routine to prove that $f: (L_\epsilon,g_\epsilon) \rightarrow (L_0,g_E)$ is an $\delta_\epsilon$-isometry, where $\delta_\epsilon \rightarrow 0$ as $\epsilon \rightarrow 0$.

\end{document}